\documentclass[12pt]{article}
\usepackage{amssymb, amsmath,amsfonts,amsthm}
\usepackage[margin=.8in]{geometry}
\usepackage{setspace}
\usepackage{mathtools}
\usepackage{bm}
\usepackage[colorlinks]{hyperref}
\usepackage{ulem}
\usepackage{mathrsfs}
\newcommand{\ARXIV}[1]{arxiv:#1}
\newcommand{\MR}[1]{MR:#1}

\newcommand{\scrI}{\mathscr{I}}
\newcommand{\SBM}{\text{SBM}}
\newcommand{\ds}{\displaystyle}
\newcommand{\MF}{{\mathcal{M}_F}}

\newcommand{\1}{\mbox {\bf 1}}
\newcommand{\0}{\mbox {\bf 0}}

\newcommand{\cP}{{\mathcal P}}
\newcommand{\uP}{{\underline{P}}}

\newcommand{\uB}{{\underline{B}}}

\newcommand{\vep}{\varepsilon}
\newcommand{\beq}{\begin{eqnarray*}}
\newcommand{\feq}{\end{eqnarray*}}
\newcommand{\beqn}{\begin{eqnarray}}
\newcommand{\feqn}{\end{eqnarray}}

\newcommand{\ba}{\mathbf {a}}
\newcommand{\bM}{\mathbf {M}}

\newcommand{\balpha}{\boldsymbol{\alpha}}
\newcommand{\To}{\Rightarrow}
\newcommand{\Z}{{\mathbb Z}}

\newcommand{\N}{{\mathbb N}}

\newcommand{\R}{{\mathbb R}}

\newcommand{\cN}{{\mathcal N}}
\newcommand{\cT}{{\mathcal T}}

\newcommand{\cV}{{\mathcal V}}

\newcommand{\cF}{{\mathcal F}}
\newcommand{\cU}{{\mathcal U}}
\newcommand{\cG}{{\mathcal G}}

\newcommand{\cA}{{\mathcal A}}
\newcommand{\cE}{{\mathcal E}}

\newcommand{\tB}{\tilde B}

\newcommand{\hxi}{\widehat \xi}

\newcommand{\bxi}{\bar \xi}
\newcommand{\uxi}{\underline{\xi}}
\newcommand{\uX}{\underline{X}}
\newcommand{\uH}{\underline{H}}

\newcommand{\VM}{{\texttt{vm}}}
\newcommand{\TV}{{\texttt{tv}}}
\newcommand{\AV}{{\texttt{av}}}
\newcommand{\GV}{{\texttt{gv}}}
\newcommand{\LV}{{\texttt{lv}}}

\newcommand{\SN}{{S_N}}
\newcommand{\Lip}{\text{\texttt{Lip}}}
\newcommand{\sumA}{\sum_{\emptyset\ne A\subsetneq \cN_N}}

\numberwithin{equation}{section}

\newtheorem{theorem}{Theorem}[section]
\newtheorem{corollary}[theorem]{Corollary}
\newtheorem{lemma}[theorem]{Lemma}
\newtheorem{proposition}[theorem]{Proposition}
\newtheorem{remark}[theorem]{Remark}

\newtheorem{conjecture}[theorem]{Conjecture}

\theoremstyle{definition}
\newtheorem{definition}[theorem]{Definition}
\theoremstyle{definition}
\newtheorem{example}[theorem]{Example}

\title{A complete convergence theorem for the $q$-voter model and other voter model perturbations in two dimensions}

\author{Ted
  Cox\footnote{\href{mailto:jtcox@syr.edu}{jtcox@syr.edu},  Mathematics Department, Syracuse
    University, }\ \ and  Ed
  Perkins\footnote{\href{mailto:perkins@math.ubc.ca}{perkins@math.ubc.ca},
  Mathematics Department, The University of British Columbia}}

\date{}
\begin{document}

\maketitle

% \keywords{Complete convergence theorem; q-voter model;
%   interacting particle system; voter model perturbation;
%   cancellative processes; super-Brownian motion;
%   two-dimensional random walk} 

% \subjclass{60K35; 82C22}

\begin{abstract}The $q$-voter model is a spin-flip system in which the
rate of flipping to type $i$ is given by the $q$th power of the
proportion of nearest neighbours in type $i$ for $i=0,1$.  If $q=1$ it
reduces to the classical voter model.  We show that in the critical
$2$-dimensional case, for $q<1$ and close enough to $1$, for any
initial state as $t\to\infty$ the system converges weakly to a mixture
of all $0$'s, all $1$'s, and a unique invariant law which contains
infinitely many sites of both types.  This follows as a special case
of a general theorem which proves a similar ``complete convergence
theorem'' for cancellative, monotone, finite range voter model
perturbations on $\Z^2$ providing a certain parameter, $\Theta_3$, is
strictly positive. Similar results follow for the affine and geometric
voter models and Lotka-Volterra models, all for parameter values close
to that giving the voter model.  This kind of asymptotic behavior is
quite different from that of the $2$-dimensional voter model itself,
which undergoes clustering, and converges to a mixture of all $0$'s
and all $1$'s.

The above parameter $\Theta_3$ has an explicit expression in terms of
asymptotic coalescing probabilities of $2$-dimensional random walk,
and we give a rather simple sufficient condition for it to be strictly
positive.  An important step in the proof is to establish weak
convergence of the rescaled spin-flip systems to super-Brownian motion
with drift $\Theta_3$.  In fact, a convergence result is proved under
weaker hypotheses which includes all known such results for
$2$-dimensional voter model perturbations and a number of new ones,
including a rescaled limit theorem for the $q$-voter model where
$q\uparrow 1$ with the rescaling.
\end{abstract}

\tableofcontents

\section{Introduction and main results}
\subsection{The \textit{q}-voter model}\label{s:qvintro}
To define the model, let $q\ge 0$ and let $\cN\subset \Z^d$ be a
non-empty finite symmetric (about the origin) set not containing $0$
such that the uniform distribution on $\cN$ is an irreducible kernel
(the group generated by $\cN$ is $\Z^d$), and for some $\sigma^2>0$,
\begin{equation}
\sum_{z\in \cN}z_iz_j/|\cN|=\delta_{ij}\sigma^2\text{ for all }i,j\le d.\label{nbhddef}
\end{equation}
Here $|\cN|$ is the cardinality of $\cN$.  We call such an $\cN$ a
neighbourhood, and the elements of $\cN$, neighbours of $0$.  Note
that the symmetry and irreducibility of $\cN$ imply
\begin{equation}\label{Nprop}
 |\cN|\ge 2d \text{ and is even for any neighbourhood $\cN$.}
\end{equation}
For $x\in\Z^d$, $x+\cN$ is the set of neighbours of $x$. \ In the
$q$-voter model the state at time $t$ is $\xi_t:\Z^d\to \{0,1\}$,
where $\xi_t(x)$ is the opinion of a voter at $x$ at time $t$. The
rate at which the voter at $x$ changes opinion is the $q$-th power of
the fraction of its neighbours with the opposite opinion.  More
formally, if $\xi\in\{0,1\}^{\Z^d}$ and
\[
f_j(x,\xi) = \frac{1}{|\cN|}\sum_{y\in \cN}1\{\xi(x+y)=j\}, \quad j=0,1,
\]
and $c^{(q)}$ is defined by 
\[
c^{(q)}(x,\xi) = \hxi(x) f^q_1(x,\xi) + \xi(x)f^q_0(x),
\]
then the $q$-voter model $\xi_t$ is the spin-flip process with rate
function $c^{(q)}$ (see Theorem B.3 in \cite{Lig99}).  Throughout we
will use the notation
\[\hxi=1-\xi,\]
and if $q=0$, $0^q:=0$ in the above.  The well-studied voter model is
obtained by taking $q=1$, in which case we will write $c^{\VM}(x,\xi)$
for $c^{(1)}(x,\xi)$.

The $q$-voter model was introduced by Nettle in \cite{N99} and also
used by Abrams and Strogatz in \cite{AS} as a model of language
death. The model, along with many variations of it, has been studied
in the physics literature (e.g., see \cite{ASW}, \cite{ASLPSW},
\cite{JSW}, \cite{MLCPS}, \cite{NSSW}, \cite{VA}) Rigorous results for
the model defined on large torii in $\Z^d$, $d\ge 3$ for $q$ close to
$1$ have been obtained by Agarwal, Simper and Durrett in
\cite{ASD}. Our goal is to study the model on $\Z^2$ in the
mathematically critical, and biologically important, two-dimensional
case, with $q$ close to $1$ and $q<1$. The methods we use to study
this model will lead to some general results for a family of
two-dimensional spin-flip processes.

We let $|\xi|=\sum_x\xi(x)$, and say that a probability measure $\nu$
on $\{0,1\}^{\Z^d}$ has the {\it coexistence property} if
\[
\nu(|\hxi|=|\xi|=\infty) =1.
\]
A translation invariant probability $\nu$ on
$\{0,1\}^{\Z^d}$ has density $p\in[0,1]$ iff
$\nu(\xi(0)=1)=p$.  In the case $d\ge 3$, it is well known
that the voter model has a one-parameter family of
translation invariant stationary distributions indexed by
density, $\{\mu_\theta,0\le \theta\le 1\}$, such that for a
wide class of initial laws with a given density, $\theta$,
the voter model converges in law as $t\to\infty$ to the
corresponding $\mu_\theta$ (see Chapter~V of
\cite{Lig85}). In particular, density is preserved over
time. For $q<1$, the flip rates to $1$'s say, are increased
relative to the voter model by a factor of $f_1^{q-1}$, and
so one is reinforcing the flip rates to locally rare
types. That this effect strongly influences the ergodic
behavior of the model is shown in Theorem~1.3 of \cite{ASD},
which considers a sequence of $q_n$-voter models
$\xi^{(n)}_t$ on torii in $\Z^3$ of side length $n$,
$q_n\uparrow 1$. That result shows that for an initial
sequence of laws with any fixed density in (0,1), with an
appropriate time rescaling, the density of $\xi^{(n)}_t$
approaches 1/2 and stays close to 1/2 for at least
polynomially (in $n$) long times.  We will refine this for
$|\cN|\le 8$ (and in particular for the nearest neighbour
setting in $d=3,4$) for $q<1$, sufficiently close to one, by
showing the existence of a translation invariant stationary
distribution $\nu_{1/2}$ with the coexistence property and
density $1/2$ to which the $q$-voter model converges weakly
starting from any initial law with the coexistence
property. See Theorem~\ref{t:CCT} where a slightly stronger
``complete convergence" result is proved).  In particular
$\nu_{1/2}$ is the unique stationary law with the
coexistence property (see Corollary~\ref{c:inv}).

Perhaps more interesting, such a convergence theorem also holds in two
dimensions, again for $|\cN|\le 8$, and so, in particular, for the
nearest neighbour case (see again Theorem~\ref{t:CCT}).  Recall that
for $d=2$, the voter model exhibits clustering, that is, it converges
weakly to a mixture of all zeros and all ones as $t$ tends to
infinity. Dynamically, a typical site becomes part of a growing
cluster of the same type as $t$ becomes large. The actual type will
change back and forth as time evolves and the clusters grow.  A
quantitative description of this dynamical clustering may be found in
\cite{CG}.  This clustering behaviour in two dimensions is typical of
branching population models such as super-Brownian motion \cite{D77}
or discrete time branching systems \cite{K77}.  In these cases there
is local extinction for large time and the mass becomes concentrated
on larger and larger clumps separated by greater and greater
distances. As the two-dimensional case is of particular importance in
population modelling this failure to converge to any local equilibrium
is referred to as the ``pain in the torus".  See \cite{F75} where a
closely related model exhibiting the above clumping behaviour is
dismissed as being ``biologically irrelevant".  The fact that, for any
$q<1$ and close to $1$, the $q$-voter model converges to an
essentially unique equilibrium is, we believe, of some general
interest.  Other interesting examples of such convergence results
which include lower dimensions ($d\le 2$), are due to Handjani
\cite{H99} for the threshold voter model, corresponding to $q=0$
(excluding only the one-dimensional nearest neighbour case where
coexistence fails), the present authors \cite{CP14} for the symmetric
two-dimensional Lotka-Volterra model, and Sturm and Swart \cite{SS}
for the one-dimensional ``rebellious voter model" for sufficiently
small competition parameter. Results and methods from the first two
works will play a role here.

Our convergence results, and the coexistence results in \cite{ASD},
both require $q$ close enough to $1$. This is counter-intuitive
because taking $q$ smaller should only increase the advantage of
locally rare types, described above.  This restriction on $q$ is due
to the perturbative nature of both arguments which use the theory of
voter model perturbations. This refers to the fact that the models
approach the voter model as a parameter approaches a particular
value--in this case as $q\uparrow 1$.  \cite{ASD} uses general results
for voter model perturbations from \cite{CDP13} which require
$d\ge 3$, as these results rely on rapid local convergence to the
appropriate invariant law of the voter model.  The two-dimensional
case is more involved but, as was noted above, a particular class of
voter model perturbations, Lotka-Volterra models, were analyzed (again
using perturbative methods) for $d=2$ in \cite{CP08} and \cite{CMP}.
The methodologies in this latter paper will be extended here to show
that rescaled $q$-voter models in which $q\uparrow 1$ as well,
converge to a two-dimensional super-Brownian motion with a positive
drift (see Section~\ref{sec:scallim}).  The positivity of the drift
will be critical to show that the process exhibits long-term
coexistence. The description of the drift in terms of the asymptotics
of long-time non-coalescing probabilities for two-dimensional random
walks (see \eqref{Theta} in Section~\ref{sec:coalrw}) is therefore
important.

When $q>1$ the above intuitive argument now goes in the
other direction; the flip rates to $1$'s are multiplied by a
factor of $f_1^{q-1}$ and so, relatively speaking, one is
reinforcing flip rates to locally dominant types.  As a
result we now expect a type of founder control, where one
type or another will take over, with probabilities depending
on the initial configuration.  The tools for proving such a
result for $q>1$ and close enough to $1$ in two dimensions,
or even for $d\ge 3$, do not seem to be currently available.

In addition to being a voter model perturbation (see
Section~\ref{sec:vmp}) the other key properties of the
$q$-voter model used in the proof are \textit{monotonicity}
(see the definition in Section~\ref{ssec:CCT}) and the
\textit{cancellative} property (see \eqref{canc} in
Section~\ref{sec:canc}), which implies it has an
\textit{annihilating} dual. The fact that the $q$-voter model
has no \textit{coalescing} dual (see Section 4 of Chapter III
in \cite{Lig85}) was established in \cite{ASD}, where it was
shown that that the $q$-voter model ($q\ne 1$) is not
additive.

Our ``complete convergence" theorem for $q$-voter models is obtained
as a corollary of a general result for two-dimensional monotone,
cancellative finite range voter model perturbations when (the natural
extension of) the above drift parameter is positive (see
Theorem~\ref{t:vgen2dCCT} below). This result will also imply such a
``complete convergence" theorem with coexistence for (two-dimensional)
affine voter models, geometric voter models and the aforementioned
Lotka-Volterra models, for appropriate choices of parameter (see the
examples in Section~\ref{sec:vmp} and then Theorems~\ref{t:CCT-LV},
\ref{t:CCT-AV} and \ref{t:CCT-GV}). We note that no restrictions on
$|\cN|$ are required for these three models or for
Theorem~\ref{t:vgen2dCCT} to hold.  The Lotka-Volterra result was
first proved in \cite{CP14}, while the other applications are new.

 \subsection{A complete convergence theorem for the \textit{q}-voter
model}\label{ssec:CCT} Assume first $c(x,\xi)$ is a rate function
satisfying condition (B4) of \cite{Lig99} (see
\eqref{pertboundunscaleda} below), and so by Theorem~B3 of that
reference is the rate function of a unique spin flip system $\xi_t$
starting in state $\xi_0$ under $P_{\xi_0}$.  Define the hitting times
$\tau_{\0} = \inf\{t\ge 0: \xi_t=\0\}$ and $\tau_{\1} = \inf\{t\ge 0:
\xi_t=\1\}$.  We identify a random vector with its probability law, as
usual, and introduce the probabilities
\[ \beta_{0}(\xi_0) = P_{\xi_0}(\tau_{\0}<\infty), \quad
\beta_{1}(\xi_0) = P_{\xi_0}(\tau_{\1}<\infty),\quad
\beta_\infty(\xi_0) = P_{\xi_0}(\tau_{\0}=\tau_{\1}=\infty).
\] By standard arguments (see (1.8) of \cite{CP14}),
\begin{equation}\label{betavalues}\beta_{0}(\xi_0)=0\text{ if
}|\xi_0|=\infty, \ \beta_{1}(\xi_0)=0\text{ if }|\hxi_0|=\infty,\text{
and hence }\beta_\infty(\xi_0)=1\text{ if }|\xi_0|=|\hxi_0|=\infty.
\end{equation} Recall that a spin-flip process with rate function
$c(x,\xi)$ is monotone iff for every $\uxi\le\xi$ (this means
$\uxi(x)\le \xi(x)$ for all $x$),
  \begin{equation*}
  \begin{aligned} c(x,\xi)&\ge c(x,\uxi)\text{ if }\xi(x)=\uxi(x)=0,
\\ c(x,\xi)&\le c(x,\uxi)\text{ if } \xi(x)=\uxi(x)=1,
\end{aligned}
\end{equation*}
We assume throughout that $0\le q \le 1$ and note that clearly, for
all such $q$,
\begin{equation}\label{qvattractive} \text{the $q$-voter model is  monotone.}
\end{equation}

We let $\To$ denote weak convergence of probability laws. A
probability $\nu$ on $\{0,1\}^{\Z^d}$ is symmetric iff
$\nu(\hxi\in\cdot)=\nu(\xi\in\cdot)$.  Both $\0$ and $\1$ (the
configurations of all $0$'s and all $1$'s, respectively) are obviously
traps for the $q$-voter model.

\begin{theorem}\label{t:CCT} Assume $|\cN|\le 8$ and $d=2,3$ or $4$.
There exists $0<q_c<1$ such that for $q_c<q<1$ there is a
translation invariant symmetric stationary distribution $\nu_{1/2}$
with density $1/2$ satisfying the coexistence property, and such that
for all initial $\xi_0\in\{0,1\}^{\Z^d}$,
\begin{equation}\label{cctheorem} \xi_t \To
\beta_{0}(\xi_0)\delta_{\0} +\beta_{\infty}(\xi_0) \nu_{1/2}
+\beta_{1}(\xi_0)\delta_{\1} \quad \text{ as }t\to\infty,
\end{equation} where $\beta_\infty(\xi_0)>0$ unless $\xi_0$ is one of
the fixed points, ${\bf 0}$ or ${\bf 1}$.
\end{theorem}

We will state Theorem~\ref{t:CCT} more compactly by saying that for
$q_c<q<1$ {\it the complete convergence theorem with coexistence (CCT)
holds for the $q$-voter model.}  As an immediate corollary we have a
complete description of all invariant laws.

\begin{corollary}\label{c:inv} For $d$, $\cN$, and $q_c$ as in
Theorem~\ref{t:CCT}, and $q_c<q<1$, $\nu_{1/2}$ is the only invariant
distribution with the coexistence property, and
$\{\nu_{1/2},\delta_{\0}, \delta_{\1}\}$ are the only extremal
invariant distributions.
\end{corollary}

\begin{remark} Note that the hypotheses of the above results are
satisfied when $\cN$ is the set of nearest neighbours in $\Z^d$ and
$2\le d\le 4$.  They also hold for $d=2$ when $\cN$ is the unit sphere
in $\Z^2$ in the $L^\infty$ norm ($|\cN|=8$). Clearly the restriction
$|\cN|\le 8$ forces $d\le 4$ by \eqref{Nprop}.
\end{remark} The $q$-voter model is a {\it nonlinear voter model} as
defined in \cite{CD91}. These are spin-flip processes where the rate
of flipping depends only on the number of sites of the opposite type
in $\cN$.  The cancellative property is defined in Section III.4 of
\cite{Lig85} and discussed in Section~\ref{sec:canc} below.  It is
verified for the $q$-voter model in Lemma~\ref{l:qcanc} when $|\cN|\le
8$ and $q\in(q_c,1)$ using a criteria from \cite{CD91} developed for
nonlinear voter models (Proposition~\ref{p:nlvm}). This criteria
involves the inverse of an $|\cN|\times|\cN|$ matrix and so becomes
more complicated as $|\cN|$ increases. This is the reason we restrict
the neighbourhood size in Theorem~\ref{t:CCT}.
\begin{conjecture}\label{conj:cct} In Theorem~\ref{t:CCT} complete
convergence with coexistence continues to hold for any neighbourhood
$\cN$ in any dimension $d\ge 2$ for $q<1$ and sufficiently close to
$1$.
\end{conjecture}
\noindent As the above discussion suggests, and a more careful
analysis of the proofs shows, this would follow from
Conjecture~\ref{conj:canc} on the cancellative property holding for
general $q$-voter models. We hasten to add that for other models like
the geometric voter model, the threshold voter, and the affine voter
model, one can apply Proposition~\ref{p:nlvm} for arbitrary $\cN$, and
for the Lotka-Volterra model, one can check the cancellative property
directly through an educated guess of the annihilating dual (see
\cite{NP99} and Section~6 of \cite{CP14}). These models are defined in
Section~\ref{sec:vmp}.

As already noted in Section~\ref{s:qvintro}, taking $q<1$ smaller than
$q_c$ should only make large clusters less likely and so the complete
convergence should follow for all $0<q<1$ for $d\ge 2$. Moreover, as
the result fails for $q=1$ (the voter model), we expect $q$ close to
$1$, handled to some extent in Theorem~\ref{t:CCT}, to be the most
delicate case. Recall also from Section~\ref{s:qvintro} that for the
extreme case $q=0$, a CCT is proved in \cite{H99} for $d\ge 2$.
\begin{conjecture}\label{conj:cctallq} In Theorem~\ref{t:CCT} complete
convergence with coexistence holds for any neighbourhood $\cN$, any
$d\ge 2$ and any $0<q<1$.
  \end{conjecture}
\noindent We state the CCT for $0<q<1$ and $d\ge 2$ as a separate
conjecture because we believe the issues here are quite different from
those underlying Conjecture~\ref{conj:canc}. Here a result in any $d$
for any neighbourhood would be of great interest. Similar ``obvious"
results should also hold for the CCT for Lotka-Volterra models for
(symmetric) competition parameter $\alpha\in(0,1)$ but are again only
proved for the most delicate case when $\alpha$ near $1$ (see
Theorem~1.1 of \cite{CP14}) due to the perturbative nature of the
proofs.

To state our general two-dimensional CCT we need to introduce the
drift parameter mentioned in Section~\ref{s:qvintro}, and for this, we
first need some long time asymptotics of non-coalescing probabilities
for two-dimensional random walks.

\subsection{Two-dimensional coalescing random walk}\label{sec:coalrw}
Let $p:\Z^2\to[0,1]$ be a symmetric, irreducible, random walk kernel
with covariance matrix $\sigma^2 I$ for some $\sigma>0$, such that
$p(0)=0$. For the particular case of $q$-voter models, $p(\cdot)$ will
be the uniform law on a neighbourhood $\cN$. Under a probability $\hat
P$, let $\ds \{B^x_t,x\in\Z^2\}$ be a system of rate one continuous
time coalescing random walks with jump kernel $p$, and for $A\subset
\Z^2$ define $B^A_t=\{B^x_t,x\in A\}$, let $|B^A_t|$ denote its
cardinality, and let the time it takes all walks starting in $A$ to
coalesce to a single walk be $\tau(A) = \inf\{t\ge 0: |B^A_t|=1\}$.
For $n\ge 2$ and nonempty, finite disjoint $A_1,\dots,A_n
\subset\Z^2$, define the stopping times
\begin{equation}\label{sigtau}
\begin{aligned}
\tau(A_1,\dots,A_n) &= \max_{1\le i\le n}\tau(A_i),\\
\sigma(A_1,\dots,A_n) &= \inf\{t\ge 0:
B^{A_i}_t\cap B^{A_j}_t\ne \emptyset \text{ for some
}i\ne j\}.
\end{aligned}
\end{equation}
At the risk of some confusion, for $x\in\Z^2$, we will often identify
$\{x\}$ with $x$. We write
\[a(t)\sim b(t)\text{ as $t\to\infty\ $ to mean }\
\lim_{t\to\infty}\frac{a(t)}{b(t)}=1.\]

Proposition~1.3 in \cite{CMP} states that for $n\ge 2$ and distinct
$x_1,\dots,x_n\in\Z^d$ there is a finite $K_n(x_1,\dots,x_n)>0$ such
that
  \begin{equation}\label{CRWold} q_{x_1,\dots,x_n}(t):=\hat
P\big(\sigma(x_1,\dots,x_n)>t\big) \sim
\frac{K_n(x_1,\dots,x_n)}{(\log t)^{\binom n2}} \text{ as }t\to\infty.
\end{equation} The $n=2$ case is well known, as $\sigma(x_1,x_2 )$ has
the same law as the hitting time of 0 of a walk starting at $x_1-x_2$
run at rate 2.  The following extension of \eqref{CRWold}, proved in
Section~\ref{s:crw} below, will be used to describe the key drift term
arising in our general CCT.

\begin{proposition}\label{p:CRWnew} Let $n\ge 2$ and $A_1,\dots,A_n$
be nonempty finite disjoint subsets of $\Z^2$. Then there exists a
finite $K_n(A_1,\dots,A_n)>0$ such that
\begin{equation}\label{CRWnew} \hat P(\sigma(A_1,\dots,A_n)>t,
\tau(A_1,\dots,A_n)<t) \sim \frac{K_n(A_1,\dots,A_n)}{(\log t)^{\binom
n2}} \text{ as }t\to\infty.
\end{equation} In fact, if $a_i\in A_i$, $1\le i\le n$,
 \begin{multline} K_n(A_1,\dots,A_n)\\ = \sum_{\text{distinct
}x_1,\dots,x_n\in \ \Z^2} K_n(x_1,\dots,x_n)\hat
P(\sigma(A_1,\dots,A_n)> \tau(A_1,\dots,A_n),\\
B^{a_i}_{\tau(A_1,\dots,A_n)}=x_i,\ 1\le i\le n).
\label{KnDef}
\end{multline}
\end{proposition}
\begin{remark}\label{r:coalbnds} Let $S$ be a finite subset of $\Z^2$.
By summing over partitions of $S$ of cardinality $n$, one sees that
for $n\ge 2$, if $u(t)\ge Ct^r$ for some $C,r>0$, then
\[ \sup_{t\ge 1} (\log t)^{\binom n2} \hat P( |B^{S}_{u(t)}|=n )
<\infty.
\]
\end{remark}
Let $\cN$ be a neighbourhood in $\Z^2$ (in practice it will contain
the support of $p$ but this is not needed for our definitions), and
set
\[\bar\cN=\cN\cup\{0\}.\]
For a set $\Gamma$, $|\Gamma|\ge k$, let $\cP_k(\Gamma)$ be the set of
partitions $\{\pi_1,\dots,\pi_k\}$ of $\Gamma$ such that each
$|\pi_i|\ge 1$.  We will write $\cP(\Gamma)$ for $\cP_2(\Gamma)$.  For
$A$ a non-empty subset of $\cN$, define
\begin{equation}\label{Theta+-}
\begin{aligned}
\Theta^{+}(A) &= \sum_{\{A_1,A_2\}\in
\cP(\bar\cN\setminus A)}
K_3(A,A_1,A_2),
\Theta^{-}(A) &= \sum_{\{A_1,A_2\}\in \cP(A)}
K_3(\bar\cN\setminus A,A_1,A_2).
\end{aligned}
\end{equation}
Note that we have suppressed the dependence of $\Theta^\pm$ on $p$ and
$\cN$.  For the $q$-voter model it will be understood that
$p= 1_{\cN}/|\cN|$ for a given neighbourhood $\cN$.

\subsection{ A general complete convergence theorem in two dimensions}\label{sec:proofCCTintro}

The main conditions we impose on our spin-flip system is that they are
cancellative and constitute a finite range voter model
perturbation. To define the latter, for $d\ge 2$, let $p:\Z^d\to[0,1]$
be a symmetric, irreducible, random walk kernel with finite support
and $p(0)=0$ (as in the last section but now with $d\ge 2$). Write
$p(A)$ for $\sum_{y\in A}p(y)$.  Assume that
\begin{equation}\label{covpintro} \text{ $p$ has covariance matrix
$\sigma^2 I$ for some $\sigma>0$.}
\end{equation} Let $f_i(x,\xi)=\sum_yp(y-x)1\{\xi(y)=i\}$ (agreeing
with our earlier notation if $p$ is uniform on $\cN$) and introduce
the associated voter model rates
\[c^{\VM}(x,\xi)=\hxi(x)f_1(x,\xi)+\xi(x)f_0(x,\xi).\] Consider also a
neighbourhood $\cN$ containing the support of $p$. We write
$\xi|_{x+\cN}$ for the function on $\cN$ which maps $y\in\cN$ to
$\xi(x+y)$.
\begin{definition}\label{def:fravmpintro} A voter model perturbation
on $\Z^d$ for $d\ge 2$ with finite range in $\cN$ is a family of
translation invariant spin-flip systems,
$\{\xi^{[\vep]}_\cdot:0<\vep\le \vep_0\}$, for some $\vep_0\in(0,1]$,
with rate functions
\begin{equation}\label{frvprates} c_\vep(x,\xi)=c^{\VM}(x,\xi)+\vep
c^*_\vep(x,\xi)\ge 0\quad\text{for all }x\in\Z^d,\
\xi\in\{0,1\}^{\Z^d},
\end{equation} where for some $g^\vep_0,g^\vep_1:\{0,1\}^\cN\to\R$,
\begin{equation}\label{cstar}
c^*_\vep(x,\xi)=\hxi(x)g_1^\vep(\xi|_{x+\cN})+\xi(x)g_0^\vep(\xi|_{x+\cN}).
\end{equation} In addition there are $g_i:\{0,1\}^\cN\to\R$ such that
\begin{equation}\label{giconv} \Vert g_i^\vep-g_i\Vert_\infty\le
c_g\vep^{r_0}\text{ for $i=0,1$ and all }\vep\text{ and some
}c_g,r_0>0,\text{ if $d\ge 3$},
\end{equation} and
\begin{equation}\label{giconv2} \lim_{\vep\to0+}\Vert
g_i^\vep-g_i\Vert_\infty=0\text{ for $i=0,1$, if $d=2$}.
\end{equation} Finally we assume
\begin{equation}\label{zerotrap} \text{for all $\vep\in(0,\vep_0]$,
${\bf 0}$ is a trap for }\xi^{[\vep]},\text{ that is,
$g^\vep_1(1_\emptyset)=0$},
\end{equation} and, in addition if $d=2$,
\begin{equation}\label{onetrap} \text{for all $\vep\in(0,\vep_0]$,
${\bf 1}$ is a trap for }\xi^{[\vep]},\text{ that is, $g^\vep_0(1_\cN
)=0$}.\qed\end{equation}
\end{definition}

This class of processes is discussed further in Section~\ref{sec:vmp}. At times we will abuse the wording and  say $\xi^{[\vep]}$ is a finite range voter model perturbation, for $0<\vep\le\vep_0$. The following ``asymptotic rate function" associated with the above finite range voter perturbation will play an important role:
\begin{equation}\label{rsdefnintro}
r^s(A):=g_1(1_A)=\lim_{\vep\to 0}\frac{c_\vep(0,1_A)-f_1(0,1_A)}{\vep}\ \ \text{ for }A\subset \cN.
\end{equation}
The above equality is elementary. 
For $d=2$ the ``drift" associated with the above voter model perturbation with finite range in $\cN$ is
\begin{equation}\label{Thetadefnsintro}
\Theta_3:=\sum_{\emptyset\neq A\subset\cN}r^s(A)(\Theta^+(A)-\Theta^-(A)).
\end{equation}

 Fix a neighbourhood $\cN$ in $\Z^d$ where $d\ge 2$. It is easy to see the family of associated $q$-voter
models for $q=1-\vep$ is a finite range voter model perturbation  (see Example~\ref{e:qv} in Section~\ref{sec:vmp}). 
If
\begin{equation}\label{relldefn}r_\ell=(\ell/|\cN|)\log(|\cN|/\ell), \text{ for }\ell=1,\dots,|\cN|,\ \text{ and }r_0=0, 
\end{equation}
then for the family of $q$ voter models and for $A\subset\cN$,
\begin{equation}\label{qvrs}r^s(A)=\lim_{\vep\to 0}\frac{(|A|/|\cN|)^{1-\vep}-(|A|/|\cN|)}{\vep}=r_{|A|}.
\end{equation}
Therefore, for the two-dimensional $q$-voter model we have
\begin{equation}\label{Theta}
\Theta_3:=\Theta = \sum_{\emptyset\neq A\subset\cN}r_{|A|}(\Theta^+(A)-\Theta^-(A)).
\end{equation}

Here is our general complete convergence theorem in two dimensions.
\begin{theorem}\label{t:vgen2dCCT} Assume for $0<\vep\le\vep_0$,
$\xi^{[\vep]}$ is a cancellative and monotone finite range voter model
perturbation in $\Z^2$, and $\Theta_3>0$.  There is an $\vep_1>0$ such
that for $\vep\in(0,\vep_1)$, the complete convergence theorem with
coexistence (CCT) holds for $\xi^{[\vep]}$.
\end{theorem}

Turning to the $q$-voter model in two dimensions with $|\cN|\le 8$, we
have already noted that in any dimension this model is monotone
(elementary), cancellative (Lemma~\ref{l:qcanc}) and a finite range
voter model perturbation (Example~\ref{e:qv}). So it remains to verify
that $\Theta_3>0$, which clearly is a crucial condition, as all the
other conditions hold equally well for the ordinary voter model where
the CCT fails, and $\Theta_3=0$. In Corollary~\ref{postheta} we will
show in complete generality that $\Theta_3>0$ will easily follow from
the strict subadditivity of $r^s$, that is from
\begin{equation}\label{rsubaddsetintro}
r^s(A\cup B)<r^s(A)+r^s(B)\text{ for all non-empty disjoint }A,B\subset\cN.
\end{equation}
For the $q$-voter model this means (recall \eqref{qvrs}) if $r_\ell$
is as in \eqref{relldefn}, then
\begin{equation}\label{rsubaddintro}r_{\ell_1+\ell_2}<r_{\ell_1}+r_{\ell_2}\text{
for all }0<\ell_i,\ \ell_1+\ell_2\le |\cN|.
\end{equation}
This follows from an elementary calculus exercise, and was noted in
Section~5 of \cite{ASD}, where it played an important role in their
analysis of the $q$-voter model for $d\ge 3$. The general condition
\eqref{rsubaddsetintro} owes much to the calculation in
\cite{ASD}. Therefore Corollary~\ref{postheta} and
\eqref{rsubaddintro} imply that
\begin{equation}\label{posthetaintro} \Theta>0.
\end{equation}
We now may apply this and Theorem~\ref{t:vgen2dCCT} to prove
Theorem~\ref{t:CCT} for $d=2$. See Corollary~\ref{c:CCTsubadd} for a
general statement of this reasoning to establish a CCT in two
dimensions.

Theorem~\ref{t:vgen2dCCT} is a two-dimensional version of Theorem~1.2
of \cite{CP14} where a similar result is stated for $d\ge 3$.  More
specifically for $d\ge 3$ this result establishes a complete
convergence theorem with coexistence for cancellative spin-flip
systems which are voter model perturbations (as defined in Section~1
of \cite{CP14}), providing a certain drift is positive. The drift is
$f'(0)$ where
\[ \frac{\partial u}{\partial t}=\sigma^2\frac{\Delta u}{2}+f(u),\] is
the limiting reaction diffusion equation under law of large numbers
scaling (see \cite{ASD} or \cite{CDP13}).  In fact the drift, $f'(0)$,
also equals the positive drift in a limiting super Brownian motion
arising in a low density scaling theorem (see Corollary~1.8 of
\cite{CP05}). That these two drifts coincide is easy and shown on
pages 33-34 in Section~1.8 of \cite{CDP13}, and the fact that the
hypotheses of Corollary~1.8 of \cite{CP05} hold for voter model
perturbations is also verified in the same place. In either
representation, the positive drift is used to show regions of low
density will repopulate to avoid local extinction and the resulting
clumping. The reaction function $f$ is defined in terms of the
invariant measures for the voter model and so is not well-defined for
$d=2$. Therefore, in extending Theorem~1.2 of \cite{CP14} to two
dimensions in Theorem~\ref{t:vgen2dCCT} we replace $f'(0)$ with the
drift $\Theta_3$ in a super Brownian motion low density limit theorem
(Theorem~\ref{t:SBMgenintro} and Remark~\ref{rem:symmSBMintro} in the
next section).  See Section~\ref{sec:scallim} for more about this
convergence to super-Brownian motion.

\begin{remark}\label{betainfdge3} In Theorem~1.2 of \cite{CP14} the
last condition for a (CCT), namely $\beta_\infty(\xi^{[\vep]})>0$ if
$\xi^{[\vep]}$ is not ${\bf 0}$ or ${\bf 1}$, was not part of the
conclusion, but in fact it is easy to argue just as in the
$2$-dimensional result above to derive this condition. See
Remark~\ref{CCTdge3gen} below.
\end{remark} Consider briefly the simpler $d\ge 3$ case of
Theorem~\ref{t:CCT}. We have noted that the hypotheses of Theorem~1.2
of \cite{CP14} have been verified, at least for $|\cN|\le 8$, aside
from the positivity of $f'(0)$.  This last property follows from
Theorem~1.2 of \cite{ASD} (the proof given there for $d=3$ holds in
any dimension).  As a result, we are then able to establish the $d\ge
3$ case of Theorem~\ref{t:CCT} as a direct consequence of Theorem~1.2
of \cite{CP14}. This simple argument is carried out in
Section~\ref{sec:cct3d} below. The same lemma
(Lemma~\ref{l:genpostheta}) which led to the positivity of $\Theta_3$
under the strict subadditivity of $r^s$ in Corollary~\ref{postheta},
also leads to a simple self-contained direct proof of $f'(0)>0$ for
$d\ge 3$ (see Proposition~\ref{p:posthetabigd}).

To prove the $2$-dimensional result, Theorem~\ref{t:vgen2dCCT}, we
will use the more fundamental Proposition~4.1 of \cite{CP14}, instead
of Theorem~1.2 in \cite{CP14}.  The former result holds for arbitrary
$d$ at the cost of bringing in some additional technical hypotheses.
(In fact this result was used to establish Theorem~2.1 in \cite{CP14}
for $d\ge 3$.) Proposition~4.1 of \cite{CP14} is combined with several
other results in \cite{CP14} to prove Theorem~\ref{t:CCT2dgen} below.
This result establishes the CCT for cancellative finite range voter
model perturbations if two additional conditions (\eqref{condA} and
\eqref{condB}) are in force. These additional conditions demonstrate
the ability of $1$'s and $0$'s to both coexist and propagate in space
and time, respectively.  The next step is to follow the derivation of
the CCT for the two-dimensional Lotka-Volterra model in Section~6 of
\cite{CP14} to show that the above conditions will follow from a
set-up which allows a comparison to super-critical oriented
percolation (see \eqref{Hyp2} below) to show simultaneous propagation
of both $0$'s and $1$'s in close proximity as time gets large.  This
is done in Theorem~\ref{t:genCCT}. The last step is to justify the
above super-critical oriented percolation set-up by using a low
density limit theorem in which the scaling limit is a super-Brownian
motion with {\it positive drift} $\Theta_3$
(Theorem~\ref{t:SBMgenintro} and Remark~\ref{rem:symmSBMintro} in the
next section.)

\subsection{A scaling limit theorem in 2 dimensions}\label{sec:scallim}

First consider a scaling limit theorem for the $q$-voter model.  We
will speed up time and scale down space for a two-dimensional
$q$-voter model in the usual Brownian manner and at the same time let
$q\uparrow 1$ at an appropriate rate. In the regime where $1$'s are
relatively rare we show the normalized empirical measure of $1$'s
converges to super-Brownian motion with drift. Such limit theorems are
technically more difficult in two dimensions than higher dimensions,
even in the simple voter model setting \cite{CDP00}. The increased
clustering in two dimensions (e.g from the stronger recurrence of the
dual for the voter model) leads to a greater branching rate, or
equivalently, a greater mass per particle.  The resulting extra $\log
N$ factor complicates even the simplest moment bounds. Convergence to
super-Brownian motion in two dimensions was established for a class of
Lotka-Volterra spin systems (also voter model perturbations) in
Theorem~1.5 of \cite{CMP}. We will refine some of the results and
methods used in that paper to obtain the required scaling limit
theorem. In \cite{CMP} a particular branching coalescing dual process
was used in some key calculations. Such duals seem more complex in our
present setting and so instead we use a systematic comparison of our
model with the voter model over small intervals (see, for example,
Section~\ref{ss:compproc}).  We believe this gives a more robust
approach to general voter model perturbations in the critical and
physically important two-dimensional case.  The higher dimensional
($d\ge 3$) analogues of this limit result follow from \cite{CP05}
which proves a limit theorem for a general class of voter model
perturbations including $q$-voter models as $q\uparrow 1$ (see
Remark~\ref{r:SBM$d>2$} below).

Let $\MF=\MF(\R^2)$ denote the space of finite measures on $\R^2$ with
the topology of weak convergence.  A $2$-dimensional super-Brownian
motion with initial condition $X_0\in\MF(\R^2)$, branching rate $b>0$,
diffusion coefficient $\sigma^2>0$, and drift $\theta\in\R$, denoted
$SBM(X_0,b,\sigma^2,\theta)$, is an $\MF(\R^2)$-valued diffusion $X$
whose law is the unique solution of the martingale problem:
\begin{eqnarray*}
(\mbox{MP}) \ \begin{cases}   \!\!\!&\!\!\!  \forall \ \phi \in C_b^3(\mathbb{R}^2), \quad M_t(\phi) = X_t(\phi)-X_0(\phi) -
\int_0^t X_s\left( \frac{\sigma^2}{2} \Delta \phi + \theta \phi \right) ds
\\ & \ \mbox{ is a continuous } \mathcal{F}_t^X \mbox{-martingale  such that } \\ & \ \ 
 \langle M(\phi) \rangle_t = \int_0^t X_s\left( b \phi^2 \right) ds. \end{cases}
\end{eqnarray*}
Here $C^3_b$ is the set of bounded $C^3$ functions with bounded
continuous partials of order $3$ or less and $\mathcal{F}^X_t$ is the
canonical right-continuous filtration generated by $X$.

For $N\ge e^3$ ($N$ denotes a real number), let 
\begin{equation}\label{vepNdef}N'=N/\log N,\text{ and }\vep_N=(\log N)^3/N. 
\end{equation}
We let $\xi^{(q_N)}_t$ denote a $q_N$-voter model on $\Z^2$ with
$q_N=1-\vep_N$.  Consider the rescaled $q_N$-voter model,
\[\xi^N_t(x)=\xi^{(q_N)}_{Nt}(x\sqrt N),\ \ x\in\SN:=\Z^2/\sqrt N,\]
and define the associated $\MF$-valued empirical process
\begin{equation}\label{XNdefn}X^N_t = \dfrac{1}{N'}\sum_{x\in \SN}\xi^N_t(x)\delta_x.
\end{equation}

\begin{theorem}\label{t:SBM} Assume $\cN$ is a neighbourhood in
$\Z^2$, $\sigma^2=\sigma^2(\cN)$ is as in \eqref{nbhddef} and
$\{\xi^N_0\}$ satisfies $X^N_0\to X_0$ in $\MF$. If $\Theta$ is as in
\eqref{Theta}, then $\Theta>0$ and
\[ X^N \To \SBM(X_0,4\pi\sigma^2,\sigma^2,\Theta) \text{ in the
Skorokhod space $D(\R_+,\MF)$ as }N\to\infty.
\]
\end{theorem}
\noindent The fact that $\Theta>0$ was already noted above in
\eqref{posthetaintro}.

\begin{remark}\label{r:negq} There is a symmetric result for $q>1$
where we take $q_N=1+\vep_N$. As noted in \cite{ASD}, the $r_\ell$ in
this case is the negative of the $r_\ell$ in \eqref{relldefn} and
therefore in \eqref{Theta}, $\Theta_3<0$ will have the opposite sign.
The same proof noted below in Remark~\ref{rem:symmSBMintro} (using the
more general Theorem \ref{t:SBMgenintro}) then gives the conclusion of
Theorem \ref{t:SBM} with drift $\Theta_3<0$, the negative of that in
Theorem~\ref{t:SBM}.
    \end{remark}
\begin{remark} \label{r:SBM$d>2$}The analogue of Theorem~\ref{t:SBM}
for $d\ge 3$ follows from a limit theorem for a class of voter model
perturbations established as Corollary~1.8 of \cite{CP05}.  In this
setting we take $N'=N$ and $\vep_N=1/N$ in our definition of $X^N$. It
is then straightforward to verify the hypotheses of the above result
and so conclude that for $\cN$, $\sigma^2$, and $\{X_0^N\}$ as in
Theorem~\ref{t:SBM} with $d\ge 3$,
\[X^N\To\SBM(X_0,2\gamma_e,\sigma^2,\Theta)\text{ in
$D(\R_+,\MF(\R^d))$ as }N\to\infty.\] Here $\gamma_e\in(0,1)$ is the
escape probability from $0$ of the random walk in $\Z^d$ whose step
kernel is uniform in $\cN$, and
\[\Theta=\sum_{\emptyset\neq A\subset \cN}\beta(A)\hat
P(\tau(A)<\infty, \tau(A\cup\{0\})=\infty)-\delta(A)\hat
P(\tau(A\cup\{0\})<\infty)>0,\] where for $r_\ell$ as in
\eqref{relldefn},
\[\beta(A)=\sum_{\emptyset\neq C\subset A} 1(C\neq
\cN)(-1)^{|A|-|C|}r_{|C|},\text{ and }\delta(A)=\sum_{\emptyset\neq
C\subset A} 1(C\neq \cN)(-1)^{|A|-|C|}r_{|\cN\setminus C|}.\] The
positivity of $\Theta$ follows from Theorem~1.2 of \cite{ASD} (or
Proposition~\ref{p:posthetabigd} below) and the fact that the drift
$\Theta$ agrees with $f'(0)$ where $f$ is the reaction function in the
limiting reaction diffusion equation (see Section~1.8 of
\cite{CDP13}), denoted by $\phi$ in \cite{ASD}.
\end{remark}

We now consider a general two-dimensional limit theorem for a large
class of finite range voter model perturbations. The random walk
kernel, $p$, is as described at the start of
Section~\ref{sec:proofCCTintro} with $d=2$.  The key condition is the
following ``asymptotic $0-1$ symmetry" which was used implicitly for
the special case of two-dimensional Lotka-Volterrra models in
\cite{CMP}.
\begin{definition} \label{def:asysymvmp}Consider a finite range voter
model perturbation in $\Z^2$, $\xi^{[\vep]}$, as in
Definition~\ref{def:fravmpintro} with rates $c_\vep(x,\xi)$ for
$0<\vep\le \vep_0$. We say $\xi^{[\vep]}$ is {\it asymptotically
symmetric} if, in addition, for some $g^a:\{0,1\}^\cN\to\R$,
\begin{equation}\label{asymsym}
\lim_{\vep\to0+}(\log1/\vep)^2(g^\vep_1(\xi)-g^\vep_0(\hxi))=g^a(\xi)\text{
for all }\xi\in\{0,1\}^\cN,
\end{equation} or, equivalently, for some
$c^a:\Z^2\times\{0,1\}^{\Z^2}\to\R$ (necessarily anti-symmetric),
\begin{equation}\label{casymsym}
\lim_{\vep\to0+}(\log1/\vep)^2\frac{c_\vep(x,\xi)-c_\vep(x,\hxi)}{\vep}=c^a(x,\xi)\text{
for all }x\in\Z^2\text{ and }\xi\in\{0,1\}^{\Z^2}.
 \end{equation}
\end{definition}
The relationship between $c^a$ and $g^a$ is (note that $c^a$ is
translation invariant)
\[c^a(0,\xi)=\widehat\xi(0)g^a(\xi\vert_\cN)-\xi(0)g^a(\widehat\xi|_\cN).\]
Further discussion may be found in Section~\ref{sec:vmp}.  The
following asymmetric function will also be important for our limit
theorem:
\begin{equation}
r^a(A):=g^a(1_A)=\lim_{\vep\to0}(\log (1/\vep))^2\frac{c_\vep(0,1_A)-c_\vep(0,1_{{\bar\cN}\setminus A})}{\vep}.
\end{equation}
The above equality is again elementary. 

We are ready to state our general two-dimensional limit theorem. It
will not require monotonicity or the cancellative property.  Recall
the notation $K_2(A_1,A_2)$ from Section~\ref{sec:coalrw}, and that
$\sigma^2$ is as in \eqref{covpintro}.  A second ``drift" parameter
associated with such voter model perturbations will be denoted by
\begin{equation}\label{Thetadefnsintro2} \Theta_2=\sum_{\emptyset\neq
A\subset\cN} r^a(A)K_2(A,\bar\cN\setminus A).
\end{equation}
\begin{theorem}\label{t:SBMgenintro} Assume $\{\xi^{[\vep]}:0<\vep\le
\vep_0\}$ is an asymptotically symmetric finite range voter model
perturbation on $\Z^2$. Let $\xi_t^N(x)=\xi^{[\vep_N]}_{Nt}(x\sqrt
N)$, $x\in\SN$. Define a measure-valued process by
\begin{equation}\label{genlXNdefintro}X^N_t=(1/N')\sum_{x\in\SN}\xi^N_t(x)\delta_x.
\end{equation}
If $X_0^N\to X_0$ in $\MF$, then 
\[
X^N \To \SBM(X_0,4\pi\sigma^2,\sigma^2,\Theta_2+\Theta_3) \text{ in
  the Skorokhod space $D(\R_+,\MF)$ as }N\to\infty. 
\]
\end{theorem}
\begin{remark} \label{rem:symmSBMintro} Clearly a finite range voter
model perturbation which is symmetric, that is
$c_\vep(x,\xi)=c_\vep(x,\hxi)$ for all $x$ and $\xi$, is
asymptotically symmetric with $c^a=r^a=\Theta_2=0$.  It will turn out
that a cancellative finite range voter model perturbation is
necessarily symmetric for each $\vep$ (see Remark~\ref{rem:sym}).
Therefore in the setting of our general CCT
(Theorem~\ref{t:vgen2dCCT}), the above limit theorem applies with
$\Theta_2=0$ and the drift for our limiting SBM is indeed $\Theta_3$,
as in the discussion in Section~\ref{sec:proofCCTintro}.  In
particular, this is the case for the $q$-voter model. Therefore,
recalling \eqref{Theta} and \eqref{posthetaintro}, we see that
Theorem~\ref{t:SBM} is an immediate consequence of the general
Theorem~\ref{t:SBMgenintro} above.
\end{remark}
 
Theorem~\ref{t:SBMgenintro} includes all the examples we know of
super-Brownian limits for voter model perturbations in two dimensions,
as well as a number of new ones. In addition to the above result for
$q$-voter models and the limit theorem for the ordinary
$2$-dimensional voter model in \cite{CDP00}, this includes the basic
limit theorems for Lotka-Volterra models in \cite{CP08} (see
Example~\ref{e:lv3}), the more refined Lotka-Volterra limit theorems
in \cite{CMP} (see Example~\ref{e:lv2}), and limit theorems for the
affine and geometric voter model (Examples~\ref{e:av3} and
\ref{e:gm3}, respectively).

The following ``survival" corollary is an easy consequence of
Theorem~\ref{t:SBMgenintro} and standard arguments (see
Section~\ref{sec:percsetup}).
\begin{corollary}\label{cor:surv} Assume for $0<\vep\le\vep_0$,
$\xi^{[\vep]}$ is a monotone asymptotically symmetric finite range
voter model perturbation in $\Z^2$, and $\Theta_2+\Theta_3>0$.  There
is an $\vep_1>0$ such that for $\vep\in(0,\vep_1)$,
$P_{\delta_0}(|\xi^{[\vep]}_t|>0\text{ for all }t\ge 0)>0$.
\end{corollary}

Section~\ref{s:crw} discusses coalescing random walks and proves
Proposition~\ref{p:CRWnew}. Cancellative processes are defined, and
many of their properties are presented, in
Section~\ref{sec:canc}. Here the criterion for a nonlinear voter model
to be cancellative (Proposition~\ref{p:nlvm} from \cite{CD91}) is
proved for completeness, and then applied to show the $q$-voter model
is cancellative for all $q\in[0,1]$ if $|\cN|=4$. The $|\cN|=8$ case
for $q$ near $1$ and $q<1$ is outlined here, while the actual
\texttt{maple}-assisted proof is presented in an Appendix. Additional
properties of finite range voter model perturbations and associated
notation are presented in Section~\ref{sec:vmp}.  Several examples of
cancellative finite range voter model perturbations are presented here
as well. The short proof of Theorem~\ref{t:CCT} for $d\ge 3$ is
presented in Section~\ref{sec:cct3d}. The general complete convergence
theorem, Theorem~\ref{t:vgen2dCCT}, is proved in
Section~\ref{sec:proofCCT}, assuming Theorem~\ref{t:ThetaOP}, which
states that the conditions for a block comparison to super-critical
percolation, \eqref{Hyp2}, will hold for a monotone, asymptotically
symmetric, finite range voter model perturbation, if
$\Theta_2+\Theta_3>0$.
This section also includes a coupled SDE construction of our particle
system $\xi$ along with $\hxi$, and killed versions of these
processes, on a common probability space
(Proposition~\ref{p:infinitesde} and the ensuing
\eqref{couplingforintro}).  This set-up is used in the proof of an
intermediate result (Theorem~\ref{t:genCCT}) in which a CCT is
established assuming \eqref{Hyp2} in place of $\Theta_2+\Theta_3>0$.
Theorem~\ref{t:ThetaOP} is proved in Section~\ref{sec:percsetup} as a
corollary to the weak convergence result, Theorem~\ref{t:SBMgenintro},
{\it and its proof}. The latter result is proved in
Sections~\ref{sec:couplingsemimart}-\ref{sec:convtoSBM} and
Section~\ref{sec:Ibound}.  Section~\ref{sec:couplingsemimart} sets up
the approximating martingale problems, using a convenient SDE
coupling, and also gives a number of examples of the weak convergence
theorem.  A number of preliminary bounds and sharp estimates on the
drift terms are given in Section~\ref{s:qvdrifttm}, while the proof of
convergence to SBM is in Section~\ref{sec:convtoSBM}.  A key technical
bound, Proposition~\ref{p:key2bnd}, used for exact asymptotics on the
drift terms, is proved in Section~\ref{sec:Ibound}.

\medskip
 
\noindent{\bf Acknowledgement.} It is a pleasure to thank Mathieu
 Merle for his help with the proof of Proposition~\ref{p:CRWnew}.

\section{Coalescing probability asymptotics for
  two-dimensional random
  walks}\label{s:crw}

We work in the setting of Section~\ref{sec:coalrw}, and in particular,
$p(\cdot)$ is the general random walk kernel considered there, and
$q_{x_1,\dots,x_n}(t)$ and $K_n(x_1,\dots,x_n)$ are as in
\eqref{CRWold}. In fact, as the reader can easily check, the proof
below includes a derivation of \eqref{CRWold}.
\begin{lemma}\label{l:CMP1.3ext} Assume $n\ge 2$ and $x_1,\dots,x_n\in
\Z^2$ are distinct. There are positive constants $C_{\ref{l:qxtbnd}}$
and $t_0>0$ depending only on $p$ and $n$ such that
\begin{equation}\label{l:qxtbnd} (\log t)^{\binom n
2}q_{x_1,\dots,x_n}(t) \le C_{\ref{l:qxtbnd}} K_n(x_1,\dots,x_n)
\text{ if } t\ge 2\big(\max_{i\ne j}\{|x_i-x_j|^4\}\vee t_0\big).
\end{equation}
\end{lemma}

\begin{proof} The argument here is an extension of the one given in
Section~9 of \cite{CMP}.  Let $\ds \{\tilde B^x_t,x\in\Z^2\}$ be a
system of (non-coalescing) independent rate one continuous time random
walks with jump kernel $p$. Let $x_1,\dots,x_n\in\Z^2$ be distinct,
set $x=(x_1,\dots,x_n)$ and define the non-collision event
\[ D_t = \big\{ \tilde B_s^{x_i}\neq\tilde B_s^{x_j}\text{ for all
}s\le t\text{ and }i\neq j \big\}.
\] Clearly, if $P_x$ is the law of $(\tilde B^{x_1},\dots,\tilde
B^{x_n})$, then $q_t(x)=P_x(D_t)$.  Dependence on the fixed natural
number $n$ is suppressed.

By Lemma~9.12 of \cite{CMP}, there is a positive constant
$C_{\ref{D2tbnd}}$ depending only on $p$ and $n$ such that for
$x_1,\dots,x_n$ as above,
\begin{align}\label{D2tbnd} \nonumber P_{x}(D_{2t}|D_t) = 1 -
\frac{\binom n2 \log 2}{\log t} +\frac{c(x,t)}{(\log t)^{3/2}},
&\text{ where }|c(x,t)| \le C_{\ref{D2tbnd}}\\ & \text{ whenever
}\max_{i\neq j}\{|x_i-x_j|^4\}\vee e^4\le t.
\end{align} For $t\ge 1$ define $f_x(t)= (\log t)^{\binom n2}P_x(D_t)$
and $\ds k(t) = \max\{i\ge 0: 2^{i}\le t <2^{i+1}\}$. Since $P_x(D_t)$
is decreasing in $t$, it is easy to see that
\begin{equation}\label{t2kbnds} \frac{ (k(t))^{\binom n2}}{
(k(t)+1)^{\binom n2}} f_x(2^{k(t)+1}) \le f_x(t) \le \frac{
(k(t)+1)^{\binom n2}} { (k(t))^{\binom n2}}f_x(2^{k(t)}) .
\end{equation}

For $m,m'\ge 1$, iterating conditional probabilities leads to
\begin{multline}\label{f(2)}
\qquad f_x(2^{m+m'}) = f_x(2^{m}) \prod_{i=0}^{m'-1}
\Big(1+\frac{1}{m+i}\Big)^{\binom n2}
P_x(D_{2^{m+i+1}}| D_{2^{m+i}})
\\= f_x(2^{m}) \prod_{k=m}^{m+m'-1}
\Big(1+\frac{1}{k}\Big)^{\binom n2}
P_x(D_{2^{k+1}}| D_{2^{k}}).\qquad
\end{multline}
To make use of \eqref{f(2)}, we let $\bar c(x,k)=c(x,2^k)/(\log
2)^{3/2}$, and note that it follows from \eqref{D2tbnd} that
\begin{align}\label{e:b0}
\nonumber P_x(D_{2^{k+1}}|D_{2^{k}} )= 1- \frac{\binom n2}{k} 
+ \frac{\bar c(x,k)}{k^{3/2}}, &\text{ where }|\bar c(x,k)| \le C_{\ref{e:b0}}\\
& \text{ whenever }\ds 2^k\ge 
\max_{i\ne j}\{|x_i-x_j|^4\}\vee e^4.
\end{align}
By the binomial theorem there is a constant $C_{\ref{e:b1}}>0$
(depending only on $n$) so that
\begin{equation}\label{e:b1} (1+\frac{1}{k}\Big)^{\binom n2} = 1 +
\frac{\binom n2}{ k} + \frac{\underline c(k)}{k^2}\text{ where
}\sup_{k\in\N}|\underline c(k)| \le C_{\ref{e:b1}}.
\end{equation}
Taken together, the last two facts imply that for some
$C_{\ref{e:b2}}>0$, depending only on $p(\cdot)$ and $n$, there are
constants $\tilde c(x,k)$ satisfying
\begin{equation}\label{e:b2}
\Big(1+\frac{1}{k}\Big)^{\binom
  n2} P_x(D_{2^{k+1}}| D_{2^{k}}) = 1 + \frac{\tilde
  c(x,k)}{k^{3/2}}\ \text{ where }\ |\tilde c(x,k)|\le C_{\ref{e:b2}}\text{ for }
2^k\ge
\max_{i\ne j}\{|x_i-x_j|^4\}\vee e^4.
\end{equation}
Use this in \eqref{f(2)} to see that, for all $m$ and $x$ satisfying
$\ds 2^m\ge \max\{|x_i-x_j|^4\}\vee e^4$,
\begin{equation}\label{f(2)2} f_x(2^{m+m'}) = f_x(2^{m})
\prod_{k=m}^{m+m'-1} \Big(1+\frac{\tilde c(x,k)}{k^{3/2}}\Big), \text{
where }\sup_{k\ge m}|\tilde c(x,k)|\le C_{\ref{e:b2}}.
\end{equation} If $m$ also satisfies $m\ge j_0\equiv
\lceil(2C_{\ref{e:b2}})^{2/3}\rceil$, then the above bound implies
\begin{equation}\label{e:1/2} 1 + \frac{\tilde c(x,k)}{k^{3/2}} \ge
1/2 \text{ if }k\ge m,
\end{equation} and in particular, each factor in the product in
\eqref{f(2)2} is strictly positive.

Define $t_0=e^4\vee 2^{j_0}$ and $\ds m_0(x)=\min\{m: 2^m\ge
\max_{i\ne j}\{|x_i-x_j|^4\vee t_0 \}\,\ge j_0$.  Then for any $m\ge
m_0(x)$, \eqref{f(2)2} and \eqref{e:1/2} imply
\begin{equation}\label{e:b3} \lim_{m'\to\infty}f_x(2^{m+m'}) =
f_x(2^{m}) \prod_{k=m}^{\infty} \Big(1+\frac{\tilde
c(x,k)}{k^{3/2}}\Big)
\end{equation} exists, and is strictly positive.  The fact that
$\lim_{k\to\infty}f_x(2^k)$ exists combined with \eqref{t2kbnds}
actually proves \eqref{CRWold}, with
\begin{equation}\label{e:b4} K_n(x_1,\dots,x_n) = f_x(2^{m})
\prod_{k=m}^{\infty} \Big(1+\frac{\tilde c(x,k)}{k^{3/2}}\Big) \text{
for all } m\ge m_0(x).
\end{equation}
It follows from this and \eqref{f(2)2} that for all $m\ge m_0(x)(\ge
j_0)$,
\begin{equation}\label{mKnbnd} \frac{f_x(2^{m})}{K_n(x_1,\dots,x_n)} =
\Big[\prod_{k=m}^{\infty} \Big(1+\frac{\tilde
c(x,k)}{k^{3/2}}\Big)\Big]^{-1} \le \Big[\prod_{k=j_0}^{\infty}
\Big(1-\frac{C_{\ref{e:b2}}}{k^{3/2}}\Big)\Big]^{-1}<\infty,
\end{equation} where we have also used \eqref{e:1/2} and the
definition of $j_0$.

Finally, if $\ds t\ge 2(\max_{i\ne j}|x_i-x_j|^4\vee t_0)$, then
$2^{k(t)}\ge t/2$ implies $k(t)\ge m_0$, and by \eqref{mKnbnd},
\[ f_x(2^{k(t)})\le K_n(x_1,\dots,x_n) \Big[\prod_{k=j_0}^{\infty}
\Big(1-\frac{C_{\ref{e:b2}}}{k^{3/2}}\Big)\Big]^{-1}.
\] It now follows from \eqref{t2kbnds} and \eqref{mKnbnd} that for
such $t$ (note also $k(t)\ge 1$),
\begin{equation} f_x(t) \le \frac{ (k(t)+1)^{\binom n2}} {
(k(t))^{\binom n2}} f_x(2^{k(t)}) \le 2^{\binom n2}
\Big[\prod_{k=j_0}^{\infty}\Big(1-
\frac{C_{\ref{e:b1}}}{k^{3/2}}\Big)\Big]^{-1} K_n(x_1,\dots,x_n) .
\end{equation} This completes the proof of Lemma~\ref{l:qxtbnd}.
\end{proof}

For disjoint finite nonempty sets of $\Z^2$, $A_1,\dots A_n$,
introduce
\[ \Gamma_t = \big\{ \sigma(A_1,\dots,A_n)> t, \tau(A_1,\dots,A_n)<
t\big\},
\] and define $q_{A_1,\dots,A_n}(t)=\widehat P(\Gamma_t)$.  Note that
if $A_i=\{a_i\}$ are all singletons, then \break
$\tau(A_1,\dots,A_n)=0$, and so
$q_{A_1,\dots,A_n}(t)=q_{a_1,\dots,a_n}(t)$ agrees with our earlier
notation.
\begin{proof}[Proof of Proposition~\ref{p:CRWnew}] Assume $A_i,\, i\le
n$ and $a_i\in A_i$ are as in Proposition~\ref{p:CRWnew}, and define
$K_n(A_1,\dots,A_n)(\le\infty)$ as in \eqref{KnDef}.  %The first step
is to prove that $K_n(A_1,\dots,A_n)$, given by \eqref{KnDef} is
finite.  On the event $\ds \Gamma_t$, $ B^{a_i}_t\ne B^{a_j}_t$ for
all $i\ne j$, which implies $q_{A_1,\dots,A_n}(t)\le
q_{a_1,\dots,a_n}(t)$. It follows from \eqref{CRWold} that
\begin{equation}\label{limsupqt} \limsup_{t\to\infty}(\log t)^{\binom
n 2}q_{A_1,\dots,A_n}(t)\le K_n(a_1,\dots,a_n) <\infty.
\end{equation} Letting $\tau^*=\tau(A_1,\dots,A_n)$ and $\sigma^*=
\sigma(A_1,\dots,A_n)$, and applying the Markov property at time
$\tau^*$, we have
\begin{multline}\label{qtrep} (\log t)^{\binom n2}q_{A_1,\dots,A_n}(t)
\\ = \sum_{\text{distinct }x_1,\dots, x_n\in\Z^2} \int_0^{t}\widehat
P(\sigma^*>\tau^*\in du, B^{a_i}_{\tau^*}=x_i\text{ for }1\le i\le n)
(\log t)^{\binom n2}q_{x_1,\dots,x_n}(t-u) .
\end{multline}
By \eqref{CRWold}, for fixed $u\in[0,t)$, and distinct
$x_1,\dots,x_n\in\Z^2$,
\begin{equation}\label{t-ulim} \lim_{t\to\infty}(\log t)^{\binom
n2}q_{x_1,\dots,x_n}(t-u) = K_n(x_1,\dots,x_n).
\end{equation} In view of the definition of $K_n(A_1,\dots,A_n)$,
\eqref{limsupqt}, \eqref{qtrep}, and \eqref{t-ulim}, Fatou's Lemma
implies that
\begin{align}\label{Kliminf}
\nonumber K_n(A_1,\dots,A_n) \le \liminf_{t\to\infty}
(\log t)^{\binom n
  2}q_{A_1,\dots,A_n}(t)&\le \limsup_{t\to\infty}
(\log t)^{\binom n
  2}q_{A_1,\dots,A_n}(t)\\
  &\le K_n(a_1,\dots,a_n) <\infty .
\end{align}

Now introduce
\[
\Delta^* = \max_{i\ne j}
\{|B^{a_i}_{\tau^*}- B^{a_j}_{\tau^*}|^4\},
\]
and the disjoint decomposition $\Gamma_t=\Gamma^{(1)}_t
\cup \Gamma^{(2)}_t \cup \Gamma^{(3)}_t$,
where
\begin{align*}
\Gamma^{(1)}_t & = \{ \sigma^*>t,
\tau^*\in(t^{1/3},t]\}\\
\Gamma^{(2)}_t &= \{ \sigma^*>t,
\tau^*\le t^{1/3},\Delta^*>t/4\} \\
\Gamma^{(3)}_t &= \{ \sigma^*>t,
\tau^*\le t^{1/3},\Delta^*\le t/4\} .
\end{align*}
We will show that as $t\to\infty$,
\begin{align} \label{gamt12}
(\log t)^{\binom n
  2}\widehat P(\Gamma_t^{(i)})&\to 0 \text{ for }i=1,2, \text{ and }\\
\label{q3lim}
(\log t)^{\binom n
  2}\widehat P(\Gamma_t^{(3)})&\to K_n(A_1,\dots,A_n),
\end{align}
proving \eqref{CRWnew}. It is in the proof of
\eqref{q3lim} that we use Lemma~\ref{l:CMP1.3ext}.

On the event $\Gamma^{(1)}_t$, $\ds\sum_{i=1}^n |B_{t^{1/3}}^{A_i}|\ge
n+1$ (recall that $|A_1|+\cdots+|A_n|\ge n+1$).  It follows that if
$A=\cup_{i=1}^n A_i$, then
\begin{equation}\label{gam1} (\log t)^{\binom n 2}\widehat
P(\Gamma^{(1)}_t) \le \sum_{\text{distinct }y_0,y_1,\dots,y_{n}\in A}
(\log t)^{\binom n 2} q_{y_0,y_1,\dots,y_n}(t^{1/3}) \to 0\text{ as
}t\to\infty,
\end{equation} since each probability in the sum is $O\big((\log
t)^{-\binom {n+1}2}\big)$ by \eqref{CRWold}.

On the event $\Gamma^{(2)}_t$, there must exist $a\ne b\in A$ such
that $|B^a_u- B^b_u|> (t/4)^{1/4}$ for some $u\le t^{1/3}$. Assuming
$\ds (t/4)^{1/4}> 2\max_{a\ne b\in A}|a-b|$, we have the crude bound
\begin{align*}
\widehat P(\Gamma^{(2)}_t) &\le \sum_{a\ne b\in A}
\widehat P\Big(\sup_{0\le u\le   t^{1/3}}|B^a_u-B^b_u| \ge (\tfrac
t4)^{1/4}\Big)\\ 
&= \sum_{a\ne b\in A}
\widehat P\Big(\sup_{0\le u\le   t^{1/3}}|B^0_{2u}| \ge \tfrac12
(\tfrac t4)^{1/4}\Big)\\ 
&\le \binom{|A|}2 \widehat P\Big(\sup_{0\le u\le   2t^{1/3}}|B^0_u|^2 \ge
\tfrac14 (\tfrac t4)^{1/2}\Big).
\end{align*}
Doob's inequality implies
\[ (\log t)^{\binom n2}\widehat P(\Gamma^{(2)}_t) \le \binom{|A|}2
(\log t)^{\binom n2} \frac{\widehat
E(|B^0_{2t^{1/3}}|^2)}{\frac14(t/4)^{1/2}} = 32 \sigma^2 \binom{|A|}2
(\log t)^{\binom n2} t^{-1/6}\to 0\text{ as }t\to\infty,
\]
completing the proof of \eqref{gamt12}

To handle $\Gamma^{(3)}_t$, we recall \eqref{l:qxtbnd} and suppose
that $t>4t_0\vee 8$ (so that $t-t^{1/3}>t/2)$.  Applying the Markov
property at time $\tau^*$ gives us
\begin{multline}
(\log t)^{\binom n 2} \widehat P(\Gamma^{(3)}_t)\\
= 
\sum_{\substack{\text{distinct }x_1,\dots,
    x_n\in\Z^2\\
   \max_{i\ne j}\{|x_i-x_j|^4\}\le t/4}}
\int_0^{t^{1/3}}\widehat P(\sigma^*>\tau^*\in 
du, B^{a_i}_{\tau^*}=x_i\text{ for }1\le i\le n)
(\log t)^{\binom n 2} q_{x_1,\dots,x_n}(t-u). \label{q3rep}
\end{multline}
For $u\le t^{1/3}$ and $t/4\ge \max_{i\ne j}\{|x_i-x_j|^4\}$ we have
(use also $t>4t_0\vee 8$)
\[(t-u)/2>(t-t^{1/3})/2>t/4\ge t_0\vee\max_{i\ne j}\{|x_i-x_j|^4\},\]
and so Lemma~\ref{l:CMP1.3ext} applies to show that
\[
(\log t)^{\binom n 2} q_{x_1,\dots,x_n}(t-u) \le
\frac{ (\log t)^{\binom n2}}
{(\log(t-u))^{\binom n2}}
C_{\ref{l:qxtbnd}} K_n(x_1,\dots,x_n)
\le  2^{\binom n 2}C_{\ref{l:qxtbnd}} K_n(x_1,\dots,x_n).
\]
For the last inequality use $t-u>t/2$ and $t>8$.  The right-hand side
is integrable with respect to $\widehat P(\sigma^*>\tau^*\in
du)1(x_1,\dots,x_n\text{ distinct})\,d\lambda$ for counting measure
$\lambda$ on $(\Z^2)^{n}$, by \eqref{Kliminf} and the definition of
$K_n(A_1,\dots,A_n)$. So we can use \eqref{t-ulim} and dominated
convergence to take the limit as $t\to\infty$ inside the integral in
\eqref{q3rep} and hence, recalling again the definition of
$K_n(A_1,\dots,A_n)$, prove \eqref{q3lim} and so complete the proof.
\end{proof} 

\section{Cancellative processes, voter model perturbations and examples}

\subsection{Cancellative processes}\label{sec:canc}
Let $\cN$ be a non-empty finite subset of $\Z^d\setminus\{0\}$, and
call such a subset a {\sl{general}} neighbourhood.  Our cancellative
processes are translation invariant spin systems with rate functions
satisfying
\begin{equation}\label{canc} c(x,\xi)= \frac{k_0}{2} \left[ 1 -
(2\xi(x)-1) \sum_{A\subset \bar \cN}\beta_0(A)H(\xi,A+x) \right],
\end{equation}
where $H(\xi,A)=\prod_{y\in A}(2\xi(y)-1)$, $k_0$ is a positive
constant, $\beta_0\ge 0$, $\beta_0(\varnothing)=0$, and
$\sum_{A\subset\bar\cN}\beta_0(A)=1$. This last implies $\1$ is a trap
for $\xi_t$, that is, $c(x,\1)\equiv0$.  The restriction to subsets
$A$ of $\bar\cN$ means our definition is a bit more restrictive than
that in (1.16) of \cite{CP14} or (4.4) in Section III.4 of
\cite{Lig85}.  A recent summary of properties of cancellative
processes is given in Sections 1 and 2 of \cite{CP14}. We say $\xi$ is
a {\it good} cancellative process if, in addition, $\beta_0(A)>0$ for
some $A$ with $|A|>1$.
\begin{remark}\label{rem:votercanc} For any probability $q_0$ on
$\cN$, the voter model with kernel $q_0$ is cancellative with $k_0=1$,
and $\beta_0(A)=q_0(y)$ if $A=\{y\}$, and zero if $|A|\neq 1$, as one
can easily check (or see Example III.4.16 in \cite{Lig85}). So the
voter model is cancellative, but not a good cancellative process. The
converse also holds and equally easy to check: If $\xi$ is
cancellative but not good cancellative, then $\xi$ has rates equal to
$k_0$ times those of a voter model with kernel $q_0(y)=\beta_0(\{y\})$
($y\in\cN$).
\end{remark}
Another useful condition, which will hold for our main results, is
\begin{equation}\label{evencond}
\text{$\beta_0(A)=0$ if
$|A|$ is even},
\end{equation} 
which by Lemma~2.1 in \cite{CP14} is equivalent to 
\begin{equation}\label{0trap}
\text{$\0$ is a trap for $\xi_t$,} 
\end{equation}
and also to
\begin{equation}\label{c01symm}
\xi \text{ is $0-1$ symmetric, that is, }c(x,\xi)=c(x,\hxi).
\end{equation}

A cancellative process has a translation invariant stationary
distribution, which is the weak limit of the process started in
Bernoulli product measure with density $1/2$ (see Corolllary~III.1.8
of \cite{G79}). Under the above symmetry it will have density $\1/2$
and so we denote it by $\nu_{1/2}$. A simple proof of the existence of
$\nu_{1/2}$ under \eqref{c01symm} is given after Lemma~2.1 in
\cite{CP14}. It is possible that $\nu_{1/2}=(1/2)(\delta_{\0} +
\delta_{\1})$, and so it need not have the coexistence property. For
example, this is the case for the voter model when $d\le 2$
(Corollary~V.1.13 of \cite{Lig85}).

As defined in \cite{CD91}, a general neighbourhood $\cN$ and
nonnegative sequence $a=(a_\ell)$, $1\le \ell\le |\cN|$,
$(a_l)\not\equiv 0$, defines a nonlinear voter model $\xi_t:\Z^d\to
\{0,1\}$ if the rate at which the opinion at a site flips to the
opposite opinion is $a_\ell$ if $\ell$ of its neighbours hold this
opposite opinion. That is, $\xi_t$ is the spin-flip system defined by
the rate function
\[
c(x,\xi) = \hxi(x)\sum_{\ell=1}^{|\cN|} a_\ell1\{n_1(x,\xi)=\ell\} +
\xi(x)\sum_{\ell=1}^{|\cN|} a_\ell1\{n_0(x,\xi)=\ell\},
\]
where $n_i(x,\xi)=\sum_{y\in x+\cN}1\{\xi(y)=i\}$, $i=1,2$. Clearly
every nonlinear voter model rate function $c(x,\xi)$ satisfies the
symmetry condition \eqref{c01symm}.  In \cite{CD91} $\cN$ satisfied
additional symmetry and irreducibility conditions, but they are not
needed for the results in this section.  Henceforth we exclude the
trivial case of $a\equiv 0$ from consideration.  The $q$-voter model
on $\Z^d$ (with neighbourhood $\cN$) is the nonlinear voter model with
$a_\ell = (\ell/|\cN|)^q$.

The following result, from page 129 of \cite{CD91} provides a means of
checking that a nonlinear voter model is cancellative without checking
\eqref{canc} directly.  We give the short proof for completeness and
to highlight the choice of $\beta_0$ and $k_0$ which will enter later.

\begin{proposition}\label{p:nlvm} Let $\cN$ be a general neighbourhood
and $a_\ell\ge 0$ for $\ell=1,\dots,|\cN|$ with $(a_\ell)\not\equiv
0$.  Define the $|\cN|\times |\cN|$ matrix $\bM$ by
\begin{equation}\label{Mdef} M(k,j) = \sum_{\text{odd }i \le j\wedge
k} \binom ji \binom{|\cN|-j}{k-i} , \qquad 1\le k,j \le |\cN|.
\end{equation} If there is a nonnegative sequence $\alpha=(\alpha_k)$,
$1\le k\le |\cN|$, such that
\begin{equation}\label{aalpha} a_\ell = \sum_{k=1}^{|\cN|} \,
\alpha_{k} \, M(k,\ell), \quad 1\le \ell\le |\cN|,
\end{equation} then the nonlinear voter model determined by $(a_\ell)$
is a \textbf{cancellative} nonlinear voter model. Moreover in
\eqref{canc} $k_0=\sum_{\ell=1}^{|\cN|}\alpha_\ell$ and for
$A\subset\bar\cN$,
\begin{equation}\label{beta0def} \beta_0(A) =\frac{1}{k_0}\times
\begin{cases} 0&\text{if } A=\{0\}\text{ or } |A|\text{ is even}\\
\alpha_{m}&\text{if }0\notin A, |A|=m\text{ is odd}\\
\alpha_{m-1}&\text{if }0\in A, |A|=m \text{ is odd},
\end{cases}
\end{equation} and so \eqref{evencond} also holds.
\end{proposition}
\begin{proof}
Define $\xi(A)=\sum_{a\in A}\xi(a)$ and $\hxi(A)=\sum_{a\in
A}\hxi(a)$.  In \eqref{canc} since $2\xi(x)-1= (-1)^{\hxi(x)}$ and
$\sum_A\beta_0(A)=1$, we can rewrite \eqref{canc} for $x=0$ as
\begin{equation}\label{e:c2} c(0,\xi)= \frac{k_0}{2}
\sum_{A\subset\bar\cN}\beta_0(A) \Big[1 - (-1)^{\hxi(A)+\hxi(0)}
\Big].
\end{equation}

Let $\balpha=(\alpha_j)$ be a given nonnegative sequence such that
\eqref{aalpha} defines non-negative $a_k$, not all identically
zero. Let $k_0=\sum_{\ell=1}^{|\cN|}\alpha_\ell>0$, the latter since
otherwise $a\equiv0$.  For $A\subset\bar\cN$, define $\beta_0(A)$ by
\eqref{beta0def}. Then $\beta_0(\emptyset)=0$ and
$\sum_{A\subset\bar\cN}\beta_0(A)=\frac{1}{k_0}\sum_{\ell=1}^{|\cN|}\alpha_\ell=1$. Note 
also that \eqref{evencond} holds and hence so does \eqref{c01symm}.
With these choices, \eqref{e:c2} becomes
\begin{equation}\label{e:c3} c(0,\xi)= \frac{1}{2}\sum_{m\text{ odd}}
\bigg\{\alpha_m\sum_{A\not\ni 0,|A|=m} + \1(m\ge
3)\alpha_{m-1}\sum_{A\ni 0,|A|=m}\bigg\} \Big[1 -
(-1)^{\hxi(A)+\hxi(0)} \Big].
\end{equation} Observe that in the second sum, $m\ge 3$ because we
have set $\beta(\{0\})=0$.

Observe that if $m$ is odd and $|A|=m$, if $\xi(0)=0$, then
\[
\frac12 \Big[1 - (-1)^{\hxi(A)+\hxi(0)} \Big]
=\frac12 \Big[1 + (-1)^{\hxi(A)} \Big]
=1\{\hxi(A)\text{ is even}\} =
1\{\xi(A)\text{ is odd}\},
\]
where we have used the fact that $\xi(A)+\hxi(A)=|A|$. Therefore,
\begin{equation}\label{e:c5}
c(0,\xi)= \sum_{m\text{ odd}}
\bigg\{\alpha_m\sum_{A\not\ni 0,|A|=m} +
\1(m\ge 3)\alpha_{m-1}\sum_{A\ni 0,|A|=m}\bigg\}
1\{\xi(A)\text{ is odd}\}.
\end{equation}
For any
$\xi$, odd $m$, $1\le j\le |\cN|$, define
\begin{align*} \cA'_m(j) &=\Big\{ A\subset\bar\cN:0\notin A, |A|=m,
\xi(\cN)=j, \xi(A)\text{ is odd}\Big\}\\ \cA''_m(j) &=\Big\{
A\subset\bar\cN:0\in A, |A|=m, \xi(\cN)=j,\xi(A)\text{ is odd}\Big\}.
\end{align*}
Then 
\begin{equation}\label{e:c6}
c(0,\xi) =
\sum_{m\text{ odd}} \alpha_m|\cA'_m(j)| +
\sum_{m\ge3,\text{ odd}} \alpha_{m-1}|\cA''_m(j)|
\qquad \text{if }\xi(0)=0,\, \xi(\cN)=j.
\end{equation}

Continue to assume $\xi(\cN)(=n_1(0,\xi))=j$ and $\xi(0)=0$, consider
$A\in\cA'_m(j)$, and the disjoint union
\[ A = (A\cap\{x:\xi(x)=1\})\cup(A\cap\{x:\xi(x)=0\}).
\] By counting the number of ways to choose the first set with the
requirement $i=\xi(A)\le j\wedge m$ odd, we find that %for any $\xi$,
summing over odd $i\le j$,
\[ | \cA'_m(j)| = \sum_{\substack{i\text{ odd}\\i\le j\wedge m}}\binom
ji \binom{|\cN|-j}{m-i} .
\] Here we build $A\subset \cN$ by first choosing the $i$ sites of $A$
to put $1$'s from the available $j$ sites in state $1$, and then
choose the $m-i$ sites of $A$ to put $0$'s from the $|\cN|-j$
available sites in state $0$.  Similar reasoning (recall $\xi(0)=0$
and now $0$ must be in $A$) leads to
\[
| \cA''_m(j)|=
\sum_{\substack{i\text{ odd}\\i\le j\wedge (m-1)}}\binom ji
\binom{|\cN|-j}{m-1-i}. 
\]
Insert the above into \eqref{e:c6} to get for $\xi(0)=0,\ \xi(\cN)=j$,
\begin{align*}
c(0,\xi) &= 
\sum_{1\le k\le |\cN|} \alpha_{k}
\sum_{\substack{i\text{ odd}\\i\le j\wedge k}}
  \binom ji \binom{|\cN|-j}{k-i}\\
& =\sum_{k=1}^{|\cN|}\alpha_kM(k,j)=a_j.
\end{align*}
If $\xi(0)=1$, and $\hxi(\cN)(=n_0(0,\xi))=j$, then $\hxi(0)=0$. Use
the symmetry \eqref{c01symm} noted earlier, and the above with $\hxi$
in place of $\xi$ to get
\begin{align*}
c(0,\xi) &= c(0,\hxi)=a_j.
\end{align*}
Therefore the cancellative system corresponding to the above $\beta_0$
and $k_0$ is indeed the non-linear voter model determined by $(a_j)$.
\end{proof}

\begin{lemma} \label{l:qcanc} Assume $d\ge 1$, and $\cN$ is a general
neighbourhood with $2\le|\cN|\le 8$.\\ (a) For $2\le |\cN|\le 4$ and
all $q\in[0,1]$, the corresponding $q$-voter model on $\Z^d$ is
cancellative, and for $q\in[0,1)$, in \eqref{canc} we have
$\beta_0(A)>0$ for all $A\subset \bar\cN$ with $|A|=3$. \\ (b) For $q$
sufficiently close to $1$, the corresponding $q$-voter model on $\Z^d$
is cancellative for $q\le1$ and the last conclusion in (a) holds for
$q<1$.
\end{lemma}
\begin{proof}(a) 
Suppose first that
$|\cN|=4$.
it is
straightforward to check that
\[
\bM = \begin{bmatrix}
1&2&3&4\\
3&4&3&0\\
3&2&1&4\\
1&0&1&0
\end{bmatrix}
\quad\text{and}\quad
\bM^{-1} = \frac18\begin{bmatrix}
-2&0&2&4\\
0&2&0&-6\\
2&0&-2&4\\	
1&-1&1&-1
\end{bmatrix}.
\]
Thus, by \eqref{aalpha}, for $a_\ell = (\ell/|\cN|)^q$, $\balpha
= \ba \bM^{-1}$ is given 
by
\begin{align*}
\alpha_1(q) &= -\frac14\Bigl(\frac14\Bigr)^q + \frac14\Bigl(\frac34\Bigr)^q +
\frac18\\
\alpha_2(q) &= \frac14\Bigl(\frac12\Bigr)^q -\frac18\\
\alpha_3(q) &= \frac14\Bigl(\frac14\Bigr)^q-\frac14\Bigl(\frac34\Bigr)^q+\frac18\\
\alpha_4(q) &= \frac12\Bigl(\frac14\Bigr)^q - \frac34\Bigl(\frac12\Bigr)^q +
\frac12\Bigl(\frac34\Bigr)^q -\frac18.
\end{align*}
It is now an enjoyable calculus exercise to check that each
$\alpha_\ell(q)\ge0$ for all $0\le q\le 1$ and each $\alpha_\ell(q)>0$
for $0\le q<1$. By Proposition~\ref{p:nlvm} the cancellative property
holds for $q\in[0,1]$, and \eqref{beta0def}, together with
$\alpha_2,\alpha_3>0$, give the final conclusion for any $q\in[0,1)$.
The cases $|\cN|=2,3$ are handled the same way, with simpler
calculations. This prove (a).

(b) Consider $|\cN|=8$.  It is still easy to write down $\bM$ from
\eqref{Mdef}. Then \texttt{maple} can be used to find $\bM^{-1}$ (see
the Appendix), and \eqref{aalpha} can be used to write down each
$\alpha_\ell(q)$, $1\le \ell\le |\cN|$ explicitly.  With these
formulas it is easy to check that $\alpha_1(q) \ge 2^{-|\cN|+1}$ for
$0\le q\le 1$, and also that $\alpha_\ell(1) = 0 \text{ for }2\le
\ell\le |\cN|$.  Some simple, if lengthy, arithmetic shows that
\begin{equation}\label{cangoal} \alpha_\ell'(1)=
-2^{-k_\ell}\log(\frac{n_\ell}{m_\ell})<0 \text{ for }2\le \ell\le
|\cN|,
  \end{equation} where $k_\ell,m_\ell, n_\ell$ are positive integers
with $m_\ell<n_\ell$ and $k_\ell\le |\cN|+1$. All of these quantities
are given in the Appendix below.  It follows that each
$\alpha_\ell(q)$ must be strictly positive for $q<1$ sufficiently
close to 1.  Thus Proposition~\ref{p:nlvm} implies the cancellative
property for $q\le 1$ close to $1$, and \eqref{beta0def} implies that
$\beta_0(A)>0$ if $|A|=3$, $A\subset\bar\cN$ for $q<1$ close to $1$.

The cases $|\cN|=5,6,7$ are handled the same way. We omit the details. 
\end{proof}

\begin{conjecture}\label{conj:canc} All $q$-voter models are
cancellative for any $d\ge 1$, any neighbourhood $\cN$, $|\cN|\ge 2$,
and any $0\le q\le 1$.
\end{conjecture}

\subsection{Finite range voter model perturbations and examples}\label{sec:vmp} 
For $d\ge 2$, let the probability kernel $p:\Z^d\to[0,1]$, the local
frequencies of type $i$, $f_i(x,\xi)$, and the voter model rates
$c^\VM(x,\xi)$ be as in Section~\ref{sec:proofCCTintro}.  In
particular,
\begin{equation}\label{covp} \text{ $p$ has covariance matrix
$\sigma^2 I$ for some $\sigma>0$.}
\end{equation} Consider a neighbourhood $\cN$ containing the support
of $p$ and recall the definition of a voter model perturbation with
finite range in $\cN$, given in
Definition~\ref{def:fravmpintro}. Recall also from \eqref{rsdefnintro}
the asymptotic rate function $r_s(A)=g_1(1_A)\text{ for }A\subset
\cN$.

\medskip

It follows from \eqref{giconv} or \eqref{giconv2} that by decreasing
$\vep_0$, if necessary, we may assume that
\begin{equation}\label{gepsbnd}
\sup_{0<\vep\le\vep_0}\|g_0^\vep\|_\infty+\|g_1^\vep\|_\infty=C(g)<\infty.
\end{equation}

\begin{remark}\label{vpdefns}For $d\ge 3$ Proposition~1.1 of
\cite{CDP13} shows that in the finite range case (both voter model and
perturbation depend only on sites in $\cN$), our definition of finite
range voter model perturbation coincides with the voter perturbations
considered in \cite{CP14} for $d\ge 3$ (i.e. those satisfying
(1.10)-(1.15) of that reference). The notation here is a bit different
as what we call $g^\vep_i$ is denoted by $h_i^\vep$ in \cite{CDP13},
where a related quantity is called $g^\vep_i$ and is non-negative.
This non-negativity plays a role in the definition of the dual process
used to study $\xi^{[\vep]}$. We will only use the dual process
implicitly when quoting arguments from \cite{CP14} in
Section~\ref{sec:proofCCT} so there should be no confusion.

For $d=2$ to obtain a scaling limit theorem to super-Brownian motion
we have added the condition that ${\bf 1}$ is a trap. That there is
distinct behavior if ${\bf 1}$ is not a trap is demonstrated by
Theorem 1.3 of \cite{DLZ14} where for the $2$-dimensional contact
process with rapid voting (a voter model perturbation) it is shown
that the critical birth rate for survival must diverge to $+\infty$ as
the voter rate gets large.  We also have dropped the H\"older
convergence rate in \eqref{giconv} which entered in the pde analysis
used in earlier work but will not be needed here.  Indeed, the
H\"older convergence rate will not be satisfied by the $2$-dimensional
Lotka-Volterra models described below.
\end{remark}
In two dimensions, recall from Definition~\ref{def:asysymvmp} the
notion of a finite range voter model perturbation which is
asymptotically symmetric, and which will play an important role in our
scaling limit theorems.
\begin{remark} \label{rem:symm} Clearly $0-1$ symmetry of the voter
model perturbations $c_\vep$ (for all $\vep$) implies asymptotic
symmetry with $r^a=g^a\equiv 0$.  Note also that by \eqref{giconv} (or
\eqref{giconv2}) and \eqref{asymsym}, asymptotic symmetry of a voter
model perturbation implies
\begin{equation}\label{g0g1eq} g_1(\xi)=g_0(\hxi).
\end{equation}
\end{remark}
For $d\ge 3$ finite range voter model perturbations rescale to SBM if
time is rescaled by $1/\vep$ and space was scaled down by $\sqrt \vep$
(see \cite{CP05}). If $d=2$ these scaling parameters will change, as
already noted in \cite{CMP} in the special case of Lotka-Volterra
models.  Here the scaling parameter $N>e^3$ is the unique solution
$N=N(\vep)$ of
\begin{equation}\label{Ndefn}
\vep=\frac{(\log N)^3}{N},\ \vep\in(0,\vep_0].
\end{equation}
It is a calculus exercise to verify the existence and uniqueness of
such an $N$ for $\vep\in(0,1]\supset(0,\vep_0]$.  Henceforth for $d=2$
we usually consider $N$ as our fundamental parameter and set
$\vep=\vep_N=(\log N)^3/N$. The constraint $\vep\le \vep_0$ leads to
\begin{equation}\label{Nlbnd}N\ge N(\vep_0)>e^3,
\end{equation}
where the actual value of $N(\vep_0)$ is of little concern
as we will be interested in taking $N\to\infty$.

Proceeding now with $d=2$, our rescaled lattice is $S_N=\Z^2/\sqrt N$.
For $\xi\in\{0,1\}^{\Z^2}$, define the rescaled state
$\xi^{(N)}\in\{0,1\}^{S_N}$ by $\xi^{(N)}(x)=\xi(x\sqrt N)$. Let
$\{c_\vep$, $\vep\le \vep_0\}$, be an asymptotically symmetric finite
range voter model perturbation and consider the rescaled rate function
\begin{equation}\label{cNvepN}
c^N(x,\xi^{(N)})=Nc_{\vep_N}(x\sqrt N,\xi),\ x\in S_N,\ \xi\in\{0,1\}^{\Z^2}
\end{equation}
for the rescaled voter model perturbation process 
\begin{equation}\label{rescaledef}
\xi^N_t(x)=\xi^{[\vep_N]}_{tN}(x\sqrt N),\ x\in S_N.  
\end{equation}
For $x\in S_N,\, \text{and }\xi\in\{0,1\}^{\Z^2}$ introduce 
\[ c^{N,\VM}(x,\xi^{(N)})=c^{\VM}(x\sqrt N,\xi),
\]
\begin{equation}\label{cNsdef}
c^{N,s}(x,\xi^{(N)})=\hxi(x\sqrt N)g_0^{\vep_N}(\hat\xi|_{x\sqrt
N+\cN})+\xi(x\sqrt N)g^{\vep_N}_0(\xi|_{x\sqrt N+\cN}),
\end{equation}
and
\begin{equation}\label{cNadef}
c^{N,a}(x,\xi^{(N)})=\hxi(x\sqrt N)(\log
N)^2[g_1^{\vep_N}(\xi|_{x\sqrt N+\cN})-g_0^{\vep_N}(\hxi|_{x\sqrt
N+\cN})].
\end{equation}
Then for $x\in S_N$ and $\xi\in\{0,1\}^{\Z^2}$,
\begin{align}
\label{dNdecompsb} c^N(x,\xi^{(N)})&=Nc^{\VM}(x\sqrt N,\xi)+(\log
N)^3c^*_{\vep_N}(x\sqrt N,\xi)\\
\label{dNdecompsa}&=Nc^{N,\VM}(x,\xi^{(N)})+(\log
N)c^{N,a}(x,\xi^{(N)})+(\log N)^3c^{N,s}(x,\xi^{(N)}).
\end{align}
Clearly $c^{N,s}$ is symmetric, that is
\begin{equation}\label{cssymm}
c^{N,s}(x,\xi^{(N)})=c^{N,s}(x,\hat{\xi}^{(N)}),
\end{equation}
and a short calculation shows that
\begin{equation}\label{cNazero} \text{if
}c^*_\vep(x,\xi)=c^*_\vep(x,\hxi)\text{ for all }\vep, x,\xi,\text{
then }c^{N,a}\equiv 0\text{ for all }N.
\end{equation}
Moreover it is easy to check that the decomposition in
\eqref{dNdecompsa} uniquely determines $c^{N,s}$ and $c^{N,a}$ if we
assume that $c^{N,s}$ is symmetric and for every $x\in\SN$,
$c^{N,a}(x,\xi^{(N)})=0$ if $\xi^{(N)}(x) =1$.  If for $x\in\Z^2$ and
$\xi\in\{0,1\}^{\Z^2}$,
\[c^s(x,\xi)=\hxi(x)g_0(\hxi|_{x+\cN})+\xi(x)g_0(\xi|_{x+\cN}),\text{
and }c^a(x,\xi)=\hxi(x)g^a(\xi|_{x+\cN}),\] then $c^s$ is symmetric,
and from \eqref{cNsdef}, \eqref{cNadef}, \eqref{giconv2},
$\log(1/\vep_N)/\log N\to 1$, and \eqref{asymsym} we have
\begin{equation}\label{cNsconv}
\sup_{x\in\Z^2,\xi\in\{0,1\}^{\Z^2}}|c^{N,s}(x/\sqrt
N,\xi^{(N)})-c^s(x,\xi)|+|c^{N,a}(x/\sqrt N,\xi^{(N)})-c^a(x,\xi)|\to
0\text{ as }N\to\infty.
\end{equation}
Use the definition of $c^{N,a}$ and the convergence in \eqref{asymsym}
to see that if
\[r^{N,a}(A)=(\log N)^2[g^{\vep_N}_1(1_A)-g_0^{\vep_N}(1_{\cN\setminus
A})]\text{ for }A\subset\cN,
\]
then
\begin{equation}\label{cNarep}
c^{N,a}(x/\sqrt N,\xi^{(N)})=\hxi(x)\sum_{\emptyset\neq A\subset\cN}r^{N,a}(A)1\{\xi|_{x+\cN}=1_{x+A}\}
\end{equation}
and
\begin{equation}\label{rNalim}\lim_{N\to\infty}r^{N,a}(A)=g^a(1_A):=r^a(A) \text{ for all }A\subset\cN.
\end{equation}
The fact that ${\bf 0}$ and ${\bf 1}$ are traps
implies
$g_1^{\vep_N}(1_\emptyset)=g_0^{\vep_N}(1_\cN)=0$ and so we have dropped
the $A=\emptyset$ term in the sum.
Similarly if
\[r^{N,s}(A)=g_0^{\vep_N}(1_{\cN\setminus A})\text{ for }A\subset\cN,\]
then we have
\begin{equation}\label{cNsrep} c^{N,s}(x/\sqrt
N,\xi^{(N)})=\hxi(x)\sum_{\emptyset\neq A\subset
\cN}r^{N,s}(A)1\{\xi|_{x+\cN}=1_{x+A}\}+\xi(x)\sum_{\emptyset\neq
A\subset \cN}r^{N,s}(A)1\{\xi|_{x+\cN}=1_{x+\cN\setminus A}\}
\end{equation}
and (recall also \eqref{g0g1eq})
\begin{equation} \label{rNslim}
\lim_{N\to\infty}r^{N,s}(A)=g_0(1_{\cN\setminus A})=g_1(1_A):=r^s(A)\text{ for all }A\subset\cN.
\end{equation}
It follows from \eqref{gepsbnd}, \eqref{rNalim} and \eqref{rNslim} that
\begin{equation}\label{normr}
\Vert r\Vert:=\sup_{N\ge N(\vep_0),\ \emptyset\neq A\subset\cN}|r^{N,a}(A)|\vee|r^{N,s}(A)|<\infty.
\end{equation}

The following result will be useful for some comparison results in Section~\ref{ss:compproc}.
\begin{lemma}\label{psupp} For all non-empty $A\subset\cN$, $p(A)=0$
implies that $r^{N,s}(A)\ge 0$ for all $N$.
\end{lemma}
\begin{proof} Assume $p(A)=0$ for a fixed set $A$ as above.  We must show that 
\begin{equation}\label{lemgoal}
g_0^\vep(1_{\cN\setminus A})\ge 0\ \ \text{ for all $\vep\in[0,\vep_0]$}.
\end{equation}
Choose $\xi\in\{0,1\}^{\Z^2}$ such that $\xi(0)=1$ and $\xi|_{\cN}=1_{\cN\setminus A}$.  Then
\begin{align*}0\le c_\vep(0,\xi)&=f_0(0,\xi)+g_0^\vep(1_{\cN\setminus A})
=\sum_{y\in A}p(y)\hxi(y)+g_0^\vep(1_{\cN\setminus A})
=g_0^\vep(1_{\cN\setminus A}).
\end{align*}
The next to last equality holds because $\hxi(y)=0$ for $y\notin A$,
and the last equality holds by our assumption on $A$. This proves
\eqref{lemgoal}.
\end{proof} Recall the notation $\Theta^\pm(A)$ from \eqref{Theta+-}
and $K_2$ from Proposition~\ref{p:CRWnew}. We also recall $\Theta_3$
from \eqref{Thetadefnsintro} and $\Theta_2$ from
\eqref{Thetadefnsintro2}:
\begin{equation}\label{Thetadefns} \Theta_3=\sum_{\emptyset\neq
A\subset\cN}r^s(A)(\Theta^+(A)-\Theta^-(A))\quad\text{ and
}\quad\Theta_2=\sum_{\emptyset\neq
A\subset\cN}r^a(A)K_2(A,\bar\cN\setminus A).
\end{equation}
The following identities (see Remark~\ref{linearid} below) will
simplify $\Theta_3$ in some of the examples we now discuss:
\begin{equation}\label{linearid0}
\sum_{\emptyset\neq A\subset \cN}|A|(\Theta^+(A)-\Theta^-(A))=0,
\end{equation}
and
\begin{equation}\label{constantid0}
\sum_{\emptyset\neq A\subset
\cN}(\Theta^+(A)-\Theta^-(A))=\kappa=\lim_{t\to\infty}(\log t)^3\hat
P(|B^{\bar\cN}_t|=3)>0.
\end{equation}
 \medskip

\begin{example}\label{e:qv}(q-Voter Models) We begin with all $d\ge
2$, follow \cite{ASD} and consider, as in the Introduction, the
$q$-voter models with kernel $p(x)=1_\cN(x)/|\cN|$, for some
neighbourhood $\cN$.  We are interested in $q$ near 1, $q<1$, and so
let $q=1-\vep$ for $\vep\in(0,1]$ and define
\[
c^{*}_\vep(x,\xi) = \hxi(x)\dfrac{f_1(x,\xi)^{1-\vep}-f_1(x,\xi)}{\vep}
+ \xi(x)\dfrac{f_0(x,\xi)^{1-\vep}-f_0(x,\xi)}{\vep}.
\]
Then the $(1-\vep)$-voter rates may be written as
\begin{equation}\label{pert}
c^{(1-\vep)}(x,\xi) = c^\VM(x,\xi) + \vep \ c^{*}_\vep(x,\xi).
\end{equation}
Recall from \eqref{relldefn} that for $1\le \ell\le |\cN|$,
$r_\ell=(\ell/|\cN|)\log(|\cN|/\ell)$, and $\ r_0=0$, and let
\begin{equation}\label{rlep}
r^\vep_\ell = \frac{(\ell/|\cN|)^{1-\vep}-(\ell/|\cN|)}{\vep}, \text{ for }\ell\in\{0,\dots,|\cN|\}.
\end{equation}
Next, define
$g_i^\vep:\{0,1\}^{\cN}\to [0,\infty)$ for $i=0,1$ by 
\[g^\vep_i(\xi)= r^\vep_\ell\quad\text{ if }\quad\sum_{y\in\cN}1\{\xi(y)=i\}=\ell,\]
and $g_i$ in the same way, but with no superscript $\vep$'s.
Then we can write
\begin{align}\label{c*ep}c^{*}_\vep(x,\xi)
&= \sum_{\ell=1}^{|\cN|}r^\vep_l\Big(\hxi(x)1\{n_1(x,\xi)=\ell\} + \xi(x)1\{n_0(x,\xi)=\ell\}\Big)\\
\nonumber&=\hxi(x)g_1^\vep(\xi|_{x+\cN})+\xi(x)g_0^\vep(\xi|_{x+\cN}).
\end{align}
A Taylor series expansion shows that for $0< u<1$, and $0<\vep<1$,
\begin{align*}
\frac{u^{1-\vep}-u}{\vep} &=
u\log(1/u) + u\vep \sum_{m=2}^\infty \frac{\vep^{m-2}(\log
  1/u)^m}{m!}
\end{align*}
and thus
\begin{equation}\label{uepbnd}
0\le \frac{u^{1-\vep}-u}{\vep} -u\log(1/u)
\le u\vep \sum_{m=0}^\infty \frac{(\log
  1/u)^m}{m!}
= \vep.
\end{equation}
The above shows that for $\ell\in\{0,\dots,|\cN|\}$, 
\begin{equation}\label{rllim}
|r^\vep_\ell - r_\ell|\le \vep.
\end{equation}
and \eqref{rllim} implies
\[\Vert g_i^\vep-g_i\Vert_\infty\le \vep.\] As ${\bf 0}$ and ${\bf 1}$
are traps for $q$-voter models, we have verified that the collection
of $q$-voter models, $\{c^{(1-\vep)}:\vep\in(0,1]\}$ is a symmetric
voter model perturbation with finite range in $\cN$, non-negative
$g_i^\vep$, and $r_0=1,\, c_g=1$ in \eqref{giconv} (even if $d=2$).
As a final note here, using the fact that $\sup_{0<u\le 1}u\log(1/u) =
1/e$, it follows from \eqref{uepbnd} that if $0<\vep < 1-1/e$, then by
\eqref{rllim},
\begin{equation}\label{rmax} 0<r^\vep_\ell \le 1 \text{ for }1\le
\ell\le n-1.
\end{equation}

Consider now the rescaled quantities when $d=2$. We have from the above that
\begin{equation}\label{qvoterrNs}
r^{N,s}(A)=g_0^{\vep_N}(1_{\cN\setminus A})=r^{\vep_N}_{|A|}\to r_{|A|}=r^s(A),
\end{equation}
agreeing with \eqref{qvrs}.
The symmetry of $c^{(1-\vep)}$ puts us in the setting of Remark~\ref{rem:symm} and shows $c^{N,a}=c^a=r^{N,a}=r^a=0$ by \eqref{cNazero}. Therefore,
\begin{equation}\label{qvoterthetas}
\Theta_3=\Theta\text{ and }\Theta_2=0.
\end{equation}
The first was already noted in \eqref{Theta}, where the definition of $\Theta$ is given.
\end{example}

\begin{example}\label{e:LV} (Lotka-Volterra Models).  The
$2$-dimensional Lotka-Volterra model of Neuhauser-Pacala \cite{NP99}
with parameters $(\alpha_0,\alpha_1)\in(0,1)^2$ is the spin-flip
system with rate function (see \cite{CMP})
\[c^\LV(x,\xi)=c^\VM(x,\xi)+\hxi(x)(\alpha_0-1)f_1(x,\xi)^2+\xi(x)(\alpha_1-1)f_0(x,\xi)^2.\]
Here both the voter model kernel and the notation $f_i$ use the kernel
$p(x)=1_\cN(x)/|\cN|$ for some neighbourhood $\cN$.  If
$\alpha_1\vee\alpha_2\ge \frac{1}{2}$, $c^\LV$ is monotone (see
Section 1 of \cite{CP07}). In the diagonal case $\alpha_1=\alpha_2$,
$c^{\LV}$ is good cancellative (see, for example, the Proof of
Theorem~1.1 in Section~6 of \cite{CP14}).  As in Section~1 of
\cite{CMP}, for $\vep\in(0,1)$ one sets
\[\alpha_i=1-\vep+\beta_i^{(\vep)}(\log(1/\vep))^{-2}\vep,\text{ where
}\lim_{\vep\downarrow 0}\beta^{(\vep)}_i=\beta_i\in\R,\ \ i=0,1,\] so
that our voter model perturbation rates are given by
\begin{align*}c_\vep(x,\xi)&=c^{\VM}(x,\xi)+\vep
g_1^\vep(\xi\vert_{x+\cN})\hxi(x)+\vep g_0^\vep(\xi|_{x+\cN})\xi(x),\\
g_i^\vep(\xi|_\cN)&=\big(-1+(\log1/\vep)^{-2}\beta_{1-i}^{(\vep)}\big)f_i(0,\xi)^2,\
\ i=0,1.
\end{align*} If $g_i(\xi|_\cN)=-f_i(0,\xi)^2$, clearly
\[\Vert g_i^\vep-g_i\Vert_\infty\le C(\log(1/\vep))^{-2}\to 0\ \
\text{ for }i=0,1.\] Obviously both ${\bf 0}$ and ${\bf 1}$ are traps
for each $\vep$, and we also have
\begin{align*} \lim_{\vep\downarrow 0}(\log
(1/\vep))^2(g_1^\vep(\xi)-g_0^\vep(\hxi))&=\lim_{\vep\downarrow
0}(\beta_0^{(\vep)}-\beta_1^{(\vep)})f_1(0,\xi)^2\\
&=(\beta_0-\beta_1)f_1(0,\xi)^2\\ &:=g^a(\xi).
\end{align*} Therefore $\{c_\vep:\vep\in(0,1)\}$ is an asymptotically
symmetric finite range voter model perturbation.  We have
\[\beta_i^N:=\Bigl(\frac{\log
N}{\log(1/\vep_N)}\Bigr)^2\beta^{(\vep_N)}_i\to\beta_i\text{ as
}N\to\infty,\] and so, for $\emptyset\neq A\subset \cN$,
\begin{align}\nonumber r^{N,s}(A)=g_0^{\vep_N}(1_{\cN\setminus
A})&=(-1+(\log N)^{-2}\beta_1^N)\Bigl(\frac{|A|}{|\cN}\Bigr)^2\\
\label{rsLV}&\to-\Bigl(\frac{|A|}{|\cN|}\Bigr)^2=r^s(A),
\end{align} and
\begin{align*} r^{N,a}(A)&=(\log
N)^2[g^{\vep_N}_1(1_A)-g^{\vep_N}_0(1_{\cN\setminus A})]\\
&=(\beta^N_0-\beta^N_1)f_1(0,1_A)^2\\
&\to(\beta_0-\beta_1)\Bigl(\frac{|A|}{|\cN|}\Bigr)^2=r^a(A).
\end{align*} Therefore in this case we have
\begin{equation}\label{lvtheta2}
\Theta_2=\Theta_2^\LV:=(\beta_0-\beta_1)\sum_{\emptyset\neq
A\subset\cN}\Big(\frac{|A|}{|\cN|}\Big)^2 K_2(A,\bar\cN\setminus A),
\end{equation} and
\begin{equation}\label{lvtheta3}
\Theta_3=\Theta_3^\LV:=\sum_{\emptyset\neq
A\subset\cN}\Big(\frac{|A|}{|\cN|}\Big)^2(\Theta^-(A)-\Theta^+(A)).
\end{equation}
\end{example}

\begin{example}\label{e:av}
(Affine Voter Models) The 2-dimensional threshold voter model rate
function, introduced in \cite{CD91}, and corresponding to $q=0$ in
Example~\ref{e:qv}, is
\[ c^\TV(x,\xi) = 1\{\xi(x+y)\ne \xi(x) \text{ for some }y\in\cN\}
  \] for a neighbourhood $\cN$. The affine voter model with parameter
$\alpha\in[0,1]$, introduced in \cite{SS}, is the spin-flip system
with rate function
\[ c^\AV(x,\xi) = \alpha c^{\VM}(x,\xi) + (1-\alpha) c^{\TV}(x,\xi).
  \]
Here the voter model kernel is $p(x)=1_\cN(x)/|\cN|$, and
$\0$ and $\1$ are traps. The fact that $c^\TV$ is
cancellative was noted in Section~2
of \cite{CD91}. It is easy to check that a convex
combination of cancellative rate functions is cancellative,
and hence $c^\AV$ is cancellative, in fact good cancellative for
$\alpha<1$. 
Monotonicity of $c^\TV$, and hence of $c^\AV$ is clear.
Letting $\vep=1-\alpha$, 
\[
c^\AV(x,\xi) = c^\VM(x,\xi) + \vep (c^{\TV}(x,\xi)-c^\VM(x,\xi))
\]
so that our voter model perturbation rates are given by 
\begin{align*}c_\vep(x,\xi)&=c^{\VM}(x,\xi)+\vep
\big(\hxi(x)g_1^\vep(\xi\vert_{x+\cN})+\xi(x) g_0^\vep(\xi|_{x+\cN})\big),\\
g_i^\vep(\xi|_\cN)&=-f_i(0,\xi)+1\{n_i(0,\xi)\ge 1\}\ \ i=0,1.
\end{align*}
Since $g_i=g_i^\vep$ does not depend on $\vep$,
\eqref{giconv2} holds. 
Obviously each $c_\vep$ is symmetric, 
so $r^{N,a}=r^a=g^a\equiv 0$. Therefore $\{c_\vep:\vep\in(0,1)\}$ is
an asymptotically symmetric finite range voter model
perturbation such that for $A\ne\emptyset$, 
\begin{align*}
r^{N,s}(A)=r^s(A) = g_0(1_{\cN\setminus A}) = - \frac{|A|}{|\cN|}
+  1,
\end{align*}
$\Theta_2=0$, and if we use \eqref{linearid0} and then
\eqref{constantid0}, we get
\begin{equation}\label{avthetas}
\Theta_3=\sum_{\emptyset\ne A\subset\cN}
\Big(1-\frac{|A|}{|\cN|}\Big)(\Theta^+(A)-\Theta^-(A))=\sum_{\emptyset\ne
  A\subset\cN}(\Theta^+(A)-\Theta^-(A))=\kappa\ 
  :=\Theta^{\AV}_3.
\end{equation}
\end{example}

\begin{example}\label{e:GV} (Geometric Voter Models) The 2-dimensional
geometric voter model with rate function, introduced in \cite{CD91},
is
\begin{equation}\label{geom1} c^\GV(x,\xi) =
\dfrac{1-\theta^{j}}{1-\theta^{|\cN|}} \text{ if } \sum_{y\in
\cN}1\{\xi(x+y)\ne \xi(x)\}=j.
\end{equation} Here $\cN$ is a neighbourhood, $0\le \theta< 1$, and
$\0$ and $\1$ are clearly traps.  As $\theta$ ranges from 0 to 1,
these dynamics range from the threshold voter model to the voter
model. The fact that $c^\GV$ is cancellative was proved in Section~2
of \cite{CD91}, and it follows that $c^\GV$ is then good
cancellative. Monotonicity of $c^\GV$ is again elementary.  By (7.3)
in \cite{CP14}, taking $\vep=1-\theta$ (instead of $\vep^2$),
\begin{equation}\label{geom2}
c^\GV(x,\xi) = c^\VM(x,\xi) + \vep \frac{|\cN|}{2}
f_0(x,\xi)f_1(x,\xi) + O(\vep^2) \text{ as }\vep\to 0.
\end{equation}
Here, $c^\VM$ and the densities $f_i$ use the kernel
$p(x)=1_{\cN}(x)/|\cN|$, and the $O(\vep^2)$ term is uniform in
$x,\xi$. Thus, there are $g^\vep_i$ such that
\begin{align*}c_\vep(x,\xi)=c^{\VM}(x,\xi)+\vep
g_1^\vep(\xi\vert_{x+\cN})\hxi(x)+\vep
g_0^\vep(\xi|_{x+\cN})\xi(x),
\end{align*}
where if $g(x,\xi) = (|\cN|/2)f_0(x,\xi)f_1(x,\xi)$, then $\|g_i^\vep
- g\| \le C|\vep|$. Clearly $c^{\GV}$ is symmetric and so by
Remark~\ref{rem:symm} $g^a=r^a=0$.  Therefore $\{c_\vep,0\le \vep<1\}$
is an asymptotically symmetric voter model perturbation such that for
$A\ne \emptyset$,
\begin{equation*}
r^s(A) = g(1_{\cN\setminus A}) = \frac{|A| (|\cN|- 
  |A|)}{2|\cN|}. 
\end{equation*}
Thus $\Theta_2=0$ and, again using \eqref{linearid0}, we have
\begin{align}\label{gvthetas}
\Theta_3=\frac{1}{2|\cN|}\sum_{\emptyset\ne A\subset\cN}
|A|(|\cN|-|A|)(\Theta^+(A)-\Theta^-(A))&=\frac{|\cN|}{2}\sum_{\emptyset\ne
A\subset\cN}\Bigl(\frac{|A|}{|\cN|}\Bigr)^2(\Theta^-(A)-\Theta^+(A))\\
\nonumber&:=\Theta^\GV_3=\frac{|\cN|}{2}\Theta_3^\LV.
\end{align}
\end{example}

\subsection{Proof of Theorem~\ref{t:CCT} for dimensions 3
  and above}\label{sec:cct3d}
We assume $\cN$ is as in Theorem~\ref{t:CCT} for $d\ge3$.  Let
$\langle\quad\rangle_u$ denote expectation with respect to the voter
model equilibrium with density $u\in[0,1]$ and define symmetric rates
\begin{equation}\label{fdef}f(u)=\langle
(1-\xi(0))c^*(0,\xi)-\xi(0)c^*(0,\xi)\rangle_u\text { for }u\in[0,1],
\end{equation} where $c^*$ is as in \eqref{c*ep} but with $r_\ell$ in
place of $r_\ell^\vep$.  As noted in Section~1 of \cite{CDP13} $f$
will be a polynomial in $u$ of degree at most $|\cN|+1$.  As we have
$d\ge 3$, Theorem~1.2 of \cite{CP14} (as strengthened in
Remark~\ref{betainfdge3}) shows that the conclusion of
%\eqref{cctheorem} will hold with $\nu_{1/2}$ as in
Theorem~\ref{t:CCT} will hold for $q<1$ and sufficiently close to $1$,
providing the following hold for some $\vep_0\in(0,1)$:
\begin{align*} &(1)\text{ For $0<\vep<\vep_0$, $c^{(1-\vep)}(x,\xi)$
is a rate function of a cancellative process (as in \eqref{canc}).}\\
&(2)\ \{c^{(1-\vep)}(x,\xi):0<\vep<\vep_0\}\text{ is a finite range
voter model perturbation.}\\ &(3)\ f'(0)>0.
\end{align*} Here we have ignored one condition from Theorem~1.2 of
\cite{CP14} (condition (1.2) there) as the required exponential tail
bound is trivially true for our finite neighbourhood setting.  The
first two conditions have been established in Sections~\ref{sec:canc}
(Lemma~\ref{l:qcanc}) and \ref{sec:vmp}, respectively.  Finally (3)
follows from Theorem~1.2 of \cite{ASD}, where it is shown for any
neighbourhood $\cN$.  \qed

\begin{remark}\label{r:+drift} The proof of $f'(0)>0$ from \cite{ASD}
is stated for $d=3$ but holds equally well for $d\ge 3$.  In fact it
gives a stronger representation for $f$ which implies $f'(0)>0$.
Alternatively see Proposition~\ref{p:posthetabigd} for a more direct
proof.
\end{remark}

\section{A general complete convergence theorem in two dimensions}\label{sec:proofCCT}
Fix an initial state $\xi_0\in\{0,1\}^{\Z^2}$. For
$\xi,\eta\in\{0,1\}^{\Z^2}$.  We first extend the stochastic
differential equation (SDE) construction of spin-flip systems and
coupling with killed processes from Section 2 of \cite{CP07} to the
setting where $|\xi_0|$ may be infinite.  Let $\{N^{x,i}:x\in\Z^2,
i=0,1\}$ be independent Poisson point processes on $\R_+^2$ with rate
$ds\times du$.  For $R\subset\R^2$ and $T\ge 0$,
\begin{equation}\label{sigmafdefn} \cG([0,T]\times R) \text{ is the
$\sigma$-field generated by }\{N^{x,i}|_{[0,T]\times \R}:x\in
R,i=0,1\}.
\end{equation}
Consider a rate function $c:\Z^2\times\{0,1\}^{\Z^2}\to[0,\infty)$,
which is bounded continuous. Let $\hat c(x,\xi)=c(x,\hxi)$ be the rate
function for the evolution of the $0$'s. We assume
\begin{equation}
\label{L1boundunscaleda}
\sum_x\big(c(x,\xi)+\hat c(x,\xi)\big)\le C|\xi|,
\end{equation}
and
\begin{equation}
\label{pertboundunscaleda}
\sup_{x\in\Z^2}\sum_{u\in\Z^2}\sup_\xi|c(x,\xi)-c(x,\xi^{(u)})|\le C,
\end{equation}
where
\[\xi^{(u)}(x)=1(x\neq u)\xi(x)+1(x=u)\hxi(x).\]
Under \eqref{pertboundunscaleda} and boundedness and continuity of
$c$, there is a unique Feller process $\xi_t$ taking values in
$\{0,1\}^{\Z^2}$ associated with the rate function $c$ (see Theorem B3
in \cite{Lig99}). Note that the above conditions hold for $c$ iff they
hold for $\hat c$.
\begin{remark}\label{votepertlip} It is easy to check all of the above
conditions are satisfied by $c_\vep$, if $\{c_\vep:0<\vep\le \vep_0\}$
are the rates of a finite range voter model perturbation. See
Corollary~2.4 of \cite{CP07} for \eqref{L1boundunscaleda} (without the
$\hat c$ term) and \eqref{pertboundunscaleda}. Condition
\eqref{L1boundunscaleda} follows easily for $\hat c$ from the fact
that ${\bf 1}$ is a trap for our finite range voter model
perturbations in $d=2$ (just as it followed for $c$ using the fact
that ${\bf 0}$ is a trap). Note that under our boundedness assumption
on $c$, \eqref{L1boundunscaleda} for $c$ alone is equivalent to
condition (2.1) in \cite{CP07}.
\end{remark}
For $\xi_0\in\{0,1\}^{\Z^2}$ consider the SDE
\begin{multline}\label{construct2}
\xi_t(x) = \xi_0(x)+\int_0^t\int
(1-\xi_{s-}(x))1(u\le c(x,\xi_{s-}))N^{x,0}(ds,du) 
\\ -\int_0^t\int \xi_{s-}(x)1(u\le
c(x,\xi_{s-}))N^{x,1}(ds,du), \quad t\ge 0,x\in\Z^2.
\end{multline}
If $|\xi_0|<\infty$, \eqref{construct2} has a pathwise unique solution
$\xi_t=\xi_t[\xi_0]$ which has the same law as the above Feller
process with initial condition $\xi_0$. For this see Proposition
2.1(a) of \cite{CP07} and note, that monotonicity of $c$ is not needed
in Proposition~2.1(c).  Now fix $M_0\in\N$, set $I'=(-M_0,M_0)^2$, and
let $\underline{c}(x,\xi)=1(x\in I')c(x,\xi)$. Clearly $\underline{c}$
also satisfies all the hypotheses we have imposed on $c$. For our
given initial condition $\xi_0$, let $\uxi_0(x)=1(x\in
I')\xi_0(x)$. Then $|\uxi_0|<\infty$ (even if $|\xi_0|=\infty$) and so
there is a unique solution, $\uxi_t=\uxi_t[\uxi_0,I']$, of
\begin{multline}\label{killsconstruct}
\uxi_t(x) = \uxi_0(x)+\int_0^t\int
(1-\uxi_{s-}(x))1(u\le \underline{c}(x,\uxi_{s-}))N^{x,0}(ds,du) 
\\ -\int_0^t\int \uxi_{s-}(x)1(u\le
\underline{c}(x,\uxi_{s-}))N^{x,1}(ds,du), \quad t\ge 0,x\in\Z^2.
\end{multline}
Note that $\uxi_t(x)=0$ for all $t\ge 0$ and all $x\notin I'$. If, in
addition, $c$ is monotone, then by Proposition~2.1(b) of \cite{CP07},
\begin{equation}\label{monotoneorder}
\text{if $\xi_{\cdot}$ satisfies
}(\ref{construct2})\text{ where $|\xi_0|<\infty$, then
}\uxi_t\le \xi_t\ \forall t\ge 0. 
\end{equation}
We now show this continues to hold even if $|\xi_0|=\infty$. {\bf
Unless otherwise indicated, monotonicity of the rate function, $c$, is
assumed in the rest of this section.} Note that $\hat c$ is monotone
iff $c$ is. Therefore our hypotheses hold for $c$ iff they hold for
$\hat c$.
\begin{lemma}\label{l:orderinfinite} If $\xi_{\cdot}$
satisfies \eqref{construct2} then $\uxi_t\le \xi_t$ for all $t\ge 0$
a.s.
\end{lemma}
\begin{proof} By \eqref{monotoneorder}, we may focus on
$|\xi_0|=\infty$. Define
\[\Lambda_t=\sum_{x\in I'}\int_0^t\int_0^{\Vert
c\Vert_\infty}(N^{x,0}+N^{x,1})(ds,du).\] Then $\Lambda$ is a Poisson
process with rate $2|I'|\Vert c\Vert_\infty<\infty$ and so has a
sequence of jump times $0<T_1<T_2<\dots$ increasing to infinity.  It
suffices to prove that $\uxi_t\le \xi_t$ for $t\in[0,T_n]$ for all
$n$, which we prove by induction.  Assume the inequality up to, and
including, time $T_n$. On $(T_n,T_{n+1})$,
$\uxi_t(x)=\uxi_{T_n}(x)\le\xi_{T_n}(x)= \xi_t(x)$ for all $x\in I'$
since the jump times of these coordinates are clearly included in the
jump times of $\Lambda$.  For $x\notin I'$, $\uxi_t(x)=0\le \xi_t(x)$
for all $t$, including those in $(T_n,T_{n+1})$.  At $T_{n+1}$ one
considers the unique $x\in I'$ for which $N^{x,i}$ jumps at $T_{n+1}$
for some $i$ and uses the monotonicity of $c$ to show that in each of
the two cases $\Delta\xi_{T_{n+1}}(x)<0$ or
$\Delta\uxi_{T_{n+1}}(x)>0$, one has
$\uxi_{T_{n+1}}(x)\le\xi_{T_{n+1}}(x)$.  This is done just as in the
proof of Proposition~2.1(b) in \cite{CP07}.  The remaining cases
trivially lead to the same conclusion.  This establishes the induction
step, and the $n=1$ step is handled in exactly the same way.
\end{proof}

If $\xi_{\cdot}$ satisfies
\eqref{construct2}, then $\hxi_{\cdot}$ satisfies
\begin{multline}\label{hatconstruct}
\hxi_t(x) = \hxi_0(x)+\int_0^t\int
(1-{\hxi}_{s-}(x))1(u\le\hat c(x,\hxi_{s-}))N^{x,1}(ds,du) 
\\ -\int_0^t\int \hxi_{s-}(x)1(u\le
\hat c(x,\hxi_{s-}))N^{x,0}(ds,du), \quad t\ge 0,x\in\Z^2.
\end{multline}
Note that $c$ has been replaced with $\hat c$ and the roles of
$N^{x,0}$ and $N^{x,1}$ have been reversed from that in
\eqref{construct2}.  We set $\hat\uxi_0(x)=\widehat\xi_0(x)1(x\in I')$
and define $\hat\uxi_t=\hat\uxi_t[\hat\uxi_0,I']$ as the unique
solution (by Proposition~2.1(a) of \cite{CP07} because $\hat c$
satisfies the same hypotheses as $c$) to
\begin{multline}\label{killedhatconstruct}
\hat\uxi_t(x) = \hat\uxi_0(x)+\int_0^t\int
(1-{\hat\uxi}_{s-}(x))1(u\le \underline{\hat c}(x,\hat\uxi_{s-}))N^{x,1}(ds,du) 
\\ -\int_0^t\int \hat\uxi_{s-}(x)1(u\le
\underline{\hat c}(x,\hat\uxi_{s-}))N^{x,0}(ds,du), \quad t\ge 0,x\in\Z^2.
\end{multline}
So in the notation $\hat\uxi$ we effectively take the hat first and then do the killing. 

If $M\in\N$ and $\xi_0\in\{0,1\}^{\Z^2}$, we may set
$\uxi_0(x)=\xi^M_0(x):=\xi_0(x)1(x\in(-M,M)^2)$, and denote the unique
solution to \eqref{construct2} with this initial state by $\xi^M_t$.
\begin{proposition}\label{p:infinitesde}(a) As $M\to\infty$,
$\xi^M_t(x)\uparrow \xi^\infty_t(x)$ for all $x\in\Z^2$ and $t\ge 0$
a.s. Moreover $\xi^\infty_t$ is the unique in law Feller process with
rates $c(x,\xi)$ starting at $\xi_0$.  If $|\xi_0|<\infty$ then
$\xi^\infty_t=\xi_t[\xi_0]$, the unique solution of\eqref{construct2}.

\noindent(b) We have $\uxi_t(x)\le \xi^\infty_t(x)$, and
$\hat\uxi_t(x)\le \widehat{\xi^\infty_t}(x)$ for all $x\in\Z^2$
$t\ge 0$ a.s.

\noindent(c) If $c$ is symmetric, then $P(\hat\uxi\in
\cdot)=P_{\hat\uxi_0}(\uxi\in\cdot)$, where the right-hand side is the
law of $\uxi_\cdot$ with initial state $\hat\uxi_0$, that is, the law
of ${\uxi_\cdot}[\hat\uxi_0,I']$.
\end{proposition}
\begin{proof} (a) By Proposition~2.1(b) of \cite{CP07}, $\xi^M_t(x)$
increases in $M$ for all $(t,x)$ a.s. and so we can define
$\xi^\infty$ as this a.s. limit. By (c) of the same Proposition the
law of $\xi^M$ is that of the unique Feller process with rates $c$ and
initial condition $\xi_0^M$. Theorem~B3 of \cite{Lig99} allows us to
apply Theorem~5.2 of \cite{Lig85} to conclude that the martingale
problem associated with the rates $c$ is well-posed, and then
Proposition~6.5 of \cite{Lig85} gives continuity of the laws in the
initial condition. This implies that $\xi^\infty$ has the required
law.  If $|\xi_0|<\infty$ and $\xi_t=\xi_t[\xi_0]$ is the unique
solution of \eqref{construct2}, then monotonicity in the initial
condition from Proposition~2.1(b) of \cite{CP07} shows that
$\xi^M_t\le \xi_t$ for all $t\ge 0$ a.s., and so taking limits we get
$\xi^\infty_t\le\xi_t$ for all $t\ge 0$ a.s. Since $\xi_{\cdot}$ and
$\xi_{\cdot}^\infty$ have the same law, they must be identical.

\noindent(b)  Let $M\ge M_0$, so that
$\underline{\xi}^M_0:=\xi^M_0(x)1(x\in I')=\uxi_0(x)$ and therefore
$\uxi_t=\uxi^M_t:=\uxi_t(\uxi^M_0,I')$. By \eqref{monotoneorder},
w.p.$1$ for any $t\ge 0$,
\[\uxi_t=\uxi^M_t\le\xi^M_t\le\xi^\infty_t.\] If
$\hat\uxi^M_t=\hat\uxi_t(\hat\xi^M_0,I')$, then, just as above with
$\hat c$ in place of $c$, $\hat \uxi_t=\hat\uxi^M_t$. The process
$\hxi^M_t=1-\xi^M_t$ satisfies \eqref{hatconstruct}, and so we may
apply Lemma~\ref{l:orderinfinite} with $\hat c$ in place of $c$ and
the roles of $N^{x,0}$ and $N^{x,1}$ reversed and so conclude that
\[\hat\uxi_t=\hat\uxi^M_t\le 1-\xi^M_t\to \widehat{\xi_t^\infty}\text{
as }M\to\infty\quad\forall t\ge 0 \ \text{a.s.},\] and so deduce the
second inequality in (b).

\noindent(c) Under symmetry of $c$, $\hat\uxi$ is the unique solution
of \eqref{killedhatconstruct} with $\hat c=c$, and so has the same law
as $\uxi[\hat\uxi_0,I']$, the unique solution of
\eqref{killsconstruct} with initial condition $\hat\uxi_0$, because
$(N^{x,0},N^{x,1})$ is equal in law to $(N^{x,1},N^{x,0})$.
\end{proof}

In view of (a) of the above we will denote $\xi_t^\infty$ by
$\xi_t[\xi_0]$ as it agrees with our earlier definition for
$|\xi_0|<\infty$. The reader will note however, we have side-stepped
the general question of pathwise existence and uniqueness of solutions
to \eqref{construct2} when $|\xi_0|=\infty$. We believe this to be the
case by uniqueness of the martingale problem and monotonicity, but
will not need it.

(a) and (b) show that for any initial state $\xi_0$, and corresponding
$\uxi_0$ and $\hat\uxi_0$, we may construct $(\xi,\uxi,\hat\uxi)$ on
the same space such that
\begin{equation}\label{couplingforintro} \uxi_t\le \xi_t \text{ and }
\hat\uxi_t\le \hxi(t) \ \ \forall t\ge0 \ a.s.
\end{equation} Here $\uxi$ evolves according to rate $c$ with killing
outside $I'$, and $\hat\uxi$ evolves according to rate $\hat c$ with
the same killing, while $\xi$ and $\hxi$ evolve according the rates
$c$ and $\hat c$, respectively with no killing. Under symmetry, the
dynamics of the two killed processes are the same but of course the
initial conditions differ.  In the symmetric case this point seems to
be made implicitly in the proof of Theorem~1.1 in Section~6 of
\cite{CP14}, but perhaps warrants the explicit construction given
above.

Turning to the complete convergence theorem, we first present an
abstract complete convergence in two dimensions essentially taken from
\cite{CP14}.  If $A\subset\Z^2$, $x_0\in\Z^2$ and
$\xi\in\{0,1\}^{\Z^2}$, let
\[A(x_0,\xi)=\{y\in A:\xi(y)=1, \xi(y+x_0)=0\}.\] Monotonicity is not
required for our first abstract complete convergence theorem.
\begin{theorem}\label{t:CCT2dgen}
Assume for $0<\vep\le \vep_0$, $\xi^{[\vep]}$ is a cancellative finite
range voter model perturbation with rate function $ c_\vep(x,\xi)$.
Assume also that for each $\vep$,
\begin{align}\label{condA} \exists x_0\in\Z^2& \text{ so that if
}|\widehat{\xi^{[\vep]}_0}|=\infty\text{ then }
\lim_{K\to\infty}\sup_{A\subset\Z^2,|A|\ge
K}\lim_{t\to\infty}P_{\xi^{[\vep]}_0}(|\xi^{[\vep]}_t|>0,\
A(x_0,\xi^{[\vep]}_t)=\emptyset)=0,
\end{align}
and
\begin{equation}\label{condB}
\limsup_{t\to\infty}P_{\delta_0}(\xi^{[\vep]}_t(0)=1)>0.
\end{equation}
There is an $\vep_1>0$ such that for $\vep\in(0,\vep_1)$, there is a
translation invariant symmetric stationary distribution
${\nu}^{\vep}_{1/2}$ with density $1/2$, satisfying the coexistence
property, such that for all initial $\xi^{[\vep]}_0$,
\begin{equation}\label{cct2dgen} \xi^{[\vep]}_t\Rightarrow
\beta_0(\xi^{[\vep]}_0)\delta_{\0}+\beta_\infty(\xi^{[\vep]}_0){\nu}^\vep_{1/2}+\beta_1(\xi^{[\vep]}_0)\delta_{\1}\quad\text{
as }t\to\infty.
\end{equation}
\end{theorem}
\begin{proof} We first show that for $\vep<\vep_0$, $\xi^{[\vep]}$ is
a good cancellative process (as defined in
Section~\ref{sec:canc}). Assume not. By Remark~\ref{rem:votercanc}, up
to a constant time change, $\xi^{[\vep]}$ is a voter model with kernel
$q_0(y)=\beta_0(\{y\})$.  Therefore starting at
$\xi^{[\vep]}_0=\delta_0$, $\xi^{[\vep]}_t=\0$ for $t$ large
a.s. (e.g. see Proposition~V.4.1(b) of \cite{Lig85}), and this
contradicts \eqref{condB}, completing the proof.

Under the above hypotheses, Remark~4, Corollary~3.3, and Lemma~4.2 of
\cite{CP14} show that for some $\vep_1>0$, the hypotheses of
Proposition~4.1 of \cite{CP14} hold for $\vep\in(0,\vep_1)$. (Note
that the {\it good} cancellative property is needed to apply
Corollary~3.3.) That Proposition implies the stationary measure
$\nu^\vep_{1/2}$ (from Section~\ref{sec:canc}) satisfies the
coexistence property and
\begin{equation}\label{wcct}\text{ if
}|\widehat{\xi^{[\vep]}_0}|=\infty,\text{ then
}\xi^{[\vep]}_t\Rightarrow
\beta_0(\xi^{[\vep]}_0)\delta_{\0}+\beta_\infty(\xi^{[\vep]}_0){\nu}^\vep_{1/2}+\beta_1(\xi^{[\vep]}_0)\delta_{\1}\text{
as }t\to\infty,
\end{equation}
where we have also used \eqref{betavalues}.  The fact that $\0$ is a
trap for $\xi^{[\vep]}$ (by \eqref{zerotrap}) implies $\xi^{[\vep]}$
is symmetric by the equivalence of \eqref{0trap} and \eqref{c01symm}
noted in Section~\ref{sec:canc}.  Therefore $ \nu^\vep_{1/2}$ is also
a symmetric law ($\nu^\vep_{1/2}(\hxi\in\cdot)=\nu^\vep_{1/2}(\cdot)$)
by the above convergence. By symmetry the conclusion of \eqref{wcct}
also holds if $|\xi_0^{[\vep]}|=\infty$ and so the proof of
\eqref{cct2dgen} is complete.
\end{proof}
\begin{remark} \label{rem:sym} As noted in the above proof, a
cancellative finite range voter model perturbation is symmetric.
\end{remark}

The key condition in the above is \eqref{condA} which will imply that
a pair of nearby sites with opposite type can be found in sufficiently
large sets for large $t$.  Such pairs are to be expected if there is
to be complete convergence with coexistence.

Assume now that for some $\vep_0>0$:
\begin{equation}\label{Hyp1}
\text{For $0<\vep\le\vep_0$, $\xi^{[\vep]}$ is a  cancellative and monotone finite range voter model perturbation in $\Z^2$.}
\end{equation}
By Remarks~\ref{votepertlip} and \ref{rem:sym} we may apply
Proposition~\ref{p:infinitesde}, and so for any initial condition
$\xi_0\in\{0,1\}^{\Z^2}$ construct $\xi^{[\vep]}$,
${\underline\xi}^{[\vep]}$, $\underline{\hat\xi}^{[\vep]}$ as
solutions of \eqref{construct2}, \eqref{killsconstruct} and
\eqref{killedhatconstruct}, respectively, with $\xi^{[\vep]}_0=\xi_0$,
all on a common probability space such that
\begin{equation}\label{xiveporder} 
{\underline\xi}^{[\vep]}_t \le
\xi^{[\vep]}_t \text{ and }
\underline{\hat\xi}^{[\vep]}_t\le \widehat{\xi^{[\vep]}_t}  \
\ \forall t\ge0 \ a.s.
\end{equation}
We now assume that $M_0=KL$ for natural numbers $K,L$ chosen below.
As in Section~\ref{sec:vmp} we often use $N\ge N(\vep_0)>e^3$
satisfying $\vep=\vep_N:=\frac{(\log N)^3}{N}$ as our fundamental
parameter. Let
\begin{equation}\label{ulinexiNdef}
\underline{\xi}^N_t(x)=\underline{\xi}^{[\vep_N]}_{Nt}(\sqrt Nx),\
x\in\SN,
  \end{equation}
and
\begin{equation}\label{killedmvd}
\underline{X}^N_t=\frac{\log N}{N}\sum_{x\in\SN}\underline{\xi}^N_t(x)\delta_{x},\text{ and }
\underline{\hat X}^N_t=\frac{\log
  N}{N}\sum_{x\in\SN}\hat{\underline\xi}^{N}_{t}(x)\delta_{x}.
\end{equation}
The next condition is the key to ensure the survival of our oriented
percolation process.  It specifies the values of $K,L$ which are used
above to define our killed particle systems through $M_0=KL$. Note
that ${\underline X}^N_t(\1)={\underline X}^N_t((-M_0,M_0)^2)<\infty.$
\begin{align}\label{Hyp2} \nonumber&\text{There are $T'>1$,
$K,J'\in\N$ with $K>2$, and $L'>3$, %and $\vep_1\in(0,\vep_0]$ so that
if }\\ &0<\vep\le\vep_0,\text{ and } I_{\pm
e_i}=\pm2L'e_i+[-L'+1,L'-1]^2,\text{ then for }
L=\lfloor\sqrt{N}L'\rfloor,\\
\nonumber&\underline{X}^N_0([-L',L']^2)\ge J'\text{ implies }
P(\underline{ X}_{T'}^N(I_{e})\ge J'\text{ for all } e\in\{\pm
e_i,i=1,2\})\ge 1-6^{-5(2K+1)^3}.
\end{align}
  
Recall that if the conclusion of Theorem~\ref{t:CCT2dgen} holds we say
for $0<\vep< \vep_1$ the complete convergence theorem with coexistence
(CCT) holds for $\xi^{[\vep]}$. We use this terminology going forward.
\begin{theorem}\label{t:genCCT} Assume for $0<\vep\le\vep_0$,
$\xi^{[\vep]}$ is a cancellative and monotone finite range voter model
perturbation in $\Z^2$ satisfying \eqref{Hyp2}.  There is an
$\vep_1>0$ such that for $\vep\in(0,\vep_1)$, the complete convergence
theorem with coexistence (CCT) holds for $\xi^{[\vep]}$.
\end{theorem}
\begin{proof}
We follow the general approach used in Section~6 of \cite{CP14} for
the $d=2$ Lotka-Volterra model.  From Theorem~\ref{t:CCT2dgen} it
suffices to establish \eqref{condA} and \eqref{condB}, as well as the
fact that $\beta_{\infty}(\xi^{[\vep]}_0)>0$ for initial conditions
distinct from ${\bf 0}$ and ${\bf 1}$ (we have suppressed the
dependence on $\vep$ in $\beta_\infty$). Remark~\ref{rem:sym} shows
that $\xi^{[\vep]}$ is symmetric. Therefore, we may use
Proposition~\ref{p:infinitesde}(c) to see that \eqref{Hyp2} implies
that,
\begin{align}\label{Hyp2b}&\underline{X}^N_0([-L',L']^2)\ge J'\text{
and }\underline{\hat{ X}}^N_0([-L',L']^2)\ge J'\text{ imply }\\
\nonumber&\qquad P(\underline{ X}_{T'}^N(I_{\pm e_i})\ge J'\text{ and
}\underline{\hat{ X}}_{T'}^N(I_{\pm e_i})\ge J'\text{ for }i=1,2)\ge
1-2\cdot6^{-5(2K+1)^3}.
\end{align}
To undo the scaling, recall $L=\lfloor\sqrt{N}L'\rfloor$ and let
$J=\frac{N}{\log N}J'$, $T=NT'$ and set $\tilde I_{\pm e_i}=\pm
2Le_i+[-L,L]^2$, $i=1,2$. In order to use Theorem~4.3 of \cite{Dur95}
we introduce a set $H\subset\{0,1\}^{\Z^2}$ of ``happy" configurations
and a good event $G_{\xi_0}=G_{\xi_0}(\vep)$ in our probability space
for each initial condition $\xi_0$.  We let
\begin{equation}\label{Hdefn}
H=\{\xi\in\{0,1\}^{\Z^2}:\, \xi([-L,L]^2)\ge J\text{ and }\hxi([-L,L]^2)\ge J\},
\end{equation}
and
\begin{equation}\label{Gdefn}
G_{\xi_0}=\{{\underline{\xi}}^{[\vep]}_T(\tilde I_{\pm e_i})\ge J,\
\hat{\underline{\xi}}^{[\vep]}_T(\tilde I_{\pm e_i})\ge J\text{ for
}i=1,2\},
\end{equation}
where $\xi_0$ is the initial condition for $\xi^{[\vep]}$, so that
$\underline\xi^{[\vep]}_0(x)=1_{I'}(x)\xi_0(x)$ and
$\hat{{\underline\xi}}^{[\vep]}_0(x)=1_{I'}(x)\hxi_0(x)$. Recall here
that $I'=(-KL,KL)^2$. For $z\in\Z^2$, $\sigma_z:\{0,1\}^{\Z^2}\to
\{0,1\}^{\Z^2}$ is the translation map,
$\sigma_z(\xi)(x)=\sigma(x+z)$.  Note that:

\noindent\phantom{\qquad}\text{(i) $G_{\xi_0}$ is
$\cG(I'\times[0,T])$-measurable for each $\xi_0$.}\\
\phantom{\qquad}\text{(ii)} If $\xi_0\in H$, then on $G_{\xi_0}$,
$\xi^{[\vep]}_T\in\sigma_{2Le}(H)$ for all $e\in\{\pm e_1,\pm
e_2\}$.\\ \phantom{\qquad}\text{(iii)} For any $\xi_0\in H$,
$P(G_{\xi_0})\ge 1-2\cdot6^{-5(2K+1)^3}:=1-\gamma'$.

\noindent Properties (i) and (ii) are clear from the definitions and
the orderings in \eqref{xiveporder}.  Property (iii) follows from
\eqref{Hyp2b}, along with a bit of arithmetic on the rescaled
intervals to show $\sqrt N I_{\pm e_i}\subset \tilde I_{\pm e_i}$,
where $N\ge e^3$ is used. Finally another bit of arithmetic shows that
$(1-\gamma')^{1/(2K+1)^3}>1-6^{-4}$, ensuring that (5.17) of
\cite{CP14} is valid. In this way we have established the set-up of
Lemma~5.2 of \cite{CP14} and we can invoke the comparison with
$2K$-dependent percolation from Theorem~4.3 of \cite{Dur95}, as
carried out in Section~5 of \cite{CP14}.  In particular, we may use
the proof of Lemma~5.3 of \cite{CP14} to conclude that \eqref{condA}
and \eqref{condB} hold for $\xi^{[\vep]}$.  Although the hypotheses of
that result require $d\ge 3$ and $f'(0)>0$ for a solution to a
reaction diffusion equation (which is not defined in $d=2$), those
hypotheses are only used to establish the set-up in Lemma~5.2 of
\cite{CP14}, which we have just verified directly, essentially using
\eqref{Hyp2}. The rest of the proof of Lemma~5.3 of \cite{CP14} only
requires arguments for general voter model perturbations and, in
particular, uses its branching coalescing dual from Section~2 of
\cite{CDP13} to bound some probabilities involving $\xi^{[\vep]}$.  In
this way the proof of Lemma~5.3 of \cite{CP14} gives us \eqref{condA},
and the proof also gives (this is (5.31) of \cite{CP14})
\begin{equation}\label{old5.31}
\inf_{\xi\neq 0}P_\xi(\xi^{[\vep]}_t\neq {\bf 0}\ \forall\, t\ge 0)\ge \rho,
\end{equation}
for some explicit $\rho>0$.  After the  proof of Lemma~5.3
in \cite{CP14}, \eqref{condB} is derived from
\eqref{old5.31} using the above oriented percolation setting
and elementary properties of voter model perturbations,
which apply equally well in our setting.  

To prove the last assertion on $\beta_\infty$, by \eqref{betavalues}
it suffices to consider $0<|\xi^{[\vep]}_0|<\infty$ or
$0<|\widehat{\xi_0^{[\vep]}}|<\infty$, and by the $0-1$-symmetry
(Remark~\ref{rem:sym}) we need only consider the first case. By
monotonicity and translation invariance we can take
$\xi^{[\vep]}_0=\delta_0$. The fact that ${\bf 0}$ is a trap implies
$P_{\delta_0}(|\xi_t^{[\vep]}|>0)$ is non-increasing in $t$.
Therefore \eqref{condB} easily implies $P_{\delta_0}(\tau_{{\bf
0}}=\infty)>0$.  But $P_{\delta_0}(\tau_{{\bf 1}}=\infty)=1$ by
\eqref{betavalues}, so we conclude $\beta_\infty(\xi^{[\vep]}_0)>0$.
\end{proof}
\begin{remark}\label{CCTdge3gen} The final paragraph in the above
proof applies equally well in $d\ge 3$ to show that the conclusion of
Theorem~1.2 of \cite{CP14} may be strengthened to include
$\beta_\infty(\xi^{[\vep]}_0)>0$ if $\xi^{[\vep]}_0$ is not ${\bf 0}$
or ${\bf 1}$.
\end{remark}
\begin{remark}\label{genvpsec5} Theorems \ref{t:CCT2dgen} and
\ref{t:genCCT} hold without the finite range assumptions. That is, we
only require that $\{\xi^{[\vep]},0<\vep\le \vep_0\}$ is a
$2$-dimensional voter model perturbation in the sense of
(1.10)--(1.15) of \cite{CP14} (where the H\"older rate of convergence
in (1.14) of \cite{CP14} (see \eqref{giconv}) is also weakened to
\eqref{giconv2}). Indeed our proofs, and those quoted in \cite{CP14},
only require these conditions.  To verify the key condition
\eqref{Hyp2}, however, we will need to work with finite range voter
model perturbations and make a critical assumption on the parameter
$\Theta_2+\Theta_3$ from \eqref{Thetadefns}.
\end{remark}
\begin{theorem}\label{t:ThetaOP} Assume $\{\xi^{[\vep]}: 0<\vep\le
\vep_0\}$ is a monotone, asymptotically symmetric finite range voter
model perturbation in $\Z^2$ with $\Theta_2+\Theta_3>0$. Then
\eqref{Hyp2} holds, perhaps with a smaller choice of $\vep_0>0$.
\end{theorem} We give the proof in Section~\ref{sec:percsetup}. It
will follow from Theorem~\ref{t:SBMgenintro}, our weak convergence
result to super-Brownian motion with drift $\Theta_2+\Theta_3$.

We are ready for the proof of our main result,
Theorem~\ref{t:vgen2dCCT}, a general complete convergence theorem for
monotone cancellative finite range voter model perturbations.
\medskip

\noindent{\it Proof of Theorem~\ref{t:vgen2dCCT}.} By
Theorem~\ref{t:genCCT} it suffices to establish \eqref{Hyp2}, perhaps
with a smaller $\vep_0>0$. By Remark~\ref{rem:sym}, $\xi^{[\vep]}$ is
symmetric and so in particular is asymptotically symmetric by
Remark~\ref{rem:symm}. The latter Remark also shows that $r^a=0$ and
therefore $\Theta_2=0$.  Hence $\Theta_3+\Theta_2=\Theta_3>0$ (by
hypothesis), and so Theorem~\ref{t:ThetaOP} gives \eqref{Hyp2}, and we
are done.  \qed

\medskip

To apply Theorem~\ref{t:vgen2dCCT} it would be useful to have a
general, and checkable, sufficient condition for $\Theta_3>0$, which
would also apply to the $q$-voter model, and so establish
Theorem~\ref{t:CCT} for $d=2$ as a special case. This is the goal of
the next section.

\section{Positivity of the drift and the Proof of Theorem~\ref{t:CCT}
for \textit{d}=2}\label{sec:posdrift} To establish a sufficient condition for
$\Theta_3>0$ in Theorem~\ref{t:vgen2dCCT}, it will be convenient to
first work in a more general setting with any {\it general}
neighbourhood $\cN$ (recall from Section~\ref{sec:canc} that $\cN$ is
finite non-empty subset of $Z^d\setminus\{0\}$) and $d\ge 2$.  We
consider a strictly subadditive map $r:\{A:A\subset\cN\}\to\R$. This
means that
\begin{equation}\label{rsubaddset} r(A\cup B)<r(A)+r(B)\text{ for all
non-empty disjoint }A,B\subset\cN.
\end{equation} By induction on $n$ this implies that for non empty
disjoint sets $A_1,\dots,A_n$ in $\cN$,
\begin{equation} \label{rsubadd2} r(\cup_{i=1}^nA_i)\le \sum_{i=1}^n
r(A_i),\ \ \text{where strict inequality holds if $n>1$}.
\end{equation}

Later we will want to consider a finite range voter model perturbation
and take $r=r^s$.  Assume for now that $d=2$.  To motivate the above
definition recall from \eqref{qvoterrNs} that for the $q$-voter model
in $d=2$ we have $r^s(A)=r_{|A|}$, where $r_\ell$ are as in
\eqref{relldefn}. We saw in Section~\ref{sec:proofCCTintro} (recall
\eqref{rsubaddintro}) that
\begin{equation}r_{\ell_1+\ell_2}< r_{\ell_1}+r_{\ell_2}\text{ for
$0<\ell_i$, and $\ell_1+\ell_2\le |\cN|$}.\label{rsubadd}
\end{equation} and so
\begin{equation}\label{rsqvsub} \text{for the $2$-dimensional
$q$-voter model, $r^s$ is strictly subadditive}.
\end{equation}

Return now to our earlier setting with $d\ge 2$ and general $\cN$.  If
$\pi$ is a partition of $\bar\cN$, $[0]$ denotes the cell of $\pi$
containing $0$ and $|\pi|$ is the cardinality of $\pi$. We assume
\eqref{rsubaddset} and that all sets in a partition are non-empty
throughout this section.
\begin{lemma}\label{l:detpi} If $\pi$ is a fixed partition of $\bar\cN$, then
\begin{equation}\label{piineq1} \sum_{\emptyset\neq
A\subset\cN}r(A)1(A\in\pi)\ge\sum_{\emptyset\neq
A\subset\cN}r(A)1(\bar \cN\setminus A=[0]).
\end{equation}
The inequality is strict if $|\pi|>2$.
\end{lemma}
\begin{proof} The left-hand side of \eqref{piineq1} trivially is
\[1(\bar\cN\setminus[0]\neq\emptyset)\sum_{\emptyset\neq A}r(A)1(A\in\pi,A\subset\bar\cN\setminus[0]),\]
while the right-hand side equals
\[r(|\bar\cN\setminus[0])1(\bar\cN\setminus[0]\neq\emptyset).\]
So to prove the result we may assume $\bar\cN\setminus [0]\neq\emptyset$, or equivalently $|\pi|>1$. Using
\[\bar\cN\setminus[0]=\cup_{A\in \pi,A\subset\bar\cN\setminus[0]} A,\]
and the subadditivity \eqref{rsubadd2}, we have
\[r(\bar\cN\setminus[0])\le \sum_{\emptyset\neq A}r(A)1(A\in\pi,A\subset\bar\cN\setminus[0]),\]
thus giving \eqref{piineq1}. If $|\pi|= 2$ there is equality in the above, 
and by \eqref{rsubadd2} there is strict inequality if $|\pi|>2$.
\end{proof}
As an immediate consequence we have: 
\begin{lemma}\label{l:genpostheta} If $\pi$ is a random
partition of $\bar\cN$ such that 
\begin{equation}\label{picond}P(|\pi|>2)>0,
\end{equation}
then
\begin{equation}\label{piineq} E\Bigl(\sum_{\emptyset\neq
A\subset\cN}r(A)1(A\in\pi)\Bigr)>E\Bigl(\sum_{\emptyset\neq
A\subset\cN}r(A)1(\bar\cN\setminus A=[0])\Bigr).
\end{equation}
\end{lemma}
Returning to our general $r$ in $d=2$, and recalling the definitions
of $\Theta^{\pm}(A)$ from \eqref{Theta+-}, we define
$\Theta_3=\Theta_3(r)$ by
\begin{equation}\label{Thetardef} \Theta_3=\sum_{\emptyset\neq
A\subset\cN} r(A)\Theta^{+}(A)-\sum_{\emptyset\neq A\subset\cN}
r(A)\Theta^{-}(A):=\Theta^+_3-\Theta^-_3.
\end{equation}
Note that this agrees with our earlier definition of $\Theta_3$ in
\eqref{Thetadefns} if $r=r^s$ for the finite range voter perturbations
in Section~\ref{sec:vmp}.  To use the above to show the positivity of
$\Theta_3$ in \eqref{Thetardef}, recall $K_3(A_1,A_2,A_3)$ from
Proposition~\ref{p:CRWnew} and the notation $\cP_k(\Gamma)$ and
$\cP(\Gamma)$ from Section~\ref{sec:coalrw}.  Recall also that
\eqref{rsubaddset} is still in force and the $q$-voter drift $\Theta$
in \eqref{Theta} corresponds to the special case $r(A)=r_{|A|}$ with
$r_\ell$ as in \eqref{relldefn}.

\begin{corollary} \label{postheta} If $d=2$, then
$\Theta_3>0$, and, in particular, $\Theta$ in \eqref{Theta}
is also strictly positive.
\end{corollary}
\begin{proof} Let
$\kappa=\sum_{\{A_1,A_2,A_3\}\in\cP_3(\bar\cN)}K_3(A_1,A_2,A_3)>0$,
where the sum is over sets, not ordered triples, and the
positivity is clear by $|\bar\cN|\ge 5$ (see \eqref{Nprop}).
Define a random partition in $\cP_3(\bar\cN)$ by
\[P(\pi=\{A_1,A_2,A_3\})=K_3(A_1,A_2,A_3)/\kappa.\] Both are
well-defined by the symmetry of $K_3$, and \eqref{picond}
holds because $|\pi|=3$ a.s.  If $A$ is a non-empty subset
of $\bar\cN$, then
\begin{equation}\label{pimarg}P(A\in\pi)=\sum_{\{A_1,A_2\}\in\cP(\bar\cN\setminus
A)}P(\pi=\{A,A_1,A_2\})=\sum_{\{A_1,A_2\}\in\cP(\bar\cN\setminus
A)}K_3(A,A_1,A_2)/\kappa.
\end{equation} Therefore from \eqref{Thetardef}
\begin{align}\label{thetaplusform}
\Theta_3^+:=\sum_{\emptyset\neq
A\subset\cN}r(A)\Theta^+(A)=\kappa\sum_{\emptyset\neq
A\subset\cN}r(A)P(A\in\pi) =\kappa
E\Bigl(\sum_{\emptyset\neq A\subset\cN}
r(A)1(A\in\pi)\Bigr).
\end{align} Similarly, apply \eqref{pimarg} with
$\bar\cN\setminus A$ in place of $A$ to see that
\begin{equation*} \nonumber\Theta_3^-:=\sum_{\emptyset\neq
A\subset\cN}r(A)\Theta^-(A)=\kappa\sum_{\emptyset\neq
A\subset\cN}r(A)P(\bar\cN\setminus A\in\pi) =\kappa
E\Bigl(\sum_{\emptyset\neq A\subset\cN}
r(A)1(\bar\cN\setminus A\in\pi)\Bigr).
\end{equation*} Note that for $A\subset \cN$ we have
$0\in\bar\cN\setminus A$, and so $\bar\cN\setminus A\in\pi$
iff $[0]=\bar\cN\setminus A$.  This shows that the above
implies
\begin{equation}\label{thetaminusform} \Theta_3^-=\kappa
E\Bigl(\sum_{\emptyset\neq A\subset\cN}
r(A)1(\bar\cN\setminus A=[0])\Bigr).
\end{equation} So by Lemma~\ref{l:genpostheta},
\eqref{thetaplusform} and \eqref{thetaminusform} we have
$\Theta_3^+>\Theta_3^-$, and therefore, $\Theta_3>0$. As
noted above, the positivity of $\Theta$ in \eqref{Theta}
follows from the special case $r(A)=r_{|A|}$.
\end{proof}

\begin{remark}\label{linearid} If $r(A)=|A|$ (not strictly
subadditive!) one easily sees that equality holds in
Lemmas~\ref{l:detpi} and \ref{l:genpostheta}, the latter
without any condition on $\pi$. The above proof then shows
that
\begin{equation}\label{linearr} \sum_{\emptyset\neq A\subset
\cN}|A|(\Theta^+(A)-\Theta^-(A))=0.
\end{equation} This identity simplified some of the formulae
for $\Theta_3$ in the examples of Section~\ref{sec:vmp}.
If $r(A)=1$ for all $A\subset\cN$ (which is strictly
subadditive), then from \eqref{thetaplusform} we have
$\Theta_3^+=2\kappa$ because there are exactly two subsets
of $\cN$ in $\pi$ corresponding to the two sets in $\pi$
other than $[0]$.  Similarly from \eqref{thetaminusform} we
get that $\Theta_3^-=\kappa$ because there is exactly one
subset of $\cN$ whose complement in $\bar\cN$ is $[0]$,
namely $\bar\cN\setminus[0]$.  %the complement of the set in
$\pi$ containing $0$.  Therefore
\begin{equation}\label{constantid} \sum_{\emptyset\neq
A\subset
\cN}(\Theta^+(A)-\Theta^-(A))=\kappa=\sum_{\{A_1,A_2,A_3\}\in\cP_3(\bar\cN)}K_3(A_1,A_2,A_3)=\lim_{t\to\infty}(\log
t)^3\hat P(|B^{\bar\cN}_t|=3)>0.
\end{equation}
The last equality follows easily from the definition of
$K_3$ by decomposing $\hat P(|B^{\bar\cN}_t|=3)$ into the
possible partitions induced by the sets of sites which have
coalesced at time $t$ and taking limits.
\end{remark}

The following corollary is immediate from
Theorem~\ref{t:vgen2dCCT} and Corollary~\ref{postheta}. It
represents our simplest criteria for a CCT to hold in $d=2$.
\begin{corollary}\label{c:CCTsubadd} Assume $d=2$,
$\{\xi^{[\vep]}:0<\vep\le \vep_0\}$ satisfies \eqref{Hyp1}
and $r^s$, given by \eqref{rsdefnintro}, is strictly
subadditive (i.e., \eqref{rsubaddset} holds for
$r=r^s$). Then $\Theta_3>0$ and there is an $\vep_1>0$ such
that for $\vep\in(0,\vep_1)$, the complete convergence
theorem with coexistence (CCT) holds for $\xi^{[\vep]}$.
\end{corollary} To illustrate the use of the Corollary, we
first show how it quickly gives Theorem~\ref{t:CCT} (which
was already outlined in Section~\ref{sec:proofCCTintro}).
\medskip

\noindent{\it Proof of Theorem~\ref{t:CCT} for $d=2$.} We
apply Corollary~\ref{c:CCTsubadd} above with $\xi^{[\vep]}$
the $(1-\vep)$-voter model, $\xi^{(1-\vep)}$.  The
cancellative property for $|\cN|\le 8$ is shown in
Lemma~\ref{l:qcanc} for $\vep$ small enough, and the finite
range voter perturbation property is established in
Example~\ref{e:qv}. The monotonicity of any $q$-voter model
is elementary (recall \eqref{qvattractive}).  Strict
subadditivity of $r^s$ was already noted in \eqref{rsqvsub}
and so Corollary~\ref{c:CCTsubadd} gives the result.  \qed

Recall from the examples at the end of Section~\ref{sec:vmp}
that for the Lotka-Volterra models, affine voter models and
geometric voter models, we have for non-empty $A\subset
\cN$, $r^s(A)=-(|A|/|\cN|)^2$,
$r^s(A)=-\frac{|A|}{|\cN|}+1$, and
$r^s(A)=\frac{|A|(|\cN|-|A|)}{2|\cN|}$, respectively. All of
these asymptotic rate functions are strictly subadditive, as
one can easily check.  Condition \eqref{Hyp1} was verified
for all of these models in Examples~\ref{e:LV}-\ref{e:GV},
where for the Lotka-Volterra model we take
$\alpha_1=\alpha_2\ge 1/2$.  The following theorems are then
also immediate consequences of
Corollary~\ref{c:CCTsubadd}. The first is the $d=2$ case of
Theorem~1.1 of \cite{CP14} which helped motivate the general
result here.
\begin{theorem}\label{t:CCT-LV}
Let $d=2$, let $\cN$ be a neighbourhood, and let
LV($\alpha$) denote the Lotka-Volterra model with parameters
$\alpha_1=\alpha_2=\alpha$. Then $\Theta^\LV_3>0$ and there
is an $\alpha_c\in(0,1)$ such that for all
$\alpha\in(\alpha_c,1)$, the complete convergence theorem
with coexistence holds for LV($\alpha$).
\end{theorem}
\begin{theorem}\label{t:CCT-AV} Let $d=2$, let $\cN$ be a
neighbourhood, and let AV($\alpha$) denote the affine voter
model with parameter $\alpha$. Then $\Theta^\AV_3>0$ and
there is an $\alpha_c\in(0,1)$ such that for all
$\alpha\in(\alpha_c,1)$, the complete convergence theorem
with coexistence holds for AV($\alpha$).
\end{theorem}
\begin{theorem}\label{t:CCT-GV}
Let $d=2$, let $\cN$ be a neighbourhood, and let
GV($\theta$) denote the geometric voter model with parameter
$\theta$. Then $\Theta^\GV_3>0$ and there is a $\theta_c\in(0,1)$ such that for all 
$\theta\in(\theta_c,1)$, the complete convergence theorem
with coexistence holds for GV($\theta$).
\end{theorem}

Finally, we give the promised direct proof of $f'(0)>0$ for
the $q$ voter model and $d\ge 3$. Recall that in this case
$r_\ell$ is as in \eqref{relldefn},
\begin{equation}\label{c*abs} c^{*}(x,\xi) =
\sum_{\ell=1}^{|\cN|}r_l \Big(\hxi(x)1\{n_1(x,\xi)=\ell\} +
\xi(x)1\{n_0(x,\xi)=\ell\}\Big),
\end{equation} and $f$ is given by $\eqref{fdef}$.
\begin{proposition}(\cite{ASD})\label{p:posthetabigd} Assume
$d\ge 3$ and $\cN$ is a fixed neighbourhood.  Then
$f'(0)>0$.
\end{proposition}
\begin{proof} It follows from \eqref{fdef} and \eqref{c*abs}
that %(recall $n=|\cN|$)
\begin{align}\label{fdef2} f(u)&=\Big\langle
\sum_{\ell=1}^{|\cN|}r_\ell\hxi(0)1(n_1(0,\xi)=\ell)\Big\rangle_u-\Big\langle
\sum_{\ell=1}^{|\cN|}r_\ell\xi(0)1(n_0(0,\xi)=\ell)\Big\rangle_u\\
\nonumber:&=f_1(u)-f_0(u).
\end{align} Let $\{B^x:x\in\bar\cN\}$ be the system of
coalescing random walks introduced in
Section~\ref{sec:coalrw} under $\hat P$, but now in
dimension $d\ge 3$, and let $\{\xi^u_0(x):x\in\Z^d\}$ have
Bernoulli product measure with density $u\in[0,1]$.  Let
$\pi_t\in\cP(\bar \cN)$ be the random partition determined
by the coalescing random walks $\{B^x_t,x\in\bar\cN\}$ using
the equivalence relation $x\sim_t y$ iff
$\sigma(x,y)=\inf\{u:B^x_s=B^y_u\}\le t$, and let
$\pi_\infty=\lim_{t\to\infty}\pi_t$. In this way
$x\sim_\infty y$ iff $\sigma(x,y)<\infty$ is the associated
equivalence relation. If $A\subset\bar\cN$ and
$T\in[0,\infty]$, let
\[ [A]_T=\{\lambda\in\pi_T:\lambda\cap A\neq \emptyset\},\]
and, abusing this notation slightly, write $[x]_T$ for the
cell of $\pi_T$ containing $x$. Use the duality between the
voter model and coalescing random walk to see that
\begin{align} \nonumber
f_1(u)&=\Big\langle\sum_{\emptyset\neq
A\subset\cN}r_{|A|}1(\xi\vert_A=1,
\xi\vert_{\bar\cN\setminus A}=0)\Big\rangle_u\\
\nonumber&=\lim_{T\to\infty}\sum_{\emptyset\neq A\subset
\cN}r_{|A|}\sum_{i=1}^{|A|}\sum_{j=1}^{|\cN|+1-|A|}\hat
P(B^A_T\subset\{x:\xi_0^u(x)=1\},|B^A_T|=i,\\
\nonumber&\phantom{=\lim_{T\to\infty}\sum_{\emptyset\neq
A\subset
\cN}r_{|A|}\sum_{i=1}^{|A|}\sum_{j=1}^{|\cN|+1-|A|}\hat P(}
B_T^{\bar\cN\setminus
A}\subset\{x:\xi_0^u(x)=0\},|B^{\cN\setminus A}_T|=j)\\
\nonumber&=\sum_{\emptyset\neq A\subset
\cN}r_{|A|}\sum_{i=1}^{|A|}\sum_{j=1}^{|\cN|+1-|A|}\lim_{T\to\infty}\hat
P(\sigma(A,\bar\cN\setminus A)>T,\ |[A]_T|=i,\
|[\bar\cN\setminus A]_T|=j)u^i(1-u)^j\\
\label{f1exp}&=\sum_{\emptyset\neq A\subset
\cN}r_{|A|}\sum_{i=1}^{|A|}\sum_{j=1}^{|\cN|+1-|A|}\hat
P(\sigma(A,\bar\cN\setminus A)=\infty,\ |[A]_\infty|=i,\
|[\bar\cN\setminus A]_\infty|=j)u^i(1-u)^j.
\end{align}
Differentiate the above at $u=0$ and so conclude that
\[f_1'(0)=\sum_{\emptyset\neq A\subset \cN}r_{|A|}\hat
P(\sigma(A,\bar\cN\setminus A)=\infty,\ |[A]_\infty|=1).\]
Note that $\sigma(A,\bar\cN\setminus A)=\infty$ iff $A$ is a
union of cells in $\pi_\infty$ and so
\[(\sigma(A,\bar\cN\setminus A)=\infty \text{ and
}|[A]_\infty|=1) \iff A\in \pi_\infty.\] The above
expression for $f_1'(0)$ now becomes %(recall also that
$r_0=r_{|\cN|}=0$)
\begin{equation}\label{fder}f_1'(0)=\hat
E\Bigl(\sum_{\emptyset\neq A\subset
\cN}r_{|A|}1(A\in\pi_\infty)\Bigr).
\end{equation}

In a similar way to \eqref{f1exp}, we get
\begin{align*} f_0(u)&=\Big\langle\sum_{\emptyset\neq
A\subset \cN}r_{|A|}1(\xi\vert_A=0,\xi_{\bar\cN\setminus
A}=1)\Big\rangle_u\\ &=\sum_{\emptyset\neq A\subset
\cN}r_{|A|}\sum_{i=1}^{|A|}\sum_{j=1}^{|\cN|+1-|A|}\hat
P(\sigma(A,\bar\cN\setminus A)=\infty,\ |[A]_\infty|=i,\
|[\bar\cN\setminus A]_\infty|=j)(1-u)^i u^j.
\end{align*}
Differentiating at $u=0$ we get
\begin{align}\nonumber f_0'(0)&=\sum_{\emptyset\neq A\subset
\cN}r_{|A|}\hat P(\sigma(A,\bar\cN\setminus A)=\infty,\
|[\bar\cN\setminus A]_\infty|=1)\\ &=\hat
E\Bigl(\sum_{\emptyset\neq A\subset
\cN}r_{|A|}1(\bar\cN\setminus
A=[0]_\infty)\Bigr).\label{fder-0}
\end{align}
For the last, note that for $A\subset \cN$, the event inside
the $\hat P$ is precisely $\{\bar\cN\setminus A=[0]_\infty\}$.

Note that $\hat P(|\pi_\infty|=|\cN|+1)>0$ for $d\ge 3$ as
there is positive probability none of the walks coalesce,
and so \eqref{picond} holds.  We now may apply
Lemma~\ref{l:genpostheta} to conclude from \eqref{fder} and
\eqref{fder-0} that $f'_1(0)>f'_0(0)$, and so $f'(0)>0$.
\end{proof}

\begin{remark} The above proof applies equally well to any
$\{r_\ell:1\le \ell\le|\cN|\}$ satisfying \eqref{rsubadd}.
In fact, with only notational changes, it applies to $r$ as
in \eqref{rsubaddset} where $f$ is given by \eqref{fdef} and
\begin{equation}\label{generalf}c^*(x,\xi)=\hxi(0)\sum_{\emptyset\neq
A\subset\cN}r(A)1(\xi|_\cN=1_A)+\xi(0)\sum_{\emptyset\neq
A\subset\cN}r(A)1(\xi|_\cN=1_{\cN\setminus
A}).\end{equation}
\end{remark}

\section{A general convergence theorem to super-Brownian motion in two dimensions}\label{sec:couplingsemimart}
In Section~\ref{sec:proofCCT} we extended the SDE
construction of our particle systems to allow for infinite
initial conditions and also simultaneously deal with
${\hxi^{[\vep]}}$.  In this setting we have no need to deal
with these extensions and so no longer require monotonicity
or have to deal with the equation \eqref{hatconstruct} for
$\hxi$.  Assume the walk kernel, $p$, is as in the beginning
of Section~\ref{sec:proofCCTintro} and
$\{\xi^{[\vep]}:0<\vep\le\vep_0\}$ is an asymptotically
symmetric finite range voter model perturbation as in
Section~\ref{sec:vmp}. This assumption will be in force
throughout the rest of this work.  We continue to use
notation from Section~\ref{sec:vmp} and, as in that Section,
$N\ge N(\vep_0)>e^3$ is our fundamental parameter where
$\vep=\vep_N=(\log N)^3/N$.  Recall from \eqref{dNdecompsa}
that the rescaled voter model perturbation, $\xi^N$ in
\eqref{rescaledef} has rate function
\begin{align}\label{dNdecompsa'}
c^N(x,\xi^{(N)})=Nc^{N,\VM}(x,\xi^{(N)})+(\log
N)c^{N,a}(x,\xi^{(N)})+&(\log N)^3c^{N,s}(x,\xi^{(N)}),\\
\nonumber& x\in\SN,\ \xi^{(N)}\in\{0,1\}^\SN,
\end{align} where $c^{N,s}$ and $c^{N,a}$ are as in
\eqref{cNsdef} and \eqref{cNadef}, respectively.  Assume
throughout that $|\xi^N_0|<\infty$.

By Remark~\ref{votepertlip} we can construct $\xi^N$, as in
Section~\ref{sec:proofCCT}, as the unique
$(\cF^N_t)$-adapted solution of
\begin{multline}\label{construct} \xi^N_t(x) =
\xi^N_0(x)+\int_0^t\int (1-\xi^N_{s-}(x))1(u\le
c^N(x,\xi^N_{s-}))N^{x,0}(ds,du) \\ -\int_0^t\int
\xi^N_{s-}(x)1(u\le c^N(x,\xi^N_{s-}))N^{x,1}(ds,du), \quad
t\ge 0,\ x\in\SN.
\end{multline} Then $\xi^N$ is the unique Feller process
associated with the rate function $c^N(x,\xi)$.  Here
$\{N^{x,i}, x\in\SN,i=0,1\}$ are independent Poisson point
processes on $\R_+\times\R_+$ with intensity $ds\times du$
and $\{\cF^N_t,t\ge0\}$ is the natural right-continuous
filtration generated by these point processes.

We again use this setting to couple spin-flip systems but
now in a different manner from Section~\ref{sec:proofCCT}.
If $\bar c^N(x,\xi)$ is a rate function, also satisfying
\eqref{L1boundunscaleda} and \eqref{pertboundunscaleda},
such that for $\xi\le\bar\xi$,
\begin{equation}\label{compcond}
\begin{aligned}
\bar c^N(x,\bar\xi)&\ge c^N(x,\xi)\text{ if
}\xi(x)=\bar\xi(x)=0, \\
\bar c^N(x,\bar\xi)&\le c^N(x,\xi)\text{ if }
\xi(x)=\bar\xi(x)=1, 
\end{aligned}
\end{equation}
and $\bar\xi^N_t$ is constructed as in 
\eqref{construct} using $\bar c^N(x,\xi)$, then
\begin{equation}\label{SDEcoupling} \bar\xi^N_0\ge
\xi^N_0\text{ implies }\bar \xi^N_t \ge \xi^N_t \text{ for
all $t\ge 0$}.
\end{equation} In Proposition~2.1 of \cite{CP07} this is
proved under monotonicity of $\xi^N$ when $\xi^N$ is a
killed version of $\bar\xi^N$, but the same elementary
argument applies without monotonicity under
\eqref{compcond}.

As in \eqref{genlXNdefintro}, define the measure-valued
process associated with $\xi_t^N(x),\ x\in\SN$ by
\begin{equation}\label{genlXNdef}X^N_t=(1/N')\sum_{x\in\SN}\xi^N_t(x)\delta_x.\end{equation}
We use the SDE \eqref{construct} to see that $X^N$ satisfies
a martingale problem reminiscent of that of SBM.  Introduce
the scaled probability kernel
\begin{equation}\label{pNdefinition} p_N(x)=p(x\sqrt N),\ \
x\in\SN.
\end{equation} Unless otherwise indicated, assume $\Phi\in
C_b([0,T]\times \R^2)$ is such that
$\overset{\centerdot}{\Phi} := \frac{\partial\Phi}{\partial
t} \in C_b([0,T]\times \R^2)$.  Define
\[ A_N \Phi(s,x) = \sum_{y \in S_N} N p_N(y-x) \left(
\Phi(s,y) - \Phi (s,x) \right),
\] and
\[ D^{N,1}_t(\Phi) = \int_0^t X^N_s\big(A_N\Phi(s,\cdot) +
\dot{\Phi}(s,\cdot)\big)ds.
\]
Introduce
\[
\ell^{(j)}_N = \begin{cases}
\log N&\text{if }j=2,\\
(\log N)^3&\text{if }j=3.
\end{cases}
\]
To be consistent with the notation in \cite{CMP} for
$x\in\SN$ and $\xi\in\{0,1\}^{\SN}$ we define
\begin{equation}\label{dNjdef}
\begin{aligned}
d^{N,2}(x,\xi)&=\hxi(x)c^{N,a}(x,\xi)\\
d^{N,3}(x,\xi)&=\hxi(x)c^{N,s}(x,\xi)-\xi(x)c^{N,s}(x,\xi)
\end{aligned}
\end{equation}
and for $j=2,3$,
\begin{align}
\label{dNjsdef}
d^{N,j}(s,\xi,\Phi) &= \frac{\ell^{(j)}_N}{N'}
\sum_{x\in\SN}\Phi(s,x)d^{N,j}(x,\xi),\\
D^{N,j}_t(\Phi)&=\int_0^t d^{N,j}(s,\xi^N_s,\Phi)ds.
\label{DNjdef}
\end{align}
One may then use the stochastic calculus for Poisson
integrals and integration by parts to rewrite
$\xi^N_t(x)\Phi(t,x)$, just as in Propositions~2.2 and 2.3
in \cite{CP05}, to see that
\begin{equation} \label{SMG1}
X_t^N(\Phi(t,\cdot))  =  X^N_0(\Phi(0,\cdot)) 
+ D_t^{N,1}(\Phi) + D_t^{N,2}(\Phi) 
+  D_t^{N,3}(\Phi)+  M_t^N(\Phi),
\end{equation}
where $M^N_t(\Phi)$ is a square integrable 
$(\mathcal{F}^N_t)$-martingale 
with previsible square function 
\begin{equation}\label{SMGsqfn}\langle M^N(\Phi)\rangle_t = \langle
M^N(\Phi)\rangle_{1,t} +
\langle M^N(\Phi)\rangle_{2,t},\end{equation}
with
\begin{equation}\label{SMG2}
\begin{aligned}
\langle M^N(\Phi)\rangle_{1,t} &= \int_0^t
\frac{\log N}{N'}
\sum_{x\in\SN} \Phi(s,x)^2 c^{N,\VM}(x,\xi^N_s)ds,\\
\langle M^N(\Phi)\rangle_{2,t} &= \int_0^t\frac{1}{{N'}^2}
 \sum_{x\in\SN} \Phi(s.x)^2[\ell^{(2)}_N c^{N,a}(x,\xi^N_s)+\ell_N^{(3)}c^{N,s}(x,\xi^N_s)]\,ds.
\end{aligned}
\end{equation}
Note that in spite of the suggestive notation the last term
may in fact be negative.  The above three displays are
reminiscent of the martingale problem (MP) for a
super-Brownian motion in Section~\ref{sec:scallim}.  In the
next two sections we will take term by term limits in
\eqref{SMG1} to establish Theorem~\ref{t:SBMgenintro}, our
general weak convergence result. We restate it below for
convenience.  Recall that $\sigma^2$ is as in \eqref{covp}
and $\Theta_2,\Theta_3$ are defined in \eqref{Thetadefns}.
\begin{theorem}\label{t:SBMgen} Assume
$\{\xi^{[\vep]}:0<\vep\le \vep_0\}$ is an asymptotically
symmetric finite range voter model perturbation on $\Z^2$.
If $X_0^N\to X_0$ in $\MF$, then
\[ X^N \To \SBM(X_0,4\pi\sigma^2,\sigma^2,\Theta_2+\Theta_3)
\text{ in the Skorokhod space $D(\R_+,\MF)$ as }N\to\infty.
\]
\end{theorem} In Remark~\ref{rem:symmSBMintro} we showed how
Theorem~\ref{t:SBM}, the rescaled limit theorem for
$2$-dimensional $q$-voter models, follows from the above.
We now describe a number of other corollaries, all of course
in two dimensions.

\begin{example}\label{e:lv2}(Lotka-Volterra Models). Recall
from Example~\ref{e:LV} that the Lotka-Volterra Models
discussed there constituted an asymptotically symmetric
finite range voter model.  The kernel $p$ is $1_\cN/|\cN|$
and so $\sigma^2$ is as in \eqref{nbhddef}.  Recall also
that $\xi^{[\vep_N]}$ denotes a Lotka-Volterra model with
parameters $(\alpha^N_1,\alpha^N_2)$, where
\[\alpha_i^N=1-\vep_N+\beta_i^{(\vep_N)}(\log(1/\vep_N))^{-2}\vep_N=1-\frac{(\log
N)^3}{N}+\beta_i^N\frac{\log N}{N},\] and
$\beta^N_i\to\beta_i\in\R$ as $N\to\infty$. Let
\begin{equation}\label{xiNXNdef}\xi^N_t(x)=\xi_{Nt}^{[\vep_N]}(x\sqrt
N)\text{ for }x\in\SN,\text{ and
}X_t^N=\frac{1}{N'}\sum_{x\in\SN}\xi^N_t(x)\delta_x.
\end{equation}
If $X^N_0\to X_0\in\MF(\R^2)$ as $N\to\infty$, then
Theorem~\ref{t:SBMgen} implies that
\[X^N \To
\SBM(X_0,4\pi\sigma^2,\sigma^2,\Theta^\LV_2+\Theta^\LV_3),\]
where $\Theta_i^\LV$ are as in Example~\ref{e:LV}.  We can
write these drifts in another way. Let $e_1,e_2$ be iid rv's
uniformly distributed over $\cN$ and, using the notation
from Section~\ref{sec:coalrw}, set
\begin{equation}\label{Kgammadef}K=E(K_3(0,e_1,e_2)1(e_1\neq e_2)),\ \text{ and }\gamma=E(K_2(\{e_1,e_2\},\{0\})).\end{equation}
Then with a bit of work one can show that $\Theta_3^\LV=K$
and $\Theta_2^\LV=(\beta_0-\beta_1)\gamma$. For example, in
the proof of the latter, the summand $A$ arising in
$\Theta_2^\LV$ will be $[e_1]_t$ (recall from the proof of
Proposition~\ref{p:posthetabigd} this is the set of initial
conditions in $\cN$ that have coalesced with $e_1$ by time
$t$) and then let $t\to\infty$.  In this way the drift of
the limiting SBM becomes $K+(\beta_0-\beta_1)\gamma$.  This
is the form of the limit derived in Theorem 1.5 of
\cite{CMP} whose proof will play an important role in the
derivation of Theorem~\ref{t:SBMgen} to come.
\end{example}

\begin{example}\label{e:av3} (Affine Voter Models). In
Example~\ref{e:av} we showed the affine voter models are an
asymptotically symmetric voter model perturbation with
kernel $p(x)=1_\cN(x)/|\cN|$, $\sigma^2$ as in
\eqref{nbhddef}, $\Theta_2=0$, and
$\Theta_3^\AV=\kappa=\lim_{t\to\infty}(\log t)^3\hat
P(|B^\cN_t|=3)>0$ (recall \eqref{constantid}). Let
$\xi^{[\vep_N]}_{Nt}(x)$, $x\in\Z^2$ be an affine voter
model with parameter $\alpha=1-\vep_N$, and define $\xi^N$
and $X^N$ as in \eqref{xiNXNdef}.  Theorem~\ref{t:SBMgen}
implies that
\[\text{if }X^N_0\to X_0\in\MF(\R^2),\text{ then }X^N \To
\SBM(X_0,4\pi\sigma^2,\sigma^2,\kappa).\]
\end{example}
\begin{example}\label{e:gm3} (Geometric Voter Models). By
Example~\ref{e:GV}, the geometric voter models give an
asymptotic symmetric voter model perturbation with $p$ and
$\sigma^2$ as above, $\Theta_2=0$ and
$\Theta^\GV_3=\frac{|\cN|}{2}\Theta_3^\LV=\frac{|\cN|}{2}K$
where $K>0$ is as in \eqref{Kgammadef}. Let $\xi^{[\vep_N]}$
be a geometric voter model with parameter $\theta=1-\vep_N$
and assume $\xi^N$ and $X^N$ are as in \eqref{xiNXNdef}.
Theorem~\ref{t:SBMgen} implies that
\[\text{if }X^N_0\to X_0\in\MF(\R^2),\text{ then }X^N \To
\SBM(X_0,4\pi\sigma^2,\sigma^2,\frac{|\cN|}{2}K).\]
\end{example}

\begin{remark}\label{r:cvrseb} We have started with a finite
range voter model perturbation $\{c_\vep:0<\vep\le \vep_0\}$
and rescaled to obtain rates $c^N$ as in
\eqref{dNdecompsa'}. At times it may be more natural to
start with the rescaled rates $c^N$ for $N\ge N_0\ge e^3$.
Assume now that
\begin{align*}c^N(x,\xi^{(N)})=Nc^{N,\VM}(x,\xi^{(N)})+(\log
N)c^{N,a,*}(x,\xi^{(N)})+&(\log
N)^3c^{N,s,*}(x,\xi^{(N)}),\\ \nonumber& x\in\SN,\
\xi^{(N)}\in\{0,1\}^\SN,
\end{align*} where for some $\R$-valued functions
$g_i^{N,a,*}$, $g^{N,s,*}$ on $\{0,1\}^\cN$, $i=0,1$,
\begin{equation}\label{cNa*}c^{N,a,*}(x,\xi^{(N)})=\hxi(x\sqrt
N)g_1^{N,a,*}(\xi|_{x\sqrt N+\cN})+\xi(x\sqrt
N)g_0^{N,a,*}(\xi|_{x\sqrt N+\cN}),
\end{equation} and
\begin{equation}\label{cNs*}c^{N,s,*}(x,\xi^{(N)})=\hxi(x\sqrt
N)g^{N,s,*}(\hxi|_{x\sqrt N+\cN})+\xi(x\sqrt
N)g^{N,s,*}(\xi|_{x\sqrt N+\cN}).
\end{equation}
We assume there are functions $g_i^{a,*}$, $g^{s,*}$ on
$\{0,1\}^\cN$, $i=0,1$, such that
\[\lim_{N\to\infty}\Bigl(\sum_{i=0}^1\Vert
g_i^{a,N,*}-g_i^{a,*}\Vert_\infty\Bigl)+\Vert
g^{s,N,*}-g^{s,*}\Vert_\infty=0,\] and also that ${\bf 0}$
and ${\bf 1}$ are traps, that is, for all $N$ and $x\in\SN$,
$c^N(x,{\bf0})=c^N(x,{\bf1})=0$.  This setting clearly
includes the $c^N$ arising in \eqref{dNdecompsa'}, and in
fact appears to be more general since \eqref{dNdecompsa'}
requires $\xi(x\sqrt N)c^{N,a}(x,\xi^{(N)})=0$. To see that
it is in fact included in \eqref{dNdecompsa'}, define for
$\xi\in\{0,1\}^\cN$,
\begin{equation}\label{gdefn0}
g_0^{\vep_N}(\xi)=g^{N,s,*}(\xi)+(\log
N)^{-2}g_0^{N,a,*}(\xi)\to g^{s,*}(\xi)\ \ \text{as
}N\to\infty,
\end{equation}
and
\begin{equation}\label{gdefn1}
g_1^{\vep_N}(\xi)=g_0^{\vep_N}(\hxi)+(\log
N)^{-2}[g_1^{N,a,*}(\xi)-g_0^{N,a,*}(\hxi)]\to
g^{s,*}(\hxi)\ \ \text{as }N\to\infty.
\end{equation}
Then one easily checks that $c^N$ is as in
\eqref{dNdecompsa'} where $c^{N,s}$ and $c^{N,a}$ are given
by \eqref{cNsdef} and \eqref{cNadef}, respectively, in terms
of the $g^{\vep_N}_i$ given above. Moreover,
\begin{align*}
g^a(\xi):=&\lim_{N\to\infty}(\log(1/\vep_N))^2(g_1^{\vep_N}(\xi)-g_0^{\vep_N}(\hxi))\\
=&\lim_{N\to\infty}(\log(1/\vep_N))^2(\log N)^{-2}[g_1^{N,a,*}(\xi)-g_0^{N,a,*}(\hxi)]=g_1^{a,*}(\xi)-g_0^{a,*}(\hxi).
\end{align*}
Use $c^N$ to define $c_{\vep_N}$ as in \eqref{cNvepN}. Then
$\{c_\vep:0<\vep\le \vep_0:=\vep_{N_0}\}$ is an
asymptotically symmetric voter model perturbation. Recalling
\eqref{rNalim}, for $A\subset\cN$ we have
\begin{equation}\label{raformula2}r^a(A)=g^a(1_A)=g_1^{a,*}(1_A)-g_0^{a,*}(1_{\cN\setminus
A}),
\end{equation}
while \eqref{gdefn0} and \eqref{rNslim} give
\begin{equation}\label{rsformula2}
r^s(A)=\lim_{N\to\infty}g_0^{\vep_N}(1_{\cN\setminus A})=g^{s,*}(1_{\cN\setminus A}).
\end{equation}
If $\xi^{[\vep]}$ is the process with rate $c_\vep$, then
\eqref{cNvepN} implies that
$\xi^N_t(x)=\xi^{[\vep_N]}_{N\cdot}(\sqrt N\cdot)$
($x\in\SN$) has rate function $c^N$. Therefore if $X^N$ is
as in \eqref{genlXNdef}, we may apply Theorem~\ref{t:SBMgen}
to conclude that \break
$X^N\To\SBM(X_0,4\pi\sigma^2,\sigma^2,\Theta_2+\Theta_3)$,
where (by \eqref{raformula2} and \eqref{rsformula2})
\begin{equation}\label{tctheta}
\Theta_2=\sum_{\emptyset\neq A\subset\cN}(g^{a,*}_1(1_A)-g_0^{a,*}(1_{\cN\setminus A}))K_2(A,\bar\cN\setminus A),\ 
 \Theta_3=\sum_{\emptyset\neq A\subset\cN}g^{s,*}(1_{\cN\setminus A})(\Theta^+(A)-\Theta^-(A)).
\end{equation}
\end{remark}
\begin{example}\label{e:lv3} As a specific example of the
above, recall the lower order weak limit theorem for
two-dimensional Lotka-Volterra models in \cite{CP08}, where
now $\alpha^N_i=1+\frac{\log N}{N}\beta_i^N$ and
$\beta^N_i\to \beta_i$, $i=0,1$.  We continue to assume
$p(x)=1_\cN/|\cN|$. This led to rescaled rates (see (1.6) in
\cite{CP08})
\begin{equation*}
c^N(x,\xi^{(N)})=Nc^{N,\VM}(x,\xi^{(N)})+\log
N\Bigl[\hxi(x\sqrt N)\beta_0^Nf_1(x\sqrt N,\xi)^2+\xi(x\sqrt
N)\beta^N_1f_0(x\sqrt N,\xi)^2\Bigr].
\end{equation*} So comparing with \eqref{cNa*} and
\eqref{cNs*} above we have $c^{N,s,*}=g^{x,*}=0$, and
$g_1^{a,*}(\xi)=\beta_{0}f_1(0,\xi)^2$ and
$g_0^{a,*}(\hxi)=\beta_{1}f_0(0,\hxi)^2=\beta_1f_1(x,\xi)^2$. So
\eqref{tctheta} implies $\Theta_3=0$ and
\[\Theta_2=(\beta_0-\beta_1)\sum_{\emptyset\neq
A\subset\cN}\Bigl(\frac{|A|}{|\cN|}\Bigr)^2K_2(A,\bar\cN\setminus
A)=\Theta_2^\LV=(\beta_0-\beta_1)\gamma,\] where $\gamma$ is
as in Example~\ref{e:lv2}, which also gives the above
expression for $\Theta_2^\LV$.  Therefore by
Remark~\ref{r:cvrseb} and Theorem~\ref{t:SBMgen}, if we
define $\xi^N$, and $X^N$ as in Example~\ref{e:lv2} and
$X^N_0\to X_0$, then
\[X^N \To
\SBM(X_0,4\pi\sigma^2,\sigma^2,(\beta_0-\beta_1)\gamma).\]
This is Theorem~1.2 of \cite{CP08} but with a seemingly
different parameter, $\gamma$, in place of the $\gamma^*$ in
\cite{CP08}. It is, however, easy to use Proposition~2.2 of
\cite{CP08} and \eqref{KnDef} with $n=2$ to check that
$\gamma^*=\gamma$.
\end{example}

\section{Controlling the drift terms, and total mass bounds}\label{s:qvdrifttm}
The goal in this section (Proposition~\ref{p:d3asymp}) is to
show that for $j=2,3$, the drift terms $D^{N,j}_t(\Phi)$
arising in \eqref{SMG1} behave asymptotically like
$\Theta_j\int_0^tX^N_s(\Phi_s)\,ds$, where
$\Phi_s(x)=\Phi(s,x)$. Use notation from the previous
section. In particular, $\xi^N$ is as in \eqref{rescaledef}
with rate function, $c^N$, as in \eqref{dNdecompsa'}.
\subsection{Small time comparison bounds}\label{ss:compproc}
We begin with some elementary bounds on the drifts.  Let
$\cN_N=\cN/\sqrt N$.  By the definitions of
$d^{N,j}(x,\xi)$, \eqref{cNarep} and \eqref{cNsrep}, we have
\begin{equation}\label{dNjrep}
\begin{aligned}
d^{N,2}(x,\xi) & =\hxi(x)\sum_{\emptyset\neq A\subset\cN_N}r^{N,a}(\sqrt NA)1\{\xi|_{x+\cN_N}=1_{x+A}\},\\
& =\sum_{\emptyset\neq A\subset\cN_N}r^{N,a}(\sqrt NA)1\{\xi|_{x+\bar\cN_N}=1_{x+A}\},\\
 d^{N,3}(x,\xi) & = \hxi(x)\sum_{\emptyset\neq A\subset\cN_N}r^{N,s}(\sqrt N A)
1\{\xi|_{x+\cN_N}=1_{x+A}\} -\xi(x)\sum_{\emptyset\neq A\subset\cN_N}
r^{N,s}(\sqrt NA)1\{\xi|_{x+\cN_N}=1_{x+\cN_N\setminus A}\}\\
& = \sum_{\emptyset\neq A\subset\cN_N}r^{N,s}(\sqrt NA)
1\{\xi|_{x+\bar\cN_N}=1_{x+A}\} -\sum_{\emptyset\neq A\subset\cN_N}
r^{N,s}(\sqrt NA)1\{\xi|_{x+\bar\cN_N}=1_{x+\bar\cN_N\setminus A}\}.
\end{aligned}
\end{equation}
Observe that at most one term in each of the above sums can
be nonzero, and this would require that $\xi(y)=1$ for some
$y\in x+\bar\cN_N$.  This implies that for $j=2,3$,
\begin{equation}\label{dNjbnd1}
|d^{N,j}(x,\xi)|\le \|r\|1\{\xi(y)=1 
\text{ for some }y\in x+\bar\cN_N\}\le\|r\|\sum_{y\in
  x+\bar\cN_N}\xi(y),
\end{equation}
and hence 
\begin{equation}\label{dNjbnd0}
|d^{N,j}(x,\xi)|\le \|r\|,
\end{equation}
and for $\Phi:\R_+\times\R^2\to\R$,
\begin{equation}\label{dNjsbnd}
|d^{N,j}(s,\xi,\Phi)|\le \Vert r\Vert |\Phi_s|_\infty|\bar\cN|\frac{\ell^{(j)}_N}{N'}\sum_{x\in\SN}\xi(x).
\end{equation}

To control the drifts over small time intervals we will
condition $d^{N,2}(s,\xi^N_s,\Phi)$ and
$d^{N,3}(s,\xi^N_s,\Phi)$ on $\cF^N_{s-u_N}$ for small $u_N$
and compare this small time evolution of $\xi^N$ with a
voter model, for which we can make explicit calculations
using duality.

Set $f_i^{(N)}(x,\xi)=\sum_{y\in\SN}p_N(y-x)1\{\xi(y)=i\}$,
$n_i^{(N)}(x,\xi)=\sum_{y\in\cN_N}1\{\xi(x+y)=i\}$, \break
$\underline p=\min\{p(y):p(y)>0\}>0$, and introduce
\begin{equation}\label{wNdef}w_N=1-\frac{\|r\|}{\underline
p}\vep_N\to 1\text { as }N\to\infty.
\end{equation}
We use the construction given in \eqref{construct} with
other rate functions to provide a coupling of $\xi^N_t$ with
some comparison processes. First, let
$\xi^{N,\VM}_t\in\{0,1\}^\SN$ be the rescaled voter model
defined as in \eqref{construct} but using the rate function
\[Nw_Nc^{N,\VM}(x,\xi),\ x\in\SN,\xi\in\{0,1\}^\SN.\] Next,
define the 1-biased rate function
\begin{align}\label{bvoterratesdef}
\bar c^{N,b}(x,\xi) &=
Nw_Nc^{N,\VM}(x,\xi) + \hxi(x)\Bigl(2+{\underline p}^{-1}\Bigr)\|r\|(\log N)^3  n^{(N)}_1(x,\xi).
\end{align}
%It is easy to check that this is a monotone rate function.
Let 
$\bxi^{N}_t$ be the 1-biased voter model constructed 
with $\bar c^{N,b}$  
as in
\eqref{construct}. 

We next verify \eqref{compcond} for $c^N$ and $\bar
c^{N,b}$, that is we will show that for $\xi\le\bar\xi$,
  \begin{equation}\label{cNborder}
  \begin{aligned} {\bar c}^{N,b}(x,\bar\xi)&\ge
c^N(x,\xi)\text{ if }\xi(x)=\bar\xi(x)=0, \\ {\bar
c}^{N,b}(x,\bxi)&\le c^N(x,\xi)\text{ if }
\xi(x)=\bar\xi(x)=1.
\end{aligned}
\end{equation} First combine \eqref{dNdecompsa},
\eqref{cNarep} and \eqref{cNsrep} to that for $x\in\SN$ and
$\xi\in\{0,1\}^\SN$,
\begin{align}\label{cNdecomprs} \nonumber
c^N(x,\xi)&=Nc^{N,\VM}(x,\xi)+\hxi(x)\Bigl[\sum_{\emptyset\neq
A\subset\cN_N}(\log N r^{N,a}(\sqrt NA)+(\log
N)^3r^{N,s}(\sqrt NA))1(\xi|_{x+\cN_N}=1_{x+A})\Bigr]\\
\nonumber&\quad+\xi(x)\Bigl[\sum_{\emptyset\neq
A\subset\cN_N} (\log N)^3 r^{N,s}(\sqrt NA)
1(\xi|_{x+\cN_N}=1_{x+\cN_N\setminus A})\Bigr]\\
\nonumber&=Nw_Nc^{N,\VM}(x,\xi)+\hxi(x)\Bigl[\sum_{\emptyset\neq
A\subset\cN_N}(\log Nr^{N,a}(\sqrt NA)+(\log
N)^3r^{N,s}(\sqrt NA))1(\xi|_{x+\cN_N}=1_{x+A})\\
\nonumber&\phantom{=Nw_Nc^{N,\VM}(x,\xi)+\hxi(x)\Bigl[\ \
}+N(1-w_N)f_1^{(N)}(x,\xi)\Bigr]\\
\nonumber&\quad+\xi(x)\Bigl[\sum_{\emptyset\neq
A\subset\cN_N}(\log N)^3 r^{N,s}(\sqrt
NA)1(\xi|_{x+\cN_N}=1_{x+\cN_N\setminus
A})+N(1-w_N)f_0^{(N)}(x,\xi)\Bigr]\\
&:=Nw_Nc^{N,\VM}(x,\xi)+\hxi(x)\tilde
c_1^N(x,\xi)+\xi(x)\tilde c_0^N(x,\xi).
\end{align}
If $\xi|_{x+\cN_N}=1_{x+\cN_N\setminus A}$ and $p(\sqrt NA)>0$, then
\[\frac{f_0^{(N)}(x,\xi)}{\underline p}=\sum_{y\in\sqrt NA}\frac{p(y)}{\underline p}\ge 1.\]
Use Lemma~\ref{psupp} and then the above to see that
\begin{align}\nonumber\tilde
c^N_0(x,\xi)&\ge\sum_{\emptyset\neq A\subset\cN_N, p(\sqrt
NA)>0} (\log N)^3 (r^{N,s}(\sqrt NA)\wedge
0)1(\xi|_{x+\cN_N}=1_{x+\cN_N\setminus A})+\frac{\Vert
r\Vert}{\underline p}(\log N)^3f_0^{(N)}(x,\xi)\\
\label{tcolb}&\ge -(\log N)^3\Vert r\Vert
\frac{f^{(N)}_0(x,\xi)}{\underline p}+\frac{\Vert
r\Vert}{\underline p}(\log N)^3f_0^{(N)}(x,\xi)= 0.
\end{align}
More simply we have 
\begin{equation}\label{tc1ub} \tilde c^N_1(x,\xi)\le
\Bigl[2\Vert r\Vert(\log N)^3+\frac{\Vert r\Vert}{\underline
p}(\log N)^3\Bigr]n_1^{(N)}(x,\xi).
\end{equation}
Turning to \eqref{cNborder}, assume now that $\xi\le \bxi$.
If $\xi(x)=\bxi(x)=1$, then by
\eqref{cNdecomprs},\eqref{tcolb} and the monotonicity of the
voter model,
\begin{align*}c^N(x,\xi)&=Nw_Nc^{N,\VM}(x,\xi)+\tilde c^N_0(x,\xi)\\
&\ge Nw_Nc^{N,\VM}(x,\xi)\ge Nw_Nc^{N,\VM}(x,\bxi)=\bar c^{N,b}(x,\bxi),
\end{align*}
the last by \eqref{bvoterratesdef}.  If $\xi(x)=\bxi(x)=0$,
then by \eqref{cNdecomprs} and \eqref{tc1ub},
\begin{align*}c^N(x,\xi)&=Nw_Nf^{(N)}_1(x,\xi) +\tilde
c^N_1(x,\xi)\\ &\le Nw_Nf^{(N)}_1(x,\bxi) +(2+{\underline
p}^{-1})\vert r\Vert (\log N)^3n_1^{(N)}(x,\xi)= \bar
c^{N,b}(x,\bxi),
\end{align*} where \eqref{bvoterratesdef} is again used in
the last equality.  This proves \eqref{cNborder}.

More simply \eqref{cNborder} also holds if $c^N(x,\xi)$ is
replaced with $w_NNc^{N,\VM}(x,\xi)$.  This is immediate
from the monotonicity of the voter model and
\eqref{bvoterratesdef}. Having verified \eqref{compcond} for
two pairs of processs we may apply the coupling result
\eqref{SDEcoupling} and conclude that if the three processes
$\xi^N,\xi^{N,\VM},\bxi^N$ have the same initial state
$\xi^N_0$, then with probability one,
\begin{equation}\label{couple1}
\xi^N_t\le \bxi^N_t,\text{ and }\xi^{N,\VM}_t\le \bxi^N_t \text{ for
  all }t\ge 0.
\end{equation}
Use these processes to define the empirical processes of
one's, ${\bar X}^N_t$ and $X^{N,\VM}_t$ respectively, as in
\eqref{genlXNdef}.  We will need to compare these processes
over small time periods. We assume $u_N$ satisfies
\begin{equation}\label{uNdef}
\frac{C_{\ref{uNdef}}}{\sqrt N}\le u_N\le(\log N)^{-p}
\text{ for some }C_{\ref{uNdef}}>0, p>6, 
\end{equation}
and recalling that as $N\ge e^3$, we have $u_N(\log N)^3<1$.  

\begin{lemma}\label{l:masscmps} For some universal
$C_{\ref{masscmps}}$ and all $N$, if
$\xi^N_0=\xi^{N,\VM}_0=\xi^N_0$ then
\begin{equation}\label{masscmps}
\begin{aligned} E\big[\bar X^N_{u_N}(\1) -
X^N_{u_N}(\1)\big] &\le C_{\ref{masscmps}}(\log
N)^{3-p}X^N_0(\1),\\ E\big[\bar X^{N}_{u_N}(\1) -
X^{N,\VM}_{u_N}(\1) \big] &\le C_{\ref{masscmps}}(\log
N)^{3-p}X^N_0(\1),\\ E\big[\big| X^{N,\VM}_{u_N}(\1) -
X^N_{u_N}(\1)\big| \big] &\le C_{\ref{masscmps}}(\log
N)^{3-p}X^N_0(\1).
\end{aligned}
\end{equation}
\end{lemma}
\begin{proof} 
For the first inequality, by Lemma 4.1 in \cite{CP05},
$E\bar X^N_s(\1) \le e^{(2+{\underline p}^{-1})\|r\|(\log N)^3s}X^N_0(\1)$. Set
$\Phi=\1$ in \eqref{SMG1} to get 
\[E[X_t^N(\1)]=X^N_0(\1)+\int_0^tE[d^{N,2}(s,\xi^N_s,\1)+ d^{N,3}(s,\xi^N_s,\1)]\,ds,  
\]   
where by  \eqref{dNjsbnd},
$E\big[|d^{N,2}(s,\xi^N_s,\1)|+|d^{N,3}(s,\xi^N_s,\1)|\big]
\le2 \|r\||\cN|(\log N)^3E[X^N_s(\1)]$.   
An elementary integration by parts now implies
$E[X^N_s(\1)]\ge e^{-2\|r\||\cN|(\log N)^3
  s}X_0^N(\1)$.
The above inequalities with $s=u_N$ give the first
inequality. The second is even simpler since then
$X^{N,\VM}(\1)$ is a martingale.  The final inequality then
follows by the triangle inequality.
\end{proof}

Let $\Phi\in
C_b([0,T]\times \R^2)$ and $\|\Phi\|_\infty$ denote its sup norm. Define $|\Phi|_\Lip$,
respectively $|\Phi|_{1/2,N}$, to be the smallest element in
$[0,\infty]$ such that 
\begin{equation}\label{|Phi|defs}
\begin{cases}
|\Phi(s,x)-\Phi(s,y)| \le |\Phi|_\Lip|x-y|, \quad
\forall s\in [0,T], x,y\in\R^2,\\
|\Phi(s-u_N,x)-\Phi(s,x)| \le |\Phi|_{1/2,N}\sqrt{u_N},
\quad\forall s\in[u_N,T], x\in\R^2.
\end{cases}
\end{equation}%\footnote{Define $\|\Phi\|_\infty$?}
We will write $\|\Phi\|_{\Lip}=\|\Phi\|_\infty +
|\Phi|_{Lip}$, $\|\Phi\|_{1/2,N}=\|\Phi\|_\infty +
|\Phi|_{1/2,N}$, and $\|\Phi\|_N= \|\Phi\|_\infty
+|\Phi|_{Lip}+ |\Phi|_{1/2,N}$. We also will abuse notation
slightly and write
$\|\Phi_s\|_\Lip=\|\Phi_s\|_\infty+|\Phi_s|_{\Lip}$ for the
usual Lipschitz norm of $x\to\Phi_s(x)=\Phi(s,x)$.  Note
that since $u_N<\sqrt{u_N}$ for $N\ge e^3$, we have
\begin{equation}\label{holderhalf} \text{if, in addition,
$\overset{\centerdot}{\Phi} \in C_b([0,T]\times \R^2)$, then
$|\Phi|_{1/2,N}\le \|\overset{\centerdot}{\Phi}\|_\infty$
for all $N\ge e^3$.}
\end{equation} We will often suppress the dependence on $N$
in the above ``norms".

\bigskip
We claim that  for  $\xi\le\eta$, $j=2,3$,
\begin{equation}\label{dNjdiff}
|d^{N,j}(x,\eta) -d^{N,j}(x,\xi)|
 \le 2\|r\|\sum_{y\in x+\bar\cN_N}(\eta(y) - \xi(y)).
\end{equation}
If $\xi\ne\eta$ on $x+\bar\cN_N$, then the right-hand side
is at least $2\|r\|$ and so the above follows from
\eqref{dNjbnd0} and the triangle inequality. If
$\xi|_{\bar\cN_N}= \eta|_{\bar\cN_N}$ then the left-hand
side is zero, and so \eqref{dNjdiff} is trivial.
\begin{lemma}\label{l:dvmcomp} There is a constant
$C_{\ref{dvmcomp}}>0$ such that for $j=2,3$, all $T>0$,
\break$\Phi\in C_b([0,T]\times\R^2)$ and all $s\in[0,T]$,
\begin{equation}\label{dvmcomp}
E_{\xi^N_0}\big[\big|d^{N,j}(s,\xi^N_{u_N},\Phi) -
d^{N,j}(s,\xi^{N,\VM}_{u_N},\Phi)\big|\big] \le
C_{\ref{dvmcomp}}\|\Phi\|_\infty(\log N)^{6-p}X^N_0(\1).
\end{equation}
\end{lemma}
\begin{proof}
Using \eqref{dNjdiff} with $\eta=\bxi^N_{u_N}$ and
$\xi=\xi^N_{u_N}$ and in \eqref{dNjsdef} we see that
\begin{align*} \Big|d^{N,j}(s,\bxi^N_{u_N},\Phi)
-d^{N,j}(s,\xi^N_{u_N},\Phi)\Big| &\le 2\|r\|\|\Phi\|_\infty
\frac{\ell^{(j)}_N}{N'} \sum_{x\in\SN}\Big[\sum_{y\in
x+\bar\cN_N}\bxi^{N}_{u_N}(y) -\xi^N_{u_N}(y)\Big]\\
&=2\|r\||\bar\cN|\|\Phi\|_\infty \ell^{(j)}_N\big[\bar
X^N_{u_N}(\1)- X^N_{u_N}(\1)\big].
\end{align*}
Similarly,
\begin{align*}
\Big|d^{N,j}(s,\bxi^N_{u_N},\Phi)
-d^{N,j}(s,\xi^{N,\VM}_{u_N},\Phi)\Big|\le 
2\|r\||\bar\cN|\|\Phi\|_\infty
\ell^{(j)}_N\big[\bar X^N_{u_N}(\1)- X^N_{u_N}(\1)\big].
\end{align*}
It follows that the left-hand side of
\eqref{dvmcomp} is bounded by
\[
2\|r\||\bar\cN| \|\Phi\|_\infty
(\log N)^3E\big[(\bar X^N_{u_N}(\1)-
{X}^N_{u_N}(\1))+(\bar X^N_{u_N}(\1)-
{X}^{N,\VM}_{u_N}(\1))\big].
\]
Now Lemma~\ref{l:masscmps} completes the proof.
\end{proof}

The next step is to consider
$E[d^{N,j}(s,\xi^{N,\VM}_{u_N},\Phi)]$.  Under a probability
$\widehat P$, let $\{B^{N,x}_t,x\in\SN\}$ be a rate $w_NN$,
coalescing random walk system on $\SN$ with jump kernel
$p_N$ and for $A\subset \SN$ let $B^{N,A}_t=\cup_{x\in
A}\{B^{N,x}_t\}.$ For finite nonempty disjoint $A_i\subset
\SN$ define $\sigma^N(A_1,\dots,A_n)$ and $\tau^N(A_1,\dots,
A_n)$ in the same way as $\sigma$ and $\tau$ are defined in
\eqref{sigtau}, but with $\{B^{N,x}:x\in\SN\}$ in place of
$\{B^x:x\in\Z^2\}$.  Similarly define $\sigma^N(x_1,\dots
x_n)$ and $\tau^N(x_1,\dots,x_n)$ for distinct $x_i$ in
$\SN$. If $x\in\SN$, then introduce
$\sigma^N_x(A_1,\dots,A_n)=\sigma^N(x+A_1,\dots,x+A_n)$, and
similarly for $\tau_x^N(A_1,\dots,A_n)$.

The duality equation connecting the voter model
$\xi^{N,\VM}$ with the coalescing random walks
$\{B^{N,x}_t\}$ that we need is the following (see Section
III.4 of Lig85). For $x\in\SN$, finite disjoint sets
$A,B\subset\SN$ and $\xi^N_0\in\{0,1\}^{\SN}$,
\begin{equation}\label{dual1}
E_{\xi^N_0}\Big[\prod_{a\in A } \xi^{N,\VM}_t(a)\,
\prod_{b \in B } (1-\xi^{N,\VM}_t(b))\Big] =
\widehat E\Big[\prod_{a\in A } \xi^N_0(B^{N,a}_{t}) 
\prod_{b \in B } (1-\xi^N_0(B^{N,b}_{t}))\Big].
\end{equation}
For $A\subset\cN_N$ define
\begin{equation}\label{IN+-def}
\begin{aligned}
I^{N,+}(x,u_N,A,\xi^N_0)
&=\prod_{a\in A } \xi^N_0(B^{N,x+a}_{u_N}) 
\prod_{b \in \bar\cN_N\setminus A } (1-\xi^N_0(B^{N,x+b}_{u_N}))\\
I^{N,-}(x,u_N,A,\xi^N_0)&
=\prod_{a\in \bar\cN_N\setminus A  } \xi^N_0(B^{N,x+a}_{u_N}) 
\prod_{b \in A } (1-\xi^N_0(B^{N,x+b}_{u_N})).
\end{aligned}
\end{equation}
With this notation, \eqref{dNjrep} and \eqref{dual1} imply
that
\begin{equation}\label{dNjvm}
\begin{aligned}
E_{\xi_0^N}\big(d^{N,2}(x,\xi^{N,\VM}_{u_N})\big)
&=  \sum_{\emptyset\neq A\subset \cN_N} r^{N,a}\sqrt N A)
\widehat E\big(I^{N,+}(x,u_N,A,\xi^N_0)\big),\\
E_{\xi_0^N}\big(d^{N,3}(x,\xi^{N,\VM}_{u_N})\big)
&= \sum_{\emptyset\neq A\subset \cN_N} r^{N,s}(\sqrt N A)\Big[
\widehat E\big(I^{N,+}(x,u_N,A,\xi^N_0)\big) -
\widehat E\big(I^{N,-}(x,u_N,A,\xi^N_0)\big)\Big]. 
\end{aligned}
\end{equation}
If we now define
\begin{align}\label{^HN2}
\hat H^{N,2}(\xi^N_0,u_N,\Phi_s) &=
\frac{\ell^{(2)}_N}{N'}\sum_{x\in \SN}\Phi(s,x)
\sum_{\emptyset\neq A\subset \cN_N} r^{N,a}(\sqrt NA)
\widehat E[I^{N,+}(x,u_N,A,\xi^N_0)]\\
\label{^HN3}\hat H^{N,3}(\xi^N_0,u_N,\Phi_s) &=
\frac{\ell^{(3)}_N}{N'}\sum_{x\in \SN}\Phi(s,x)
\sum_{\emptyset\neq A\subset \cN_N} r^{N,s}(\sqrt NA)
\Big(\widehat E[I^{N,+}(x,u_N,A,\xi^N_0)]
\\\notag &\phantom{=
\frac{\ell^{(3)}_N}{N'}\sum_{x\in \SN}\Phi(s,x)
\sum_{A\subset \cN_N} r^{N,s}({\sqrt N}A)
\Big(}-
\widehat E[I^{N,-}(x,u_N,A,\xi^N_0)]\Big), 
\end{align}
then for $j=2,3$, %by \eqref{dNjdef} and \eqref{dual2}, 
\begin{equation}\label{dN3rep1}
E_{\xi^N_0}[d^{N,j}(s,\xi^{N,\VM}_{u_N},\Phi)] =
\hat H^{N,j}(\xi^N_0,u_N,\Phi_s).
\end{equation}
\begin{lemma}\label{l:d3H} There is a constant $C_{\ref{d3H}}$
such for $j=2,3$, all $T>0$, $\Phi\in C_b([0,T]\times\R^2)$
and all $s\in[u_N,T]$, 
\begin{equation}\label{d3H}
\Big|E_{\xi^N_0}\big(d^{N,j}(s,\xi^N_{s},\Phi)|\cF^N_{s-u_N}\big) -
\hat H^{N,j}(\xi^N_{s-u_N},u_N,\Phi_{s-u_N}) \Big|
\le C_{\ref{d3H}}\|\Phi\|_{1/2,N}(\log
N)^{3-\frac{p}{2}}X^N_{s-u_N}(\1). 
\end{equation}
\end{lemma}
\begin{proof} 
By the Markov property, Lemma~\ref{l:dvmcomp} and
\eqref{dN3rep1}, the left-hand side of \eqref{d3H} is
bounded above by
\begin{equation}\label{d3H1}
C_{\ref{dvmcomp}}\|\Phi\|_\infty(\log
N)^{6-p}X^N_{s-u_N}(\1) + \Big|
\hat H^{N,j}(\xi^N_{s-u_N},u_N,\Phi_{s}) - 
\hat H^{N,j}(\xi^N_{s-u_N},u_N,\Phi_{s-u_N})
\big]\Big|.
\end{equation}
Use \eqref{dN3rep1} and then \eqref{dNjbnd1} with the voter
duality \eqref{dual1} to see that the second term above is
bounded by
\begin{align*}
\|r\|\frac{\ell^{(j)}_N}{N'}\sum_{x\in\SN}&|\Phi(s,x)-\Phi(s-u_N,x)|
\widehat E\Big[\sum_{y\in \bar\cN_N}\xi^N_{s-u_N}(B^{N,x+y}_{u_N})\Big]
\\
&\le \|r\||\Phi|_{1/2}\sqrt u_N\ell^{(j)}_N \sum_{y\in \bar\cN_N} 
\widehat E\Bigl[\frac{1}{N'}\sum_{x\in\SN}
\xi^N_{s-u_N}(x+B^{N,y}_{u_N})\Bigr]\\
&= \|r\||\Phi|_{1/2}|\bar\cN|(\log N)^{3-\frac{p}{2}}X^N_{s-u_N}(\1). 
\end{align*}
Inserting this bound into \eqref{d3H1} completes the proof.
\end{proof}

We next further decompose the $\hat H^{N,j}$.  Consider
$\emptyset\ne A\subset\cN_N$. If
$I^{N,\pm}(x,u_N,A,\xi^N_0)\ne 0$ then necessarily $2\le
|B^{N,x+\bar \cN_N}_{u_N}|\le |\bar\cN|$, and by defining
\begin{equation}\label{INi}
I^{N,\pm}_i(x,u_N,A,\xi^N_0)= 
I^{N,\pm}(x,u_N,A,\xi^N_0)1\{|B^{N,x+\bar \cN_N}_{u_N}|=i\},
\end{equation}
we can write
\begin{equation}\label{INrep}
I^{N,\pm}(x,u_N,A,\xi^N_0) = \sum_{i=2}^{|\bar\cN|}
I^{N,\pm}_i(x,u_N,A,\xi^N_0).
\end{equation}
Letting
\begin{equation}\label{hNiAdef}
h^{N,\pm}_{i,j}(\xi^N_0,u_N,A,\Phi_s) =
\frac{\ell^{(j)}_N}{N'} \sum_{x\in\SN}\Phi(s,x) 
I^{N,\pm}_i(x,u_N,A,\xi^N_0)
\end{equation}
and
\begin{equation}
\label{^HNidef}
\begin{aligned}
\hat H^{N,2}_i(\xi^N_0,u_N,\Phi_s) &=
\sum_{\emptyset\neq A\subset \cN_N} r^{N,a}(\sqrt NA)\widehat E\Big(
h_{i,2}^{N,+}(\xi^N_0,u_N,A,\Phi_s)\Big)\\
\hat H^{N,3}_i(\xi^N_0,u_N,\Phi_s) &=
\sum_{\emptyset\neq A\subset \cN_N} r^{N,s}(\sqrt NA)\widehat E\Big(
h_{i,3}^{N,+}(\xi^N_0,u_N,A,\Phi_s)
- h_{i,3}^{N,-}(\xi^N_0,u_N,A,\Phi_s)\Big),
\end{aligned}
\end{equation}
we obtain the decomposition
\begin{equation}
\label{HNrep}
\hat H^{N,j}(\xi^N_0,u_N,\Phi_s) =
\sum_{i=2}^{|\bar\cN|} \hat H^{N,j}_i(\xi^N_0,u_N,\Phi_s).
\end{equation}

We will now obtain simple bounds on the summands in
\eqref{HNrep}, leaving a more detailed analysis of the main
terms $\hat H^{N,2}_2(\xi^N_0,u_N,\Phi_s)$ and $\hat
H^{N,3}_3(\xi^N_0,u_N,\Phi_s)$ to
Section~\ref{ss:exacd3asymp}.

\begin{lemma}\label{l:dN3bnd} There are constants
$C_{\ref{dN2i}}$, $C_{\ref{dN3i}}$ such that for any $T>0$,
$\phi\in C_b([0,T]\times\R^2)$ and $s\in [0,T]$,
\begin{align} \Large|\hat
H^{N,2}_i(\xi^N_0,u_N,\Phi_s)\Large|&\le
C_{\ref{dN2i}}\|r\|\|\Phi_s\|_\infty (\log N)^{1-\binom
i2}X^N_0(\1), \quad 2\le i\le |\bar\cN|,
\label{dN2i}\\ \Large|\hat
H^{N,3}_i(\xi^N_0,u_N,\Phi_s)\Large|&\le
C_{\ref{dN3i}}\|r\|\|\Phi_s\|_\infty (\log N)^{3-\binom
i2}X^N_0(\1), \quad 3\le i\le |\bar\cN|,
\label{dN3i} \intertext{and}\label{dN32} |\hat
H^{N,3}_2(\xi^N_0,u_N,\Phi_s)| &\le C_{\ref{dN3i}}\|r\|
\|\Phi_s\|_{\Lip} (\log N)^{(6-p)/2}X^N_0(\1).
\end{align}
\end{lemma}
\begin{proof}
We start with the observation that for $\emptyset\neq A\subset\cN_N$ and
any $a\in A$, 
\begin{equation}\label{INpmbnd}
|I^{N,\pm}_i(x,u_N,A,\xi_0^N)|
\le 
\xi^N_0(B^{N,x+a}_{u_N})1\{|B_{u_N}^{N,x+\bar\cN_N}|=i\} +
\xi^N_0(B^{N,x}_{u_N})1\{|B_{u_N}^{N,x+\bar\cN_N}|=i\}.
\end{equation}
We will bound $E\big(|h^{N,\pm}_{i,j}| \big)$ with this
inequality.  By translation invariance,
\begin{align*}
\frac{\ell^{(j)}_N}{N'}&\sum_{x\in\SN}|\Phi(s,x)|
\widehat E\big(\xi^N_0(B^{N,x+a}_{u_N}) 1\{
|B_{u_N}^{N,x+\bar\cN_N}|=i\} \big)\\
&= 
\frac{\ell^{(j)}_N}{N'}\sum_{x\in\SN}\sum_{w\in\SN}|\Phi(s,x)|
\xi^N_0(x+w) \widehat P\big(B^{N,x+a}_{u_N}=x+w,
|B_{u_N}^{N,x+\bar\cN_N}|=i \big)  \\
&= 
\frac{\ell^{(j)}_N}{N'}\sum_{x\in\SN}\sum_{w\in\SN}|\Phi(s,x)|
\xi^N_0(x+w) \widehat P\big(B^{N,a}_{u_N}=w,
|B_{u_N}^{N,\bar\cN_N}|=i \big)  \\
&\le \|\Phi_s\|_{\infty} 
X^N_0(\1) \ell^{(j)}_N \widehat P\big(|B_{u_N}^{N,\bar\cN_N}|=i \big). 
\end{align*}
The same bound holds if we replace $B^{N,x+a}_{u_N}$ with
$B^{N,x}_{u_N}$, and thus we have shown that
\begin{equation}\label{hNijbnd}
\widehat E\big( |h^{N,\pm}_{ij}(\xi^N_0,u_N,A,\Phi_s)|\big) \le
2\|\Phi_s\|_\infty X^N_0(\1)\ell^{(j)}_N \widehat
P\big(|B_{u_N}^{N,\bar\cN_N}|=i \big)  .
\end{equation}
By Remark~\ref{r:coalbnds}, the lower bound on $u_N$ in
\eqref{uNdef}, and the fact that $w_N\ge 1/2$ for $N$ large,
\[
\ell^{(j)}_N \widehat P\big(|B_{u_N}^{N,\bar\cN_N}|=i \big) =
\begin{cases}
O((\log N)^{1-\binom i2}),& j=2,\\
O((\log N)^{3-\binom i2}),& j=3.
\end{cases}
\]
Using this bound in \eqref{hNijbnd} above, and substituting into
\eqref{^HNidef} we obtain \eqref{dN2i} and \eqref{dN3i}.

The proof of \eqref{dN32} is more delicate, as it relies on
cancellation.  For $x\in S_N$ define
\begin{equation}\label{OmegaN}
\Omega^N_x(A) = \{\sigma^N_x(A,\bar\cN_N\setminus A)>u_N,
\tau^N_x(A,\bar\cN_N\setminus A)<u_N\}.
\end{equation}
If  $|B^{N,\bar\cN_N}_{u_N}|=2$, and $B^{N,x+A}_{u_N}$ and 
$B^{N,x+\bar\cN_N\setminus A}_{u_N}$ are disjoint,
then 
$B^{N,x+\bar\cN_N\setminus A}_{u_N}=B^{N,x}_{u_N}$ and 
for any $a\in A$, $B^{N,x+A}_{u_N}=B^{N,x+a}_{u_N}$. 
Thus
\begin{align*}
I^{N,+}_2(x,u_N,A,\xi_0^N) - &I^{N,-}_2(x,u_N,A,\xi_0^N)\\
&=
\Big(\xi^N_0(B^{N,x+a}_{u_N})(1-\xi^N_0(
B^{N,x}_{u_N})) -
(1-\xi^N_0(B^{N,x+a}_{u_N}))\xi^N_0(
B^{N,x}_{u_N})\Big)
1\{\Omega^N_x(A)\}\\
&=
\Big(\xi^N_0(B^{N,x+a}_{u_N})
-\xi^N_0(B^{N,x}_{u_N}\Big)
1\{\Omega^N_x(A)\}\\
&=
\sum_{w\in\SN}\xi^N_0(w)1\{B^{N,x+a}_{u_N}=w\}1\{\Omega^N_x(A)\}
-\sum_{w\in\SN}\xi^N_0(w)1\{B^{N,x}_{u_N}=w\}
1\{\Omega^N_x(A)\}.
\end{align*}
Now by translation invariance,
\begin{multline*}
\widehat E[I^{N,+}_2(x,u_N,A,\xi_0^N) - I^{N,-}_2(x,u_N,A,\xi_0^N)]\\
\\=\sum_{w\in\SN}\xi^N_0(w) \widehat E[1\{B^{N,a}_{u_N}=w-x\}1\{\Omega^N_0(A)\}]
-\sum_{w\in\SN}\xi^N_0(w) \widehat E[1\{B^{N,0}_{u_N}=w-x\}1\{\Omega^N_0(A)\}].
\end{multline*}
Plugging into the definition of 
$h^{N,\pm}_{2,3}(\xi^N_0,u_N,A,\Phi_s)$ gives
\begin{align*}
\widehat E\big((h^{N,+}_{2,3}-h^{N,-}_{2,3})(\xi^N_0,u_N,A,\Phi_s)\big)
&= \frac{(\log
N)^3}{N'}\sum_{w\in\SN}\xi^N_0(w) \sum_{x\in \SN}\Phi(s,x)
\Big[\widehat E\big(1\{B^{N,a}_{u_N}=w-x\}1\{\Omega^N_0(A)\}\big)\\
&\qquad\qquad\qquad\qquad -\widehat E\big(1\{B^{N,0}_{u_N}=w-x\}1\{\Omega^N_0(A)\}\big)\Big]\\
&= \frac{(\log
N)^3}{N'}\sum_{w\in\SN}\xi^N_0(w)
\widehat E\Big[(\Phi(s,w-B^{N,a}_{u_N}) -
\Phi(s,w-B^{N,0}_{u_N}))1\{\Omega^N_0(A)\}\Big].
\end{align*}
By the above, 
\begin{align}\notag
\Big|\widehat E\big((h^{N,+}_{2,3}-h^{N,-}_{2,3})(\xi^N_0,u_N,A,\Phi_s)\big)]\Big|
&\le 
\frac{(\log N)^3}{N'}\sum_{w\in \SN}
\xi^N_0(w)\widehat E\big(\big|\Phi(s, w-B^{N,a}_{u_N})
- \Phi(s, w-B^{N,0}_{u_N})\big|\big)\\
\notag &\le \|\Phi_s\|_\Lip
\frac{(\log N)^3}{N'}\sum_{w}\xi^N_0(w)
\widehat E\big(|B^{N,a}_{u_N}-
B^{N,0}_{u_N}|\big)\\
\notag &= \|\Phi_s\|_{\Lip} X^N_0(\1) (\log N)^3
\widehat E\big(|a|+|B^{N,0}_{2u_N}|\big)\\
\notag &\le \|\Phi_s\|_{\Lip} X^N_0(\1) (\log N)^3
\big(c/\sqrt N +\sqrt {\sigma^2 4u_N} \big)\\
&\le C(\cN,\sigma^2)\|\Phi_s\|_\Lip X^N_0(\1) (\log N)^{(6-p)/2},
\label{lipbnd}
\end{align}
where $c=\max_{e\in\cN}|e|$ and we recall from \eqref{covp} that 
$\widehat E\big(|B^{N,0}_s|^2\big)=2w_N\sigma^2 s\le 2\sigma^2 s$.  Using this bound in
\eqref{^HNidef} we obtain  \eqref{dN32}. 
\end{proof}

In order to use the above results to effectively handle the
drift terms $d^{N,2}(x,\xi^N_s,\Phi)$ and
$d^{N,3}(x,\xi^N_s,\Phi)$ we must first obtain bounds on the
first and second moments of the total mass, which will play
an important roll in what follows. Therefore we interrupt
our current analysis to handle the total mass next.
  
\subsection{Total mass bounds}

We now introduce a particular choice of $u_N$, namely
\begin{equation}\label{tNdef}
t_N = (\log N)^{-19}.
\end{equation}
\begin{lemma}\label{dn3ubnds} There is a constant
$C_{\ref{dN3cnd1}}>0$ so that for  $j=2,3$, all $T>0$ and
$\Phi\in C_b([0,T]\times\R^2)$, 
\begin{align}\label{crude1}
|d^{N,j}(x,\xi^N_s,\Phi)|\le \|r\| \|\Phi\|_\infty
|\bar\cN|\ell_N^{(j)} X^N_s(\1)\quad \forall
s\in[0,T]
\end{align}
and
\begin{equation}\label{dN3cnd1}
\Big|E[d^{N,j}(s,\xi^N_{s},\Phi)|\cF^N_{s-t_N}]\Big|
\le C_{\ref{dN3cnd1}} \|\Phi\|X^N_{s-t_N}(\1)\quad \forall
s\in[t_N,T]. 
\end{equation}
\end{lemma}
\begin{proof} \eqref{crude1} holds by \eqref{dNjsbnd}, while
\eqref{dN3cnd1} follows from Lemma~\ref{l:d3H},
\eqref{HNrep}, and Lemma~\ref{l:dN3bnd}.
\end{proof}

\begin{proposition}\label{p:mobnds} There exists a
$c_{\ref{mobnds}}>0$, and for $T>0$ a constant
$C_{\ref{mobnds}}>0$ depending on $T$, such that for any $t\le T$,
\begin{equation}\label{mobnds}
\begin{aligned}
(a)&\qquad\qquad   E[X^N_t(\1)]\le (1 + C_{\ref{mobnds}}(\log N)^{-16})X^N_0(\1)
\exp(c_{\ref{mobnds}}t),\\
(b)&\qquad\qquad E[(X^N_t(\1))^2]\le C_{\ref{mobnds}}(X^N_0(\1)+(X^N_0(\1))^2).
\end{aligned}
\end{equation}
\end{proposition}
Therefore, for $T>0$ there is a constant $C_{\ref{stbnd}}>0$,
depending on $T$, such that for all $s,t\in[0,T]$,
\begin{equation}\label{stbnd}
E[X^N_s(\1)X^N_t(\1)] \le C_{\ref{stbnd}}(X^N_0(\1)+(X^N_0(\1))^2). 
\end{equation}

\noindent{\sl{Proof of (a).}}
By \eqref{SMG1},
\begin{equation}\label{DN3bnd}
E[X^N_t(\1)] = X^N_0(\1)  +
\int_0^{t\wedge t_N}\sum_{j=2}^3E[d^{N,j}(s,\xi^N_s,\1)]ds
+ \int_{t\wedge
  t_N}^{t}\sum_{j=2}^3E\Big[E[d^{N,j}(s,
\xi^N_{s},\1)|\cF^N_{s-t_N}]\Big]ds. 
\end{equation}
Use the two bounds from Lemma~\ref{dn3ubnds} (with $\Phi=1$)
and $\ell^{(j)}_N\le (\log N)^3$ in \eqref{DN3bnd} to get
\begin{equation}\label{XN1bnd}
E[X^N_t(\1)] \le X^N_0(\1)  +
2\|r\||\bar\cN|(\log N)^3\int_0^{t\wedge t_N}E[X^N_s(\1)]ds
+ C_{\ref{dN3cnd1}} \int_{0}^{(t-t_N)^+}E[X^N_s(\1)]ds.
\end{equation}
This implies that for $t\le t_N$,
\begin{equation}\label{XN1bnd1}
E[X^N_t(\1)] \le X^N_0(\1)  +
2\|r\||\bar\cN|(\log N)^3\int_0^{t}E[X^N_s(\1)]ds,
\end{equation}
and thus by Gronwall's inequality, for $t\le t_N$, 
\begin{equation}\label{XN1bnd2}
E[X^N_t(\1)] \le \exp(2\|r\||\bar\cN|(\log N)^3t)X^N_0(\1)
\le \exp(2\|r\||\bar\cN|(\log N)^{-16})X^N_0(\1) \le
e^{2\|r\||\bar\cN|}X^N_0(\1).
\end{equation}
By plugging this bound into \eqref{XN1bnd} we obtain for all
$t>0$, (recall $N\ge3$)
\begin{equation*}
E[X^N_t(\1)] \le X^N_0(\1) +
e^{2\|r\||\bar\cN|}2\|r\||\bar\cN|(\log N)^{-16}X^N_0(\1) +
C_{\ref{dN3cnd1}}\int_{0}^{t}E[X^N_s(\1)]ds.
\end{equation*}
Another use of Gronwall's inequality completes the proof of part (a).
\qed

\medskip

Before proving (b) we establish some preparatory results
which will also be useful later. Let
\begin{equation}\label{KNdefn}K_N=\sum_{y\in\N_N}p_N(y)\log
N\widehat P(\sigma^N(0,y)>t_N).\end{equation}
By Lemma~A.3(ii) in \cite{CDP00}, 
\begin{equation}\label{sqfnconst}
\lim_{N\to\infty}K_N=2\pi\sigma^2.
\end{equation}	
From \eqref{SMG1} we have 
\begin{equation}\label{XN^2}
E[X^N_t(\1)^2] \le 4 \Big[
X^N_0(\1)^2  + E[\langle M^N(\1)\rangle_t] +
E[D^{N,2}_t(\1)^2]+ E[D^{N,3}_t(\1)^2] \Big].
\end{equation}
Recall from \eqref{SMGsqfn} that $\langle M^N(\Phi)\rangle_t$ = $\langle
M^N(\Phi)\rangle_{1,t} + \langle M^N(\Phi)\rangle_{2,t}$.
\begin{lemma}\label{l:oldl5.1} For any $T>0$ there is a
constant $C_{\ref{CMPL51b}}>0$, depending on $T$, so that
for any bounded Borel function $\Phi$ on $[0,T]\times\R^2$
and all $t\in[0,T]$, 
\begin{align}
\notag
\langle M^N(\Phi)\rangle_{1,t} &= 2\int_0^t
X^N_s((\log N)\Phi_s^2f^{(N)}_0(\cdot,\xi^N_s))ds +
\int_0^t \tilde m^N_{1,s}(\Phi)ds\\
\label{CMPL51b}
&\qquad\qquad\text{ where }
\ds |\tilde m^N_{1,s}(\Phi)| \le C_{\ref{CMPL51b}}
\frac{\|\Phi\|^2_\Lip (\log
  N)}{\sqrt N} X^N_s(\1),
\intertext{and }
\label{CMPL51a}
\langle M^N(\Phi)\rangle_{2,t}&=\int_0^t\tilde m^N_{2,s}(\Phi)ds,\ \text{where } |\tilde m^N_{2,s}(\Phi)|\le 
2\|r\||\bar\cN|\|\Phi\|_\infty^2\frac{\ell^{(3)}_N}{N'}
X^N_s(\1) ds.
\end{align}
\end{lemma}
\begin{proof}
\eqref{CMPL51b} holds by the corresponding result in
Lemma~4.8 of \cite{CP08}. Note that the result there is
actually an identity and bound for a generic state
$\xi\in\{0,1\}^\SN$ (although it is not stated as such) and
applies immediately to our setting as well.  For
\eqref{CMPL51a} note first that by their definitions,
$|c^{N,a}(x,\xi)|=|d^{N,2}(x,\xi)|$ and
$|c^{N,s}(x,\xi)|=|d^{N,3}(x,\xi)|$. The second result now
follows immediately by using \eqref{dNjsbnd} to bound the
absolute value of the integrand of the integrals in
\eqref{SMG2} for $\langle M^N(\Phi)\rangle_{2,t}$.
\end{proof}
\begin{lemma}\label{l:oldl5.2}
There is a $C_{\ref{CMPL52}}>0$ so that for any bounded
Borel $\Phi$ on $\R^2$:\\
\noindent(a)\begin{equation}\label{usefulbnd} \frac{\log
N}{N'}\sum_{x\in\SN}\sum_{y\in\cN_N}\Phi(x)p_N(y)\hat
E(\xi_0^N(B^{N,x}_{t_N})1(\sigma^N(x,x+y)>t_N))=K_NX_0^N(\Phi)+\cE_N,
\end{equation}
where
\begin{equation}\label{usefulbndb}
|\cE_N|\le C_{\ref{CMPL52}}|\Phi|_\Lip(\log N)^{-17/2}X_0^N(\1).
\end{equation}

\noindent (b) If, in addition, $\Phi\ge 0$, then
\begin{equation}
\label{CMPL52}
E\Big[X^N_{t_N}((\log N)\Phi f^{(N)}_0(\cdot,\xi^N_{t_N}))\Big]
\le C_{\ref{CMPL52}} \Big(\|\Phi\|_\Lip (\log N)^{-17/2}X^N_0(\1)
+ X^N_0(\Phi)\Big).
\end{equation}
\end{lemma}
\begin{proof} (a) Let $\Sigma_N$ denote the left-hand side
of \eqref{usefulbnd}.  We may write $B^{N,x}$ as $x+B^{N,0}$
and sum over the values of $B^{N,0}_{t_N}$ to see that
\begin{align*} \Sigma_N&=\frac{\log
N}{N'}\sum_{x\in\SN}\sum_{w\in\SN}\sum_{y\in\cN_N}p_N(y)[\Phi(x)-\Phi(x+w)]\xi_0^N(x+w)\hat
P(B_{t_N}^{N,0}=w,\sigma^N(0,y)>t_N)\\ &\ \ +\frac{\log
N}{N'}\sum_{x\in\SN}\sum_{w\in\SN}\Phi(x+w)\xi_0^N(x+w)\sum_{y\in\cN_N}p_N(y)\hat
P(B_{t_N}^{N,0}=w,\sigma^N(0,y)>t_N)\\
&:=\Sigma_N^{(1)}+\Sigma_N^{(2)}.
\end{align*}
Use the implicit Lipschitz continuity of $\Phi$ to see that
\begin{align} \nonumber |\Sigma_N^{(1)}|&\le
|\Phi|_\Lip\sum_{w\in\SN}|w|\Bigl(\frac{\log
N}{N'}\sum_{x\in\SN}\xi^N_0(x+w)\Bigr)\sum_{y\in\cN_N}p_N(y)\hat
P(B^{N,0}_{t_N}=w,\sigma^N(0,y)>t_N)\\
\label{err3x}&\le |\Phi|_\Lip X_0^N(\1)\log N\hat
E(|B^{N,0}_{t_N}|)\le |\Phi|_\Lip X_0^N(\1)\log N\sqrt
2\sigma\sqrt{t_N}.
\end{align} The fact that $w_N\le 1$ is used in the last
line as well. Sum over $x$ first and then $w$, to see that
$\Sigma^{(2)}_N=K_NX_0^N(\Phi)$. This gives (a).

(b) Assume now $\Phi\ge 0$.  We will compare with the
corresponding expression for $\xi^{N,\VM}_{t_N}$, and use
the duality \eqref{dual1} to compute the latter, and hence
see that the left-hand side of \eqref{CMPL52} equals
\begin{multline}\label{CMPL52a}
\frac{\log N}{N'}
\sum_{x\in \SN}\sum_{y\in\cN}\Phi(x)p_N(y)
E_{\xi^N_0}[\xi^N_{t_N}(x)\hxi^N_{t_N}(x+y)
-\xi^{N,\VM}_{t_N}(x)\hxi^{N,\VM}_{t_N}(x+y)]\\
+\frac{\log N}{N'}
\sum_{x\in \SN}\sum_{y\in\cN_N}\Phi(x)p_N(y)
\hat
E[\xi^N_0(B^{N,x}_{t_N})\hxi^N_0(B^{N,x+y}_{t_N})]. 
\end{multline}
By the coupling \eqref{couple1}, Lemma~\ref{l:masscmps}, and
the triangle inequality, the absolute value of the first sum
in \eqref{CMPL52a} is bounded by
\begin{align}\notag
\|\Phi\|_\infty\frac{\log N}{N'}
\sum_{x\in \SN}\sum_{y\in\cN}p_N(y) & \Big(
E[\bar\xi^N_{t_N}(x)-{\xi}^{N}_{t_N}(x)]+E[\bar\xi^N_{t_N}(x)-\xi^{N,\VM}_{t_N}(x)]\\
&
\notag+E[\bar \xi^N_{t_N}(x+y)
-{\xi}^{N}_{t_N}(x+y)]+E[\bar \xi^N_{t_N}(x+y)
-{\xi}^{N.\VM}_{t_N}(x+y)]\Bigr) \\
\notag &=
2\|\Phi\|_\infty(\log N) \Bigl(E[\bar X^N_{t_N}(\1) -
{X}^N_{t_N}(\1)]+E[\bar X^N_{t_N}(\1) -{X}^{N,\VM}_{t_N}(\1)]\Bigr)\\
&\le C\|\Phi\|_\infty (\log N)^{-15}X^N_0(\1).
\label{CMPL52b}
\end{align}
The second sum in \eqref{CMPL52a} is bounded by $\Sigma_N$
(the left-hand side of part (a)), and so (a) and
\eqref{sqfnconst} give the required bound. 	 
\end{proof}

\noindent{\sl{Proof of
    Proposition~\ref{p:mobnds}(b).}}  It now
follows quite easily from Lemmas~\ref{l:oldl5.1} and
\ref{l:oldl5.2}(b), as well as part (a), that there is  
a constant $C_{\ref{qvbnd1}}$ depending on $T$ such that,
\begin{equation}\label{qvbnd1}
E[\langle M^N(\1)\rangle_T] \le C_{\ref{qvbnd1}} X^N_0(\1).
\end{equation}
Here one uses the Markov property and
Lemma~\ref{l:oldl5.2}(b) to bound $\int_{t_N}^{T}
X^N_s((\log N)f^{(N)}_0(\xi^N_s))ds$, and the obvious crude
bound on the integrand to handle the integral from $0$ to
$t_N$.

Turning to the last terms in \eqref{XN^2}, we consider more
generally $E[(D^{N,j}_t(\Phi))^2]$, where $\Phi:\R^2\to\R$
is Lipschitz continuous.   We claim that for
$j=2,3$,
\begin{equation}\label{DNtN}
E[ (D^{N,j}_{t\wedge t_N}(\Phi))^2]
\le \|\Phi\|^2_\infty|\bar \cN|^2 (\log N)^6t_N
E\Big[\int_0^{t\wedge t_N}X^N_s(\1)^2ds \Big],
\end{equation}
and there is a constant $C_{\ref{DNt1t2}}>0$ such that if 
$t_2>t_1\ge t_N$, then
\begin{multline}\label{DNt1t2}
E[
(D^{N,j}_{t_2}(\Phi) -
D^{N,j}_{t_1}(\Phi))^2] 
\\ \le 
C_{\ref{DNt1t2}}\Big( \|\Phi\|_\Lip^2(t_2-t_1)+ \|\Phi\|_\infty^2(\log
N)^6 (t_N\wedge (t_2-t_1))\Big) 
\int_{t_1}^{t_2} E[X^N_s(\1)^2] ds .
\end{multline}
The proofs of the corresponding inequalities (66) and (69)
in \cite{CMP} apply here without change if (41) and (48)
there are replaced by \eqref{crude1} and \eqref{dN3cnd1}
here, using $\ell^{(2)}_N=\log N\le (\log N)^3$. (The
increment bound is stronger than we need here but will be
useful later in establishing tightness.) Now use
\eqref{qvbnd1}, \eqref{DNtN} and \eqref{DNt1t2} in
\eqref{XN^2} to see there is a constant $C_{\ref{XN2(1)}}$
depending on $T$ such that if $t\le T$ then
\begin{equation}\label{XN2(1)} 
E[
X^N_t(\1)^2] \le
C_{\ref{XN2(1)}}\Big[ X^N_0(\1)^2+ X^N_0(\1) + 
\int_0^{t}E[X^N_{s}(\1)^2]ds \Big].
\end{equation}
A simple Gronwall argument finishes the proof of
Proposition~\ref{p:mobnds} (b). 
\qed

\subsection{Exact drift asymptotics and first moment measure bounds}\label{ss:exacd3asymp}	

Introduce, for $\delta>0$,
\begin{equation}\label{calNdef}
\scrI^N(\delta,\xi_0^N)=\int\int 1_{\{|w-z|<\delta\}}
dX_0^N(w)dX_0^N(z),
\end{equation}
and define 
\begin{equation}
\scrI^N(\xi_0^N)=\scrI^N\Big(\sqrt{t_N(\log
N)^{1/2}},\xi^N_0\Big).
\end{equation}

\begin{lemma}\label{p:2walks} There is a 
$C_{\ref{2walks}}$ so that for $N\ge N(\vep_0)$, distinct
$a_0, a_1, a_2\in 
\bar{\cN}_N$, and $\xi^N_0\in\{0,1\}^{S_N}$,
\begin{align}\label{2walks}
\nonumber \frac{(\log N)^3}{N'}&\sum_{x\in \SN}\widehat P\Big(
\xi^N_0(B^{N,x+a_1}_{t_N}) = \xi^N_0(B^{N,x+a_2}_{t_N}) =1,
\sigma^N_x(a_0,a_1,a_2)>t_N\Big)\\
&\le
C_{\ref{2walks}}\Bigl(\frac{1}{t_N\log
  N}\scrI^N(\xi^N_0)+(\log
N)^{-(1/2)}X^N_0(\1)\Bigr),
\end{align}
and
\begin{align}\label{2walksM}
\nonumber \frac{(\log N)}{N'}&\sum_{x\in \SN}\widehat P\Big(
\xi^N_0(B^{N,x+a_1}_{t_N}) = \xi^N_0(B^{N,x+a_2}_{t_N}) =1,
\sigma^N_x(a_1,a_2)>t_N\Big)\\
&\le
C_{\ref{2walks}}\Bigl(\frac{1}{t_N\log
  N}\scrI^N(\xi^N_0)+(\log
N)^{-(1/2)}X^N_0(\1)\Bigr).
\end{align}
\end{lemma}
\begin{proof} The proof of \eqref{2walks} may be found in
the derivation of (46) in \cite{CMP} with $\Phi=\1$. See, in
particular, the bound on $\mathcal{T}_N(\Phi)$ in that proof
on pages 1214 and 1215, and note we are setting $\eta=1/2$
in the final display of Section~4.2 of that article.  In
that result the stochastic process $\xi^N_u(\1)$ is a
rescaled two-dimensional Lotka-Volterra model but may be
replaced in the proof by a fixed point
$\xi^N_0\in\{0,1\}^{S_N}$ and so holds equally well in our
setting.  The derivation of \eqref{2walksM} is the same. One
just replaces $\sigma_x^N(a_0,a_1,a_2)$ with
$\sigma_x^N(a_1,a_2)$.  Note that \eqref{CRWold} is used in
the proofs with $n=3$ and $n=2$, respectively.
\end{proof}

For nonempty sets $A\subset \bar\cN$ define
\begin{align*}
\Theta^{N}_2(A)
&= (\log N) \widehat P(\sigma^N(A/\sqrt N,(\bar\cN\setminus
A)/\sqrt N)>t_N,\, 
\tau^N(A/\sqrt N,(\bar\cN\setminus A)/\sqrt N)<t_N)\\
&=(\log N) \widehat P(\sigma(A,\bar\cN\setminus
A)>Nw_Nt_N,\,\tau(A,\bar\cN\setminus A)<Nw_Nt_N)\\
\Theta^{N}_2&=
\sum_{\emptyset\neq A\subset\cN}r^{N,a}(A)\Theta^{N}_2(A) .
\end{align*}
By Proposition~\ref{p:CRWnew} and \eqref{rNalim},
\begin{equation}\label{Tht2lim}
\lim_{N\to\infty}\Theta^{N}_2=
\Theta_2.
\end{equation}

\begin{proposition}\label{p:^HTht2} 
%For $T>0$ 
There is a 
$C_{\ref{^HTht2}}>0$ such
that for any $T>0$, $\phi\in C_b([0,T]\times\R^2)$, 
and $0\le s,u\le T$ with $|u-s|\le t_N$,
\begin{multline}
\Big| \hat H^{N,2}(\xi^N_0,t_N,\Phi_s) 
-\Theta^N_2 X^{N}_{0}(\Phi_u)\Big| \le
C_{\ref{^HTht2}}\|\Phi\|
\Bigl(\frac{1}{t_N\log
  N}\scrI^N(\xi^N_{0})+(\log
N)^{-1/2}X^N_{0}(\1)\Bigr).
\label{^HTht2}
\end{multline}
\end{proposition} 
\begin{proof}
In view of Lemma~\ref{l:dN3bnd} we will start with the
difference $\hat
H^{N,2}_2(\xi^N_0,t_N,\Phi_s)-\Theta^{N}_2X^N_0(\Phi_s)$.
Let $\emptyset\neq A\subset\cN_N$, and recall the definition
of $\Omega^N_x(A)$ in \eqref{OmegaN}, taking $u_N=t_N$.
For any fixed $a\in A$,
\begin{equation}
I^{N,+}_2(x,t_N,A,\xi^N_0) = \big[\xi^N_0(B^{N,x+a}_{t_N}) 
-\xi^N_0(B^{N,x+a}_{t_N})\xi^N_0(B^{N,x}_{t_N})\big]
1\{\Omega^N_x(A)\},
\end{equation}
and thus
\begin{multline}
|\hat E\big( I^{N,+}_2(x,t_N,A,\xi^N_0)\big)-
\hat P\big(B^{N,x+a}_{t_N}\in\xi^N_0, \Omega^N_x(A)\big)|
\\ \le %\max_{e\in\cN_N}
\hat P\big(B^{N,x+a}_{t_N}\in\xi^N_0,
B^{N,x}_{t_N}\in\xi^N_0, \sigma^N(x+a,x)>t_N\big).
\end{multline}
If we choose an $a=a(A)\in A$ for each non-empty
$A\subset \cN_N$, define
\begin{equation}
\tilde H^{N,2}_2(\xi^N_0,t_N,\Phi_s) =
\sum_{\emptyset\neq A\subset \cN_N} r^{N,a}(\sqrt N A)
\frac{\log N}{N'}\sum_{x\in\SN} \Phi(s,x)
\hat P(B^{N,x+a}_{t_N}\in\xi^N_0, \Omega^N_x(A)).
\end{equation}
It then follows from Proposition~\ref{p:2walks} that for a
constant $C$
\begin{equation}\label{tildeH2}
|\hat H^{N,2}_2(\xi^N_0,t_N,\Phi_s) -
\tilde H^{N,2}_2(\xi^N_0,t_N,\Phi_s)| \le
C\Vert\Phi\Vert_\infty \Bigl(\frac{1}{t_N\log
  N}\scrI^N(\xi^N_0)+(\log
N)^{-(1/2)}X^N_0(\1)\Bigr).
\end{equation}

To evaluate $\tilde H^{N,2}_{t_N}$, for fixed $A$ we have
\begin{align*}
\frac{\log N}{N'}\sum_{x\in \SN} &\Phi(s,x) \hat P
(B^{N,x+a}\in\xi^N_0,\Omega^N_x(A))\\
& = 
\frac{\log N}{N'}\sum_{x\in \SN} \Phi(s,x)
\sum_{w\in \SN} \xi^N_0(x+w) \hat P 
(B^{N,x+a}_{t_N}=x+w,\Omega^N_x(A)) \\
& =
\frac{\log N}{N'}\sum_{x\in \SN} \sum_{w\in\SN} 
[\Phi(s,x)-\Phi(s,x+w)] \xi^N_0(x+w)
\hat P(B^{N,a}_{t_N}=w,\Omega^N_0(A))\\
& \qquad +
\frac{\log N}{N'}\sum_{x\in \SN} \sum_{w\in\SN} 
\Phi(s,x+w) \xi^N_0(x+w)
\hat P(B^{N,a}_{t_N}=w,\Omega^N_0(A)).
\end{align*}
As in the proof of \eqref{dN32} (see
\eqref{lipbnd}),  
\begin{multline}
\frac{\log N}{N'}\sum_{x\in \SN} \sum_{w\in\SN} 
|\Phi(s,x)-\Phi(s,x+w)| \xi^N_0(x+w)
\hat P(B^{N,a}_{t_N}=w)\\
\le C(\cN)\|\Phi_s\|_\Lip X^N_0(\1)(\log N)^{-17/2}.
\label{Lip1}
\end{multline}
Summing first over $x$ and then over $w$,
\begin{equation}\label{=XP(Om)}
\frac{\log N}{N'}\sum_{x\in \SN} \sum_{w\in\SN} 
\Phi(s,x+w) \xi^N_0(x+w)
\hat P(B^{N,0}_{t_N}=w,\Omega^N_0(A))
= X^N_0(\Phi_s) (\log N)\hat P(\Omega^N_0(A)).
\end{equation}
Combining \eqref{tildeH2}--\eqref{=XP(Om)}, summing over 
$\emptyset\neq A\subset\cN$, and using the
definition of $\Theta^{N,2}$
we get
\begin{equation}
\begin{aligned}
| \hat H^{N,2}_2(\xi^N_0,&t_N,\Phi_s)
- \Theta^{N}_2X^N_0(\Phi_s) |
\\
&\le C\Vert\Phi\Vert\Bigl(\frac{1}{t_N\log
  N}\scrI^N(\xi^N_0)+(\log
N)^{-(1/2)}X^N_0(\1)\Bigr). %+ C(\log N)^{-1/2}X^N_0(\1)
\end{aligned}
\end{equation}
Now combine this with the easy bound $|X^N_0(\Phi(s,\cdot)
-X^N_0(\Phi(u,\cdot)|\le |\Phi|_{1/2,N}\sqrt{t_N} X^N_0(\1)$
and the boundedness of $\{\Theta^N_2\}$ to derive
\eqref{^HTht2}, but with $\hat H^{N,2}_2$ in place of $\hat
H^{N,2}$. Finally, use \eqref{dN2i} for $i\ge 3$ and the
above bound in the representation \eqref{HNrep} for $\hat
H^{N,2}$ to complete the proof.
\end{proof}

The proof of Proposition~\ref{p:^HTht} below, the analogue
of Proposition~\ref{p:^HTht2} for $\hat H^{N,3}$, is more
involved and requires additional notation. 
For nonempty finite disjoint sets $A,A_1,A_2\subset \bar\cN$
let
\begin{align*}
\Theta^N_3&(A,A_1,A_2)\\
&= (\log N)^3 \widehat P(\sigma^N(A/\sqrt N,A_1/\sqrt N,A_2/\sqrt N)>t_N,\,
\tau^N(A/\sqrt N,A_1/\sqrt N,A_2/\sqrt N)<t_N)\\
&=(\log N)^3 \widehat P(\sigma(A,A_1,A_2)>Nw_Nt_N,\,\tau(A,A_1,A_2)<Nw_Nt_N),
\end{align*}
and define
\begin{align}
\nonumber&\Theta^{N,+}_3(A) = \sum_{\{A_1,A_2\}\in
  \cP(\bar\cN\setminus A)}
\Theta^N_3(A,A_1,A_2),\\
\nonumber&\Theta^{N,-}_3(A) = \sum_{\{A_1,A_2\}\in \cP(A)}
\Theta^N_3(\bar\cN\setminus A,A_1,A_2),\\
 &
 \Theta^{N\pm}_3=\sum_{\emptyset\neq A\subset\cN}r^{N,s}(A)\Theta^{N,\pm}(A)
  \label{THtNdef}
\end{align}
and
$\Theta^N_3=\Theta^{N,+}_3-\Theta^{N,-}_3$. Proposition~\ref{p:CRWnew}
and \eqref{rNslim} imply that if $\Theta_3$ is as in
\eqref{Thetadefns}, then
\begin{equation}\label{Tht3lim}
\lim_N\Theta^N_3=\Theta_3.
\end{equation}
\begin{proposition}\label{p:^HTht} 
There is a constant
$C_{\ref{^HTht}}>0$ such
that for any $T>0$, $\phi\in C_b([0,T]\times\R^2)$, 
and $0\le s,u\le T$ with $|u-s|\le t_N$,
\begin{multline}
\Big| \hat H^{N,3}(\xi^N_0,t_N,\Phi_s) 
- \Theta^N_3 X^{N}_{0}(\Phi_u)\Big| \le
C_{\ref{^HTht}}\|\Phi\|
\Bigl(\frac{1}{t_N\log
  N}\scrI^N(\xi^N_{0})+(\log
N)^{-1/2}X^N_{0}(\1)\Bigr).
\label{^HTht}
\end{multline}
\end{proposition}
\begin{proof} Let us start with 
the difference  
$\hat H^N_3(\xi^N_{0},t_N,\Phi_s) - \Theta^N
X^{N}_{0}(\Phi(u,\cdot))$. Assume $\emptyset\neq A\subset
\cN_N$.  
It is easy to see that if 
$I^N_3(x,t_N,A,\xi^N_0)\ne 0$, then exactly one of 
$|B^{N,A}_{t_N}|$ or $|B^{N,\bar\cN_N\setminus A}_{t_N}|$ must
be 2, while the other must be 1. 
To account for these possibilities, we introduce
\begin{align*}
\chi^{N,+}_1(x,t_N,A,\xi^N_0)
&= 1\{B^{N,x+A}_{t_N}\subset \xi^N_0,|B^{N,x+A}_{t_N}|=2,
B^{N,x+\bar\cN_N\setminus A}_{t_N}\subset
\hxi^N_0,|B^{N,x+\bar\cN_N\setminus A}_{t_N}|=1\},\\
\chi^{N,-}_1(x,t_N,A,\xi^N_0) 
&=1\{B^{N,x+A}_{t_N}\subset \hxi^N_0,|B^{N,x+A}_{t_N}|=1,
B^{N,x+\bar\cN_N\setminus A}_{t_N}\subset
\xi^N_0,|B^{N,x+\bar\cN_N\setminus A}_{t_N}|=2\},\\
\chi^{N,+}_2(x,t_N,A,\xi^N_0)
&= 1\{B^{N,x+A}_{t_N}\subset \xi^N_0,|B^{N,x+A}_{t_N}|=1,
B^{N,x+\bar\cN_N\setminus A}_{t_N}\subset
\hxi^N_0,|B^{N,x+\bar\cN_N\setminus A}_{t_N}|=2\},\\
\chi^{N,-}_2(x,t_N,A,\xi^N_0) 
&=1\{B^{N,x+A}_{t_N}\subset \hxi^N_0,|B^{N,x+A}_{t_N}|=2,
B^{N,x+\bar\cN_N\setminus A}_{t_N}\subset
\xi^N_0,|B^{N,x+\bar\cN_N\setminus A}_{t_N}|=1\},
\end{align*}
and for $i=1,2$, 
\begin{align*}
\hat\Sigma^{N,i}_3&(\xi^N_0,t_N,\Phi_s)\\
&= 
\sumA r^{N,s}(\sqrt NA) \frac{(\log N)^3}{N'} \sum_{x\in\SN}\Phi(s,x)
\widehat E[\chi_i^{N,+}(x,t_N,A,\xi^N_0) - \chi_i^{N,-}(x,t_N,A,\xi^N_0)]. 
\end{align*}
It follows that
\begin{equation*}
(I^{N,+}_3-I^{N,-}_3)(x,t_N,A,\xi^N_0)=
(\chi^{N,+}_1 - \chi^{N,-}_1) (x,t_N,A,\xi^N_0)+ (\chi^{N,+}_2 -
\chi^{N,-}_2)(x,t_N,A,\xi^N_0), 
\end{equation*}
and thus
\begin{equation}\label{HN3decomp}
\hat H^{N}_3(\xi^N_0, t_N,\Phi_s) =
\hat\Sigma^{N,1}_3(\xi^N_0,t_N,\Phi_s)
+ \hat \Sigma^{N,2}_3(\xi^N_0,t_N,\Phi_s).
\end{equation}
The term $\hat\Sigma^{N,1}_3(\xi^N_0,t_N,\Phi_s)$ is
small. This is because 
both $\chi^{N,+}_1$ and $\chi^{N,-}_1$ are bounded above by
\[
\sum_{\text{distinct }a_0,a_1,a_2\in \bar\cN}
1\{B^{N,x+a_1}_{t_N}\in
\xi^N_0,B^{N,x+a_2}_{t_N}\in \xi^N_0,
\sigma^N_x(a_0,a_1,a_2)>t_N 
\},
\]
which implies by Lemma~\ref{p:2walks} that 
\begin{equation}\label{SigN13bnd}
|\hat \Sigma^{N,1}_3(\xi^N_0,t_N,\Phi_s)| \le
C_{\ref{2walks}}\binom{|\bar\cN|}{3}\|\Phi\|_\infty\Bigl(\frac{1}{t_N\log
  N}\scrI^N(\xi^N_0)+(\log  
N)^{-1/2}X^N_0(\1)\Bigr) .
\end{equation}

To handle $\hat\Sigma^{N,2}_3$ it is convenient to define
\begin{align*}
\hat\Sigma^{N,2,\pm}_3(\xi^N_0,t_N,\Phi_s) =
\sum_{\emptyset\neq A\subset\cN_N} r^{N,s}(\sqrt NA)\frac{(\log N)^3}{N'} \sum_{x\in\SN}\Phi(s,x)
\widehat E\big(\chi_2^{N,\pm}(x,t_N,A,\xi^N_0)\big),
\end{align*}
so that 
\begin{equation}\label{SigN23dcmp}
\hat \Sigma^{N,2}_3(\xi_0^N,t_N,\Phi_s)
=\hat\Sigma^{N,2,+}_3(\xi_0^N,t_N,\Phi_s)
-\hat\Sigma^{N,2,-}_3(\xi_0^N,t_N,\Phi_s).  
\end{equation}
Consider $\hat\Sigma^{N,2,+}_3$, and let
 $\emptyset\neq A\subset\cN_N$.  
On the event defining
$\chi^{N,+}_2(x,t_N,A,\xi^N_0)$  there must
exist nonempty disjoint $A_1,A_2$ whose union is
$\bar\cN_N\setminus A$, 
such that none of the walks starting from the three disjoint sets $A,A_1$ and $A_2$ 
meet by time $t_N$, 
$|B^{N,x+A_1}_{t_N}|=|B^{N,x+A_2}_{t_N}|=1$ and
$B^{N,x+A_1}_{t_N},B^{N,x+A_2}_{t_N}\subset\hxi^N_0$. 
Thus, if we define
\[
\Omega^N_x(A,A_1,A_2)=\{
\sigma^N_x(A,A_1,A_2)>t_N,\tau^N_x(A,A_1,A_2)<t_N\},
\]
then for any 
$a\in A$ and $a_i\in A_i$,
\begin{equation}\label{e:+1}
 \widehat E\big(\chi^{N,+}_2(x,t_N,A,\xi^N_0)\big) =
\sum_{\substack{\{A_1,A_2\}\in\\
  \cP(\bar\cN_N\setminus A)}}\widehat P\Big( \Omega^N_x(A,A_1,A_2),
B^{N,x+a}_{t_N}\in \xi^N_0,
B^{N,x+a_1}_{t_N}\in \hxi^N_0,
B^{N,x+a_2}_{t_N}\in \hxi^N_0\Big).
\end{equation}

The next step is to see that we may drop requirement that
$B^{N,x+a_1}_{t_N},B^{N,x+a_2}\in \hxi^N_0$ above. Observe
that
\begin{align}
\notag \Big|\widehat P\Big( &\Omega^N_x(A,A_1,A_2),
B^{N,x+a}_{t_N}\in \xi^N_0,
B^{N,x+a_1}_{t_N}\in \hxi^N_0,
B^{N,x+a_2}_{t_N}\in \hxi^N_0\Big)\\
\notag&\qquad\qquad\qquad - \widehat P\Big( \Omega^N_x(A,A_1,A_2),
B^{N,x+a}_{t_N}\in \xi^N_0\Big)\Big|\\
\notag&\le \widehat P\Big( \sigma^N_x(A,A_1,A_2)>t_N,
B^{N,x+a}_{t_N}\in \xi^N_0,
B^{N,x+a_1}_{t_N}\text{ or }
B^{N,x+a_2}_{t_N}\in \xi^N_0\Big)\\
&\label{e:+2} \le
2\max_{\text{distinct }a,a_1,a_2\in\bar\cN_N}\widehat P(\sigma^N_x(a,a_1,a_2)>t_N,
B^{N,x+a}_{t_N}\in\xi^N_0,B^{N,x+a_1}_{t_N}\in\xi^N_0) .
\end{align}
If we let
\begin{multline}\label{tildsig}
\widetilde \Sigma^{N,2,+}_3(\xi^N_0,t_N,\Phi_s)\\ =
\sum_{A\subset\cN} r^{N,s}(\sqrt N A)\frac{(\log N)^3}{N'}
\sum_{x\in\SN}\Phi(s,x)
\sum_{\substack{\{A_1,A_2\}\in\\
  \cP(\bar\cN_N\setminus A)}}\widehat P\Big( \Omega^N_x(A,A_1,A_2),
B^{N,x+a}_{t_N}\in \xi^N_0\Big),
\end{multline}
it follows from the above and Lemma~\ref{p:2walks} that
there is a constant $C_{\ref{e:+3}}$ such that
\begin{multline}
\Big|\hat\Sigma^{N,2,+}_3(\xi^N_0,t_N,\Phi_s)
-\widetilde\Sigma^{N,2,+}_3(\xi^N_0,t_N,\Phi_s)\Big|\\ \le
C_{\ref{e:+3}}\|\Phi\|_\infty \Bigl(\frac{1}{t_N\log 
  N}\scrI^N(\xi^N_0)+(\log
N)^{-1/2}X_0^N(\1)\Bigr).\label{e:+3} 
\end{multline}

Again, consider a fixed $\emptyset\neq A\subset\cN_N$ and
let $\{A_1,A_2\}\in \cP(\bar\cN_N\setminus A)$. By
translation invariance,
\begin{align}\label{e:+4}
\nonumber\frac{(\log N)^3}{N'}&\sum_{x\in\SN}\Phi(s,x)\widehat P\Big(
\Omega^N_x(A,A_1,A_2),
B^{N,x+a}_{t_N}\in \xi^N_0)\\
\nonumber&= 
\frac{(\log
  N)^3}{N'}\sum_{x\in\SN}\sum_{w\in\SN}\Phi(s,x )\xi^N_0(x+w)\widehat P\Big( 
\Omega^N_x(A,A_1,A_2),
B^{N,x+a}_{t_N}=x+w\Big) \\
\nonumber&= \frac{(\log N)^3}{N'}
\sum_{x\in\SN}\sum_{w\in\SN}
[\Phi(s,x)-\Phi(s,x+w)]\xi^N_0(x+w)
\widehat P\Big(
\Omega^N_0(A,A_1,A_2),
B^{N,a}_{t_N}=w\Big) \\
&\qquad + \frac{(\log N)^3}{N'}
\sum_{x\in\SN}\sum_{w\in\SN}
\Phi(s,x+w )\xi^N_0(x+w)
\widehat P\Big(\Omega^N_0(A,A_1,A_2),
B^{N,a}_{t_N}=w\Big).
\end{align}
As in the proof of \eqref{dN32} (see
\eqref{lipbnd}),  
\begin{multline}\label{e:+5}
\frac{(\log N)^3}{N'}
\sum_{x\in\SN}\sum_{w\in\SN}
\Big|(\Phi(s,x)-\Phi(s,x+w))\Big|\xi^N_0(x+w)
\widehat P\Big(
\Omega^N_0(A,A_1,A_2),
B^{N,a}_{t_N}=w\Big) \\
\le
C_{\ref{e:+5}}\|\Phi\|_\Lip X^N_0(\1)(\log N)^{-13/2}.
\end{multline}
For the final sum in \eqref{e:+4}, summing first over $x$
and then over $w$ yields
\begin{multline}
\frac{(\log N)^3}{N'}\sum_{x\in\SN}\sum_{w\in\SN}
\Phi(s,x+w)\xi^N_0(x+w)
\widehat P\Big(\Omega^N_0(A,A_1,A_2),
B^{N,a}_{t_N}=w\Big) 
\\= \Theta^N_3(\sqrt NA,\sqrt NA_1,\sqrt NA_2)
X^N_0(\Phi_s).\label{e:+6}
\end{multline}
Combining \eqref{e:+4}--\eqref{e:+6} and using the
definition of $\tilde \Sigma$ we find that
\begin{multline}
\Big| \frac{(\log N)^3}{N'}\sum_{x\in\SN}\Phi(s,x)\widehat P\Big(
\Omega^N_x(A,A_1,A_2),
B^{N,x+a}_{t_N}\in \xi^N_0\Big) -
\Theta^N_3(\sqrt NA,\sqrt NA_1,\sqrt NA_2)
X^N_0(\Phi_s)\Big|\\ \le C_{\ref{e:+5}}\|\Phi\|_\Lip X^N_0(\1) (\log N)^{-13/2},
\end{multline} 
\begin{equation}
\Big| \tilde\Sigma^{N,2,+}(\xi^N_0,t_N,\Phi_s) - 
\Theta^{N,+}
X^N_0(\Phi_s)\Big|\\ \le C_{\ref{e:+5}}\binom{|\bar\cN_N|}{3}
(\log N)^{-13/2} \|\Phi\|_\Lip X^N_0(\1) ,
\end{equation}
and by \eqref{e:+3},
\begin{multline}\label{SigN23+bnd}
\Big|\hat\Sigma^{N,2,+}_3(\xi^N_0,t_N,\Phi_s) -
\Theta^{N,+}
X^N_0(\Phi_s)\Big|\\ \le
C_{\ref{SigN23+bnd}}\|\Phi\|_\Lip \Bigl(\frac{1}{t_N\log 
  N}\scrI^N(\xi^N_0)+(\log
N)^{-1/2}X_0^N(\1)\Bigr).
\end{multline}
In the above we have rescaled the sets and summed over
subsets of $\cN$ (or $\bar\cN$) instead of $\cN_N$ (or $\bar
\cN_N$).  The same bound holds for the difference
$\hat\Sigma^{N,2,-}_3(\xi^N_0,t_N,\Phi_s) - \Theta^{N,-}
X^N_0(\Phi_s)$, and thus in view of \eqref{HN3decomp},
\eqref{SigN13bnd}, \eqref{SigN23dcmp}, \eqref{SigN23+bnd},
and the simple bound $|X^N_0(\Phi(s,\cdot)
-X^N_0(\Phi(u,\cdot)|\le |\Phi|_{1/2,N}\sqrt{t_N}
X^N_0(\1)$, we conclude that
\begin{equation*} \Big| \hat H_3^N(\xi^N_0,t_N,\Phi_s) -
\Theta^N X^{N}_{0}(\Phi_u)\Big| \le C_{\ref{^HTht}}\|\Phi\|
\Bigl(\frac{1}{t_N\log N}\scrI^N(\xi^N_{0})+(\log
N)^{-1/2}X^N_{0}(\1)\Bigr).
\end{equation*} Combine this with Lemma~\ref{l:dN3bnd} and
\eqref{HNrep} to obtain \eqref{^HTht}.
\end{proof}

As an immediate consequence of Lemma~\ref{l:d3H} (with
$u_N=t_N$), Proposition~\ref{p:^HTht2} and
Proposition~\ref{p:^HTht} we obtain the following result.
\begin{proposition}\label{r:conddriftbnd} There is a
constant $C_{\ref{dsThtX}}>0$ such that for $j=2,3$, any
$T>0$, $\Phi\in C_b([0,T]\times\R^2)$, and $s\in[t_N,T]$,
\begin{multline}\label{dsThtX}
\Big|E\big(d^{N,j}(s,\xi^N_s,\Phi)|\cF^N_{s-t_N}\big) -
\Theta^N_jX^N_{s-t_N}(\Phi)\Big| \\ \le
C_{\ref{dsThtX}}\|\Phi\| \Bigl(\frac{1}{t_N\log
N}\scrI^N(\xi^N_{s-t_N})+(\log
N)^{-1/2}X_{s-t_N}^N(\1)\Bigr).
\end{multline}
\end{proposition}

Turning briefly to the martingale square function, and recalling \eqref{CMPL51b} and the definition of $K_N$ from \eqref{KNdefn}, we have the following analogue of the above.
\begin{proposition}\label{p:condsqfn}There is a $C_{\ref{dssqfn}}$
such that for any bounded Lipschitz continuous
$\Phi:\R^2\to\R$ and $s\ge t_N$, 
\begin{multline}\label{dssqfn}
\Big|E\big(X^N_s(\log N\Phi
f_0^{(N)}(\cdot,\xi^N_s))|\cF^N_{s-t_N}\big)
-K_NX^N_{s-t_N}(\Phi)\Big|\\ \le
C_{\ref{dssqfn}}\|\Phi\|_\Lip\Bigl(\frac{1}{t_N\log
N}\scrI^N(\xi^N_{s-t_N})+(\log
N)^{-1/2}X_{s-t_N}^N(\1)\Bigr).
\end{multline}
\end{proposition}
\begin{proof} By \eqref{CMPL52a} and \eqref{CMPL52b},
\begin{align}
\label{dcompI}E(X^N_{t_N}(\log N\Phi f_0^{(N)}(\cdot,\xi^N_{t_N}))
&=\frac{\log
  N}{N'}\sum_{x\in\SN}\sum_{y\in\cN_N}\Phi(x)p_N(y)
\hat E(\xi^N_0(B^{N,x}_{t_N})\hxi_0^N(B^{N,x+y}_{t_N}))+R_N\\
\nonumber&:=S_N+R_N,
\end{align}
where 
\begin{equation}\label{err1}
|R_N|\le C\|\Phi\|_\infty(\log N)^{-15}X_0^N(\1).
\end{equation}
Clearly
\begin{align*}S_N&=\frac{\log
N}{N'}\sum_{x\in\SN}\sum_{y\in\cN_N}\Phi(x)p_N(y)\hat
E(\xi^N_0(B^{N,x}_{t_N})1(\sigma^N(x,x+y)>t_N))\\ &\ \
-\frac{\log
N}{N'}\sum_{x\in\SN}\sum_{y\in\cN_N}\Phi(x)p_N(y)\hat
E(\xi^N_0(B^{N,x}_{t_N})\xi_0^N(B^{N,x+y}_{t_N})1(\sigma^N(x,x+y)>t_N)\\
&:=S_N^1-S_N^2.
\end{align*}
\eqref{2walksM} shows that 
\begin{equation}\label{err2}
|S^2_N|\le C_{\ref{2walks}}\|\Phi\|_\infty\Bigl(\frac{1}{t_N\log N}\scrI^N(\xi^N_0)+(\log
N)^{-(1/2)}X^N_0(\1)\Bigr).
\end{equation}
By Lemma~\ref{l:oldl5.2} (a), 
\[S^1_N=K_NX_0^N(\Phi)+\cE_N,\]
where
\begin{equation}\label{err3}
|\cE_N|\le C_{\ref{CMPL52}}|\Phi|_\Lip(\log N)^{-17/2}X_0^N(\1).
\end{equation}
Use the error bounds \eqref{err1}, \eqref{err2} and
\eqref{err3} in the decomposition \eqref{dcompI}, and then
apply the Markov property to complete the proof. 
\end{proof}

In view of the above results, to obtain exact asymptotics
for the limiting drift arising from the $q$-voter
perturbation term as well as the square function of the
martingale term, we will show that $\frac{1}{t_N\log
  N}E\left[ \int_{t_N}^{t\vee t_N} 
\scrI^N(\xi^{N}_{s-t_N})ds \right]$ is negligible for large
$N$.   

\begin{proposition} \label{p:key2bnd} Let $(\delta_N)_{N \ge
3}$ be a positive sequence such that $\lim_{N \to \infty}
\delta_N = 0$ and \\ $\liminf_{N \to \infty} \sqrt{N}
\delta_N >0$.  For any $T$, there exists
$C_{\ref{key2bnd}}$, depending on $T$ and $\{\delta_N\}$,
such that for any $t \le T$, for any $N\ge 3$,
\begin{equation}\label{key2bnd} E [
\scrI^N(\delta_N,\xi^N_t)] \le C_{\ref{key2bnd}} \Big(
X_0^N(1)+X_0^N(1)^2\Big) \Big[ \delta_N\Big(1+ \log
\Big(1+\frac{t}{\delta_N}\Big)\Big) +
\frac{\delta_N}{t+\delta_N}\Big].
\end{equation}
\end{proposition}
The proof of this key result is similar in outline to that
of the same result, Proposition~3.10 in \cite{CMP}, for
Lotka-Volterra models, although additional complications
will arise in the present setting. We will prove it in
Section~\ref{sec:Ibound} below but assume it in what
follows.

\begin{corollary}\label{c:keyint} If $T>0$ there exists a constant
$C_{\ref{keyint}}$ depending on $T$ such that for any $t\le T$,
\begin{equation}\label{keyint}
\frac{1}{t_N\log N}E\left[ \int_{t_N}^{t\vee t_N}
\scrI^N(\xi^{N}_{s-t_N})ds  \right]
\le C_{\ref{keyint}} (\log N)^{-1/2}(X^N_0(\1)+X^N_0(\1)^2). 
\end{equation}
\end{corollary}
\begin{proof}Given Proposition~\ref{p:key2bnd}, the
elementary proof of the above is identical to that of
Corollary~3.11 in \cite{CMP}. (The bound in the final line
of the proof there leads to an upper bound as above but
instead with $(\log N)^{-\eta+\delta}$ for any $\delta>0$
and $\eta\in(0,1)$, and hence, in particular, $(\log
N)^{-1/2}$).\end{proof}

We are finally ready to give the asymptotic behavior of the
drifts, $D^{N,j}_t(\Phi)$, $j=2,3$, and so prove the main
result of this Section.
\begin{proposition}\label{p:d3asymp} For $j=2,3$, any
$\Phi:\R^2\to\R$ with $\|\Phi\|_\Lip<\infty$ and all $t>0$,
\[\lim_{N\to\infty}E\Bigl(\Bigl|D^{N,j}_t(\Phi)-\Theta_j\int_0^tX_s^N(\Phi)\,ds\Bigr|\Bigr)=0.\]
\end{proposition}
\begin{proof} Let $t\in(0,T]$.  A little thought shows the above expectation is bounded by
\begin{align}
\nonumber E\Bigl(&\int_0^{t_N\wedge t}|d^{N,j}(s,\xi^N_s,\Phi)|
\,ds\Bigr)+E\Bigl(\Bigl(\int_{t_N}^{t_N\vee
  t}d^{N,j}(s,\xi^N_s,\Phi)-E(d^{N,j}(s,\xi^N_s,\Phi)|\cF^N_{s-t_N})\,ds\Bigr)^2\Bigr)^{1/2}\\ 
\label{trianglebound}&+E\Bigl(\int_{t_N}^{t_N\vee
  t}|E(d^{N,j}(s,\xi^N_s,\Phi)|\cF^N_{s-t_N})-\Theta^N_jX_{s-t_N}^N(\Phi)|\,ds\Bigr)\\ 
\nonumber&+|\Theta^N_j-\Theta_j|\|\Phi\|_\infty
E\Bigl(\int_0^{(t-t_N)^+}X_s^N(\1)\,ds\Bigr)+|\Theta_j|\|\Phi\|_\infty
E\Bigl(\int_{(t-t_N)^+}^tX_s^N(\1)\,ds\Bigr). 
\end{align}
Proposition~\ref{p:mobnds}(a) and \eqref{crude1} shows the
first term is at most
\[C\|\Phi\|_\infty (\log N)^3 t_N X^N_0(\1)\to 0\text{ as
}N\to\infty.\] Use an orthogonality argument as in the
derivation of (67) in \cite{CMP} and also \eqref{crude1} to
bound the second term in \eqref{trianglebound} by
\[C\|\Phi\|_\infty(\log N)^3
t_N^{1/2}\Bigl[E\Bigl(\int_{t_N}^{t_N\vee
t}X^N_s(\1)^2\,ds\Bigr)\Bigr]^{1/2}\le
C_T\|\Phi\|_\infty(\log N)^{-13/2}(1+X_0^N(\1))\to0\text{ as
}N\to\infty,
\]
where the inequality holds by
Proposition~\ref{p:mobnds}(b). Since
$|\Theta^N_j-\Theta_j|\to 0$ (by \eqref{Tht2lim} and
\eqref{Tht3lim}), we may apply Proposition~\ref{p:mobnds}(a)
to see that the last two terms in \eqref{trianglebound} are
bounded by
\[C_T\|\Phi\|_\infty X_0^N(\1)[|\Theta^N_j-\Theta_j|+t_N]\to 0\text{ as }N\to\infty.\]
Finally we may use Proposition~\ref{r:conddriftbnd},
Corollary~\ref{c:keyint} and Proposition~\ref{p:mobnds}(a)
to bound the remaining (middle) term of
\eqref{trianglebound} by
\[C_T\|\Phi\|_\Lip(\log N)^{-1/2}(X_0^N(\1)+X_0^N(\1)^2)\to
0\text{ as }N\to\infty.\]
Combining the above displays, we may complete the proof.
\end{proof}
\begin{remark}\label{r:d3ebnd} If we keep track of the bounds
in the above proof (and drop the
$|\Theta^N_j-\Theta_j|$ term) we get for $j=2,3$, 
\begin{equation}\label{d3errorbnd}
E\Bigl(\Bigl|D^{N,j}_t(\Phi)-\Theta^N_j\int_0^tX_s^N(\Phi)\,ds\Bigr|\Bigr)\le
C_T\|\Phi\|_\Lip(1+X_0^N(\1)^2)(\log N)^{-1/2} \quad\text{for all }t\in[0,T].
\end{equation}
\end{remark}

\begin{proposition}\label{p:sqmasymp}
For any $\Phi:\R^2\to\R$ with $\|\Phi\|_\Lip<\infty$ and all
$t>0$,
\[\lim_{N\to\infty}E\Bigl(\Bigl|\langle
M^N(\Phi)\rangle_t-4\pi\sigma^2\int_0^tX_s^N(\Phi^2)\,ds\Bigr|\Bigr)=0.\]
\end{proposition} 
\begin{proof} Note that $\|\Phi\|_\Lip$ finite implies that
$\|\Phi^2\|_\Lip$ is finite. Arguing exactly as in the proof
of Proposition~\ref{p:d3asymp}, using
Proposition~\ref{p:condsqfn} in place of
Proposition~\ref{r:conddriftbnd} and \eqref{sqfnconst} in
place of \eqref{Tht2lim} and \eqref{Tht3lim}, we get
\[\lim_{N\to\infty}E\bigl(\Big|\int_0^tX_s^N(2\log
N\Phi^2f_0^{(N)}(\cdot,\xi^N_s))\,ds-\int_0^t4\pi\sigma^2X^N_s(\Phi^2)\,ds\Big|\Bigr)=0.\]
Now use \eqref{CMPL51b}, \eqref{CMPL51a}, and
Proposition~\ref{p:mobnds}(a) to complete the proof.
\end{proof}

Define $(P^N_t,t\ge0)$ as the semigroup of the rate-$N$ random
walk on $\SN$ with jump kernel $p_N$. By translation invariance we can
have $P^N_t$ operate on functions on the plane, even though $\SN$ is the natural state space. 

As an application of the control of the drift terms
$d^{N,j}$ given by Proposition~\ref{p:d3asymp}, we obtain an
effective bound on the mean measures of our rescaled
$q$-voter models.
\begin{lemma}\label{l:CMPL35} There exists a
$c_{\ref{firstmomeasbnd}}>0$, and for any $T>0$ a constant
$C_{\ref{firstmomeasbnd}}(T)>0$ so that for all $t\in[0, T]$
and any $\Psi:\R^2\to\R_+$ such that $\|\Psi\|_{\Lip}\le T$,
\begin{equation}\label{firstmomeasbnd} E[X^N_t(\Psi)] \le
e^{c_{\ref{firstmomeasbnd}}t} X^N_0(P^N_t(\Psi)) +
C_{\ref{firstmomeasbnd}}(\log N)^{-1/2} (X^N_0(\1) +
X^N_0(\1)^2) .
\end{equation}
\end{lemma}
\begin{proof} Fix $T>0$, let $t\in[0,T]$ and $\Psi$ be as in
the Lemma, and let $c\in\R$.  Define $\Phi\in
C_b([0,t]\times\R^d)$ by
$\Phi(s,x)=e^{-cs}P^N_{t-s}\Psi(x)$, so that
\begin{equation}\label{HeatEq}
\dot\Phi(s,x)=-A_N\Phi(s,x)-c\Phi(s,x)\in
C_b([0,t]\times\R^2).
\end{equation} Let
$d^N(s,\xi^N_s,\Phi)=d^{N,2}(s,\xi^N_s,\Phi)+d^{N,3}(s,\xi^N_s,\Phi)$
and $\Theta^N= \Theta^N_2+ \Theta^N_3$.  By \eqref{SMG1} and
\eqref{HeatEq} we have
\begin{equation}\label{semimartI}
E(X^N_t(\Phi(t,\cdot)))=X^N_0(P^N_t(\Psi))+\int_0^tE(d^{N}(s,\xi^N_s,\Phi)-cX^N_s(\Phi))\,ds.
\end{equation} Now choose $c>1\vee\sigma\vee (\sup_{N\ge
e^3}\Theta^N)$ (recall \eqref{Tht2lim} and
\eqref{Tht3lim}). It is then easy to see (and in fact is
shown in the proof of Lemma 3.5 in \cite{CMP}--see p. 1232)
that $c>1\vee\sigma$ implies
\begin{equation}\label{phinormbnd} \|\Phi\|\le
2c\|\Psi\|_{\Lip}\le 2cT,
\end{equation} where the last inequality is by
hypothesis. Now use \eqref{d3errorbnd} and
\eqref{phinormbnd} to deduce that
\[ E\Bigl(\int_{0}^td^{N}(s,\xi^N_s,\Phi)\,ds\Bigr)\le
\int_0^tcE(X^N_s(\Phi))\,ds+C_T\|\Psi\|_\Lip (\log
N)^{-1/2}(1+X_0^N(\1)^2),\ \ \forall t\in[0,T].\] Finally
use this in \eqref{semimartI}, noting that
$\Phi(t,x)=e^{-ct}\Psi(x)$, to complete the proof with
$c_{\ref{firstmomeasbnd}}=c$.
\end{proof}

\section{Convergence to super-Brownian motion: Proof of Theorem~\ref{t:SBMgen}}\label{sec:convtoSBM}
We assume the hypotheses of Theorem~\ref{t:SBMgen} whose proof is the objective of this section.
\subsection{Relative compactness}
A collection of stochastic processes $\{Y^N:N\ge \alpha\}$
with sample paths in $D(\R_+,S)$ for some Polish space $S$
is $C$-relatively compact iff for every sequence
$N_k\uparrow \infty$ in $[\alpha,\infty)$ there is a
subsequence $\{N'_k\}$ so that $Y^{N'_k}$ converges weakly
in $D(\R_+,S)$ to a process with continuous paths. The same
definition applies to a given {\it sequence} of processes.

\begin{proposition}\label{p:tness} The set $\{X^N,N\ge 
  N(\vep_0)\}$ is $C$-relatively compact in $D(\R^+,\MF(\R^2))$.
\end{proposition}
To prove this it clearly suffices to show that for every
$N_k\uparrow\infty$ $(N_k\ge N(\vep_0))$, $\{X^{N_k}\}$, is
$C$-relatively compact.  To ease the notation we take
$N_k=k$ as the proof in the general case is the same.  Hence
we reduce to the case of showing $\{X^N:N\in\N,N\ge
N(\vep_0)\}$ is relatively compact.  This result will follow
from Jakubowski's theorem (see Theorem~II.4.1 in [P2002])
and the following two lemmas.

\begin{lemma}\label{l:Phit}
For any $\Phi\in C_b^3(\R^2)$, the sequence
$\{X^N(\Phi), N\in\N^{\ge
  N(\vep_0)}\}$ is $C$-relatively compact in $D(\R_+,\R)$.
\end{lemma}

\begin{lemma}\label{l:ccc} For any $\vep>0$, $T>0$ there
exists $A>0$ such that
\[
\sup_{N\in\N^{\ge N(\vep_0)}}P\left( \sup_{t\le T} X^N_t(B(0,A)^c)>\vep\right)
< \vep. 
\]
\end{lemma}

\medskip

\noindent{\sl{Proof of Lemma~\ref{l:Phit}.}} We use
\eqref{SMG1} to establish $C$-relative compactness of
$\{X^N(\Phi):N\in\N^{\ge N(\vep_0)}\}$ by establishing the
$C$-relative compactness of each of the terms on the
right-hand side.  For the latter we proceed as in Lemma~6.1
of \cite{CMP}.  Consider first $\{D^{N,j}(\Phi):N\in\N^{\ge
N(\vep_0)}\}$ for $j=2\text{ or }3$.  Use
Proposition~\ref{p:mobnds}(b) in \eqref{DNt1t2} and
$\min(a,b)\le \sqrt{ab}$ for $a,b>0$, to see that for
$t_N\le s_1<s_2\le T$,
\begin{align} \nonumber
E((D^{N,j}_{s_2}(\Phi)-D^{N,j}_{s_1}(\Phi))^2)&\le
C_T\|\Phi\|_\Lip^2[(s_2-s_1)^2+\log^6N\sqrt{t_N}(s_2-s_1)^{3/2}](X_0^N(\1)+X_0^N(\1)^2)\\
&\le
C_T\|\Phi\|_\Lip^2(X_0^N(\1)+X_0^N(\1)^2)(s_2-s_1)^{3/2}.\label{DN3inc1}
\end{align}
Moreover, we have from \eqref{crude1} that for $j=2,3$, 
\begin{equation}\label{dn3roughbnd}
|d^{N,j}(s,\xi_s^N,\Phi)|\le |\bar\cN|(\log N)^3\|\Phi\|_\infty X_s^N(\1),
\end{equation}
and so by Proposition~\ref{p:mobnds}(b) for $0\le s_1<s_2\le t_N$,
\begin{align}\label{DN3inc2} \nonumber
E\Bigl(\Bigl(\int_{s_1}^{s_2}|d^{N,j}(s,\xi^N_s,\Phi)|\,ds\Bigr)^2\Bigr)&\le
C|\bar\cN| \|\Phi\|_\infty^2(X^N_0(\1)+X^N_0(\1)^2)(\log
N)^6(s_2-s_1)^2\\ \nonumber &\le C
\|\Phi\|_\infty^2(X^N_0(\1)+X^N_0(\1)^2)(\log
N)^6\sqrt{t_N}(s_2-s_1)^{3/2}\\ &\le
C\|\Phi\|_\infty^2(X^N_0(\1)+X^N_0(\1)^2)(s_2-s_1)^{3/2}.
\end{align}
The $C$-relative compactness of
$\{D^{N,j}(\Phi):N\in\N^{\ge N(\vep_0)}\}$ is now immediate
from $D_0^{N,j}=0$, \eqref{DN3inc1}, \eqref{DN3inc2} and
Kolmogorov's criterion.

Turning to the $C$-relative compactness of
$\{M^N(\Phi):N\in\N^{\ge N(\vep_0)}\}$, as in Lemma~6.1 of
\cite{CMP}, because the maximum jump of $M^N(\Phi)$ goes to
zero as $N\to\infty$, it suffices to prove $C$-relative
compactness of $\{\langle M^N(\Phi)\rangle:N\in\N^{\ge
N(\vep_0)}\}$.  From \eqref{SMGsqfn}, \eqref{CMPL51b},
\eqref{CMPL51a}, and the second moment bound in
Proposition~\ref{p:mobnds}(b), we have for $0\le s_1<s_2\le
T$,
\begin{align}\label{MnincI} \nonumber E\Bigl(\Bigl(\langle
M^N(\Phi)\rangle_{s_2}-\langle
M^N(\Phi)\rangle_{s_1}\Bigr)^2\Bigr)&\le
C_T\|\Phi\|_\Lip^4(X_0^N(\1)+X_0^N(\1)^2)(s_2-s_1)^2\\
&\quad+C\|\Phi\|_\infty^4E\Bigl(\Bigl(\int_{s_1}^{s_2}Z^N_s\,ds\Bigr)^2\Bigr),
\end{align}
where
\begin{equation}\label{Ybound}Z^N_s=X^N_s(\log Nf_0^{(N)}(\xi^N_s))\le \log NX^N_s(\1).\end{equation}
Using \eqref{CMPL52}, the above upper bound, and arguing
exactly as in the proof of (104)--(106) in Lemma~6.1 of
\cite{CMP} (conditioning back in time by $t_N$ and employing
the Markov property) we can show that for $T\ge s_2>s_1\ge
t_N$,
\begin{equation}\label{MnincII}E\Bigl(\Bigl(\int_{s_1}^{s_2}
Z^N_s\,ds\Bigr)^2\Bigr)\le
C_T(X^N_0(\1)+X^N_0(\1)^2)(s_2-s_1)^{3/2}.\end{equation} For
$0\le s_1<s_2\le t_N$ we may argue as in \eqref{DN3inc2}
using \eqref{CMPL51a}, \eqref{CMPL51b}, and the upper bound
in \eqref{Ybound}, to see that
\[E\Bigl(\Bigl(\langle M^N(\Phi)\rangle_{s_2}-\langle
M^N(\Phi)\rangle_{s_1}\Bigr)^2\Bigr)\le
C_T\|\Phi\|_{\Lip}^4(X^N_0(\1)+X^N_0(\1)^2)(s_2-s_1)^{3/2}.\]
Combine the above with \eqref{MnincI}, and \eqref{MnincII}
to conclude that
\begin{equation}\label{sqfncontrol} E\Bigl(\Bigl(\langle
M^N(\Phi)\rangle_{s_2}-\langle
M^N(\Phi)\rangle_{s_1}\Bigr)^2\Bigr)\le
C_T\|\Phi\|_{\Lip}^4(X^N_0(\1)+X^N_0(\1)^2)(s_2-s_1)^{3/2}\text{
for } 0\le s_1<s_2\le T.
\end{equation}
The $C$-relative compactness of
$\{\langle M^N(\Phi)\rangle:N\in\N^{\ge N(\vep_0)}\}$ now
follows from Kolmogorov's criterion.

Finally, the simple argument in Lemma~6.1 of \cite{CMP}
using Proposition~\ref{p:mobnds}(b) shows for $0\le
s_1<s_2\le T$,
\begin{equation}\label{DNoneinc}
E\Bigl(\Bigl(D^{N,1}_{s_2}(\Phi)-D^{N,1}_{s_1}(\Phi)\Bigr)^2\Bigr)\le
C_{T,\Phi}(X^N_0(\1)+X^N_0(\1)^2)(s_2-s_1)^{2}.
\end{equation} (This is one place where the assumed
regularity of $\Phi$ is used.)  The $C$-relative compactness
of $\{D^{N,1}(\Phi):N\in\N^{\ge N(\vep_0)}\}$ follows as
usual. As $X_0^N(\Phi)\to X_0(\Phi)$ by hypothesis, the
above results give the $C$-relative compactness of
$\{X^N(\Phi)\}$.
\qed

\medskip

\noindent{\sl{Proof of Lemma~\ref{l:ccc}.}}  Let
$\{h_n:n\in\N\}$ be a sequence of $[0,1]$-valued functions
in $C^3_b(\R^2)$ such that
\begin{equation}\label{hnprops}\1_{\{|x|>n+1\}}\le h_n(x)\le
\1_{\{|x|>n\}},\ \text{and} \ \sup_n\|h_n\|_\Lip\le C.
\end{equation} For example, if $h:\R\to[0,1]$ is $C^\infty$,
increasing, and $1_{\{z>1\}}\le h(z)\le 1_{\{z>0\}}$ we can
take $h_n(x)=h(|x|-n)$. It clearly suffices to show that for
each fixed $\vep,T>0$,
\begin{equation}\label{hndecay}
\lim_{n\to\infty}\sup_{N\ge N(\vep_0)}P\Bigl(\sup_{t\le T}X_t^N(h_n)\ge\vep\Bigr)=0.
\end{equation}
By \eqref{SMG1}
\begin{equation}\label{hndecomp}
X^N_t(h_n)=M^N_t(h_n)+Y^N_t(h_n),
\end{equation}
where
\[Y^N_t(h_n)=X_0^N(h_n)+\int_0^tX_s^N(A_Nh_n)\,ds+D^{N,2}_t(h_n)+D^{N,3}_t(h_n).\]
Now argue as in the derivation of (110) in the proof of
Lemma~6.1 in \cite{CMP}, using \eqref{SMGsqfn},
\eqref{CMPL51b}, \eqref{CMPL51a}, the second inequality in
\eqref{hnprops}, and \eqref{CMPL52} to conclude that for
some $\vep_N\to 0$, independent of $n$,
\begin{equation}\label{sqfnbound1} E(\langle
M^N(h_n)\rangle_T)\le \vep_N+\int_0^T
C_TE(X^N_s(h_n^2))\,ds\le
\vep_N+\int_0^TC_TX_0^N(P^N_s(h_n^2))\,ds.
\end{equation} In the last inequality we used
Lemma~\ref{l:CMPL35} and absorbed some of the constants and
terms there into the $C_T$ and $\vep_N$.  Chebychev's
inequality (recall \eqref{hnprops}) shows that for any
$K>0$,
\[\lim_{n\to\infty}\sup_{N\ge N(\vep_0)}\sup_{|x|\le K,s\le
T} P_s^N(h_n)(x)=0.\]
The tightness of $\{X_0^N\}$ now shows 
\begin{equation}\label{meantozero}
\lim_{n\to 0}\sup_{N\ge N(\vep_0)}\sup_{s\le T}X^N_0(P^N_s(h_n))=0,
\end{equation}
and so the integral on the righthand side of
\eqref{sqfnbound1} approaches $0$ as $n\to\infty$, uniformly
in $N$.  For each $N$ fixed it is elementary to use
\eqref{SMG2} to see that $\lim_nE(\langle
M^N(h_n)\rangle_T)=0$.  It now follows from Doob's strong
inequality $L^2$ and the above that %for any $K>0$,
\begin{equation}\label{MNtightness}
\lim_{n\to\infty}\sup_{N\ge N(\vep_0)}E(\sup_{t\le T}
M_t^N(h_n)^2)=0.
\end{equation}

To prove \eqref{hndecay}, by \eqref{hndecomp} it now clearly
suffices to show
\begin{equation}\label{YNtight} \lim_{n\to\infty}\sup_{N\ge
N(\vep_0)} P(\sup_{t\le T} |Y^N_t(h_n)|\ge \vep)=0.
\end{equation} This is (115) in the proof of Lemma~6.1 in
\cite{CMP} for the Lotka-Volterra model, and the proof given
there now goes through without change in our more general
setting. The required inputs are \eqref{DN3inc1},
\eqref{DN3inc2}, \eqref{DNoneinc}, \eqref{MNtightness},
Lemma~\ref{l:CMPL35}, and \eqref{meantozero}.
\qed

\subsection{Identification of the limit}

{\sl Proof of Theorem~\ref{t:SBM}.} By the $C$-relative
compactness, established above, it remains only to show that
the sequential limit points of $\{X^N\}$ coincide.  By the
Skorokhod representation theorem we may assume that for a
sequence $N_k\uparrow\infty,\ N_k\ge N(\vep_0)$, we have
 \[X^{N_k}\rightarrow X\in
C(R_+,M_F(\R^2))\quad\text{a.s.}\] We will take limits in
\eqref{SMG1} and \eqref{SMGsqfn} to see that the law of $X$
satisfies (MP), the martingale problem characterizing the
law of SBM$(X_0,,4\pi\sigma^2,\sigma^2,\Theta)$. To see
this, fix $\Phi\in C_b^3(\R^2)$. By Lemma 2.6 of
\cite{CDP00} we have
 \[\|A_N\Phi-(\sigma^2/2)\Delta\Phi\|_\infty\to 0\hbox { as
}N\to\infty,\] and therefore by
Proposition~\ref{p:mobnds}(a) for all $t>0$,
 \begin{equation}\label{easydr}\lim_{N\to\infty}E\Bigl(\Bigl|\int_0^t
X^{N}_s(A_N\Phi)\,ds-\int_0^tX^N_s((\sigma^2/2)\Delta\Phi)\,ds\Bigr|\Bigr)=0.\end{equation}
One now can use Propositions~\ref{p:d3asymp} and
\ref{p:sqmasymp}, and \eqref{easydr} to argue exactly as in
the proof of Proposition~3.2 of \cite{CP05} and take limits
along $\{N_k\}$ in \eqref{SMG1}, \eqref{SMG2} to see that
$X$ satisfies (MP) and so is
SBM$(X_0,,4\pi\sigma^2,\sigma^2,\Theta)$. (Only the last two
paragraphs of the proof there are used.) The other inputs
needed there are the $C$-relative compactness of
$\{X^N(\Phi)\}$, $\{M^N(\Phi)\}$, and $\{D^{N,j}(\Phi)\}$
for $j=1,2,3$, established in the proof of
Lemma~\ref{l:Phit}, as well as (take $s_1=0, s_2=T$ in
\eqref{sqfncontrol})
 \[E((\langle M^N(\Phi)\rangle_T)^2)\le C(T,\Phi)\ \
\text{for all }N.\] The latter allows one to conclude that
the limiting $M(\phi)$ is a martingale with the appropriate
square function.  \qed

\section{
Proof of Proposition~\ref{p:key2bnd}}\label{sec:Ibound}
We need an elementary 
bound for $p_t^N(x):=NP(B^{N,0}_t=x)$: 
\begin{equation}\label{pNtbound}
p^N_t(x)\le \frac{C_{\ref{pNtbound}}}{t}\text{ for all $t>0$, $x\in\SN$, $N\ge3$}.
\end{equation}
For example, see (A7) in \cite{CDP00}. 

Assume $\delta_N>0$ converges to zero, and also satisfies
\begin{equation}\label{deltacondn}\liminf_N\sqrt N\delta_N>0.
\end{equation}
Fix $T\ge 1$ and consider $0\le t\le T$.  Let
$p^{N,z}_s(x)=p^N_s(z-x)$ and
$\phi^z_s=p^{N,z}_{t-s+\delta_N}$.  Argue as in the proof of
Proposition~3.10 in Section~5.2 of \cite{CMP}, using the
semimartingale decomposition \eqref{SMG1} to see that for a
universal constant $C_{\ref{3.8bound1}}$ we have
\begin{equation}\label{3.8bound1}E\Bigl(\int\int1_{\{|x-y|\le\sqrt{\delta_N}\}}dX^N_t(x)dX^N_t(y)\Bigr)\le C_{\ref{3.8bound1}}[\cT_0+\cT_1+\cT_2+\cT_3],
\end{equation}
where
\begin{align*}
\cT_0&=\frac{\delta_N}{N}\sum_{z\in\SN}X_0^N(p^{N,z}_{t+\delta_N})^2=\int\int
\delta_Np^N_{2(t+\delta_N)}(y-x)dX^N_0(x)dX^N_0(y),\\
\cT_1&=\frac{\delta_N}{N}\sum_{z\in\SN}E(\langle
M(\phi^z)\rangle_{1,t}),\\
\cT_2&=\frac{\delta_N}{N}\sum_{z\in\SN}E(\langle
M(\phi^z)\rangle_{2,t}),\\
\cT_3&=\sum_{j=2}^3\frac{\delta_N}{N}\sum_{z\in\SN}E\Bigl(\Bigl(\int_0^t
d^{N,j}(s,\xi^N_s,\phi^z)\,ds\Bigr)^2\Bigr):=\sum_{j=2}^3\cT^{(j)}_{3}.
\end{align*}

By \eqref{pNtbound}, 
\begin{equation}\label{e:T0}
\cT_0\le \frac{C_{\ref{e:T0}}\delta_N}{t+\delta_N}\, X_0^N(\1)^2.
\end{equation}

For $\cT_2$, use \eqref{SMG2},
$|c^{N,a}(x,\xi^N_s)|=|d^{N,2}(x,\xi^N_s)|$,
$|c^{N,s}(x,\xi^N_s)|=|d^{N,3}(x,\xi^N_s)|$, and then
Chapman-Kolmogorov and \eqref{pNtbound}, to conclude that
\begin{align}\label{e:T2bnd} \nonumber
\cT_2&\le\frac{\delta_N}{N}E\Bigl(\int_0^t\frac{(\log
N)^3}{(N')^2}\sum_{x\in\SN}\sum_{z\in\SN}p^N_{t-s+\delta_N}(z-x)^2\Bigl[|d^{N,2}(x,\xi^N_s)|+|d^{N,3}(x,\xi^N_s)|\Bigr]\,ds\Bigr)\\
&\le C\frac{\delta_N}{N}(\log N)^3
E\Bigl(\int_0^t\frac{N}{(N')^2} (t-s+\delta_N)^{-1}
\sum_{x\in\SN}[|d^{N,2}(x,\xi^N_s)|+|d^{N,3}(x,\xi^N_s)|]\,ds\Bigr).
\end{align} From the bounds in \eqref{dNjbnd1} we have
\[
\frac{1}{N'}\sum_{x\in\SN} |d^{N,j}(x.\xi^N_s)|\le \frac{\Vert r\Vert}{N'}\sum_{x\in\SN}\sum_{e\in\bar \cN_N}\xi_s^N(x+e)= \Vert r\Vert|\bar\cN|X^N_s(\1).
\]
Use the above in \eqref{e:T2bnd} and conclude from
Proposition~\ref{p:mobnds}(a) that
\begin{align}\label{e:T2} \cT_2\le \frac{\delta_N(\log
N)^4}{N}
E\Bigl(\int_0^t\frac{C}{t-s+\delta_N}X^N_s(\1)\,ds\Bigr) \le
C_T\delta_N
X_0^N(\1)\log\Bigl(1+\frac{t}{\delta_N}\Bigr)\frac{(\log
N)^4}{N}.
\end{align}

Turning to $\cT_1$, from \eqref{SMG2} and
Chapman-Kolmogorov, we have
\begin{align}\label{e:T1bnd} \nonumber \cT_1&=\delta_N\log N
E\Bigl(\int_0^t\frac{1}{N'}p^N_{2(t-s)+2\delta_N}(0)\Bigr[\sum_{x,y\in\SN}p_N(y-x)(\xi^N_s(x)\hxi_s^N(y)+\hxi_s^N(x)\xi^N_s(y))\Bigl]\,ds\Bigr)\\
&\le C_{\ref{e:T1bnd}}\delta_N\log
N\int_0^t(t-s+\delta_N)^{-1}E\Bigl(\frac{1}{N'}\sum_{e\in\cN_N}\sum_{x\in\SN}p_N(e)\xi_s^N(x)\hxi^N_s(x+e)\Bigr)\,ds,
\end{align}
the last by \eqref{pNtbound}.  Set
$u_N=\frac{\delta_N}{2}\wedge (\log N)^{-p}$ for some $p\ge
11$.  Note that \eqref{deltacondn} shows that $u_N$
satisfies \eqref{uNdef}.  If $\xi^N\in\{0,1\}^\SN$, define
\[G^N(\xi^N)=\frac{1}{N'}\sum_{e\in\cN_N}\sum_{x\in\SN}p_N(e)\xi^N(x)\hxi^N(x+e),\]
and for $s\ge u_N$, let
\begin{align}\label{DeltaNdef}
\Delta_N(s)=\Bigl|E(G^N(\xi^N_s)|\cF^N_{s-u_N})-\widehat
E\Bigl(\frac{1}{N'}\sum_{e\in\cN_N}\sum_{x\in\SN}p_N(e)\xi_{s-u_N}^N(B^{N,x}_{u_N})\hat\xi^N_{s-u_n}(B^{N,x+e}_{u_N})\Bigr)\Bigr|.
\end{align}
The Markov property and coalescing duality for the voter
model show that
\begin{equation}\label{DeltaNdef2}
\Delta_N(s)=\Bigl|E_{\xi^N_{s-u_N}}(G^N(\xi^N_{u_N}))-E_{\xi^N_{s-u_N}}(G^N(\xi^{N,\VM}_{u_N}))\Bigr|.
\end{equation}
We assume that $\xi^N$, $\overline\xi^N$ (the biased voter
model with rates as in \eqref{bvoterratesdef}), and
$\xi^{N,\VM}$ are constructed as in \eqref{construct} and
\eqref{couple1}, all starting at $\xi^0_N$ (with finitely
many ones as usual).  Now use the elementary inequality
\begin{align}\label{prodbnd}
|\prod_{i=1}^m\xi^N(x_i)\prod_{i=m+1}^{m+k}\hxi^N(x_i)-\prod_{i=1}^m\eta^N(x_i)\prod_{i=m+1}^{m+k}&\widehat\eta^N(x_i)|\le\sum_{i=1}^{m+k}
|\xi^N(x_i)-\eta^N(x_i)|\\ \nonumber&\ \ \forall\,
\xi^N,\eta^N\in\{0,1\}^\SN,\ x_i\in\SN,
\end{align}
with $m=k=1$, and the coupling
$\xi^{N,\VM}_{u_N}\vee\xi^N_{u_N}\le \overline\xi^N_{u_N}$,
to conclude that for all $s\ge u_N$,
\begin{align}\label{e:Deltabnd} \nonumber \Delta_N(s)&\le
E_{\xi^N_{s-u_N}}\Bigl(|G^N(\xi^N_{u_N})-G^N(\xi^{N,\VM}_{u_N})|\Bigr)\\
\nonumber&\le
E_{\xi^N_{\xi^N_{s-u_N}}}\Bigl(\frac{1}{N'}\sum_{e\in\cN_N}p_N(e)
\sum_{x\in\SN}[(\overline\xi^N_{u_N}(x)-\xi^N_{u_N}(x))+
(\overline\xi^N_{u_N}(x+e)-\xi^N_{u_N}(x+e))]\\
\nonumber&\phantom{\le
  E_{\xi_0^N}\Bigl(\frac{1}{N'}\sum_{e\in\cN_N}p_N(e)\sum_{x\in\SN}}
+[(\overline\xi^N_{u_N}(x)-\xi^{N,\VM}_{u_N}(x))
+(\overline\xi^{N}_{u_N}(x+e)-\xi^{N,\VM}_{u_N}(x+e))]\Bigr)\\
&\le 2E_{\xi^N_{s-u_N}}(2\overline
X^N_{u_N}(\1)-X_{u_N}^N(\1)-X_{u_N}^{N,\VM}(\1))\le C(\log
N)^{3-p}X^N_{s-u_N}(\1).
\end{align}
The last inequality holds by Lemma~\ref{l:masscmps}. 
For small $s$ in \eqref{e:T1bnd} we will use the crude bound
\begin{equation}\label{e:4T1bnd}
\frac{1}{N'}\sum_{e\in\cN_N}\sum_{x\in\SN}p_N(e)\xi_s^N(x)\hxi^N_s(x+e)\le X^N_s(\1).
\end{equation}
Use \eqref{DeltaNdef}, \eqref{e:Deltabnd} and \eqref{e:4T1bnd} (the latter for $s\le u_N$) in \eqref{e:T1bnd} to see that
\begin{align}\label{e:T1decomp} \nonumber \cT_1\le &
C\delta_N\log N\int_0^{u_N\wedge
t}(t-s+\delta_N)^{-1}E(X_s^N(\1))\,ds+C\delta_N\log
N\int_{u_N}^{u_N\vee
t}(t-s+\delta_N)^{-1}E(\Delta_N(s))\,ds\\
\nonumber&+C\delta_N\log N\int_{u_N}^{u_N\vee
t}(t-s+\delta_N)^{-1}E\Bigl(\widehat
E\Bigl(\frac{1}{N'}\sum_{e\in\cN_N}\sum_{x\in\SN}p_N(e)\xi_{s-u_N}^N(B^{N,x}_{u_N})\hxi^N_{s-u_n}(B^{N,x+e}_{u_N})\Bigr)\Bigr)\,ds\\
\nonumber\le& C_T\delta_N\log N X^N_0(\1)\int_0^{u_N\wedge
t}(t-s+\delta_N)^{-1}\,ds +C\delta_N(\log
N)^{4-p}X^N_0(\1)\int_{u_N}^{u_N\vee
t}(t-s+\delta_N)^{-1}\,ds\\ \nonumber&+C\delta_N\log
N\int_{u_N}^{u_N\vee t}(t-s+\delta_N)^{-1}E\Bigl(\widehat
E\Bigl(\frac{1}{N'}\sum_{e\in\cN_N}\sum_{x\in\SN}p_N(e)\xi_{s-u_N}^N(B^{N,x}_{u_N})1(\sigma^N(x,x+e)>u_N)\Bigr)\Bigr)\,ds\\
:=&\cT_{1,1}+\cT_{1,2}+\cT_{1,3},
\end{align}
where the mean mass bound from Proposition~\ref{p:mobnds}
(a) is again used in the second inequality. For $s\le
u_N(\le \delta_N/2)$ we have $s\le \delta_N/2\le
(t+\delta_N)/2$ and so
\begin{equation}\label{e:T11bnd}
\cT_{1,1}\le C_T \delta_N\log N X_0^N(\1)u_N2(t+\delta_N)^{-1}\le C_TX_0^N(\1)\frac{\delta_N}{t+\delta_N}.
\end{equation}
We also have
\begin{align}\label{e:T12bnd}
\cT_{1,2}\le C\delta_N(\log N)^{4-p}X_0^N(\1)1(t>u_N)\log \Bigl(1+\frac{t}{\delta_N}\Bigr).
\end{align}
Next use translation invariance of the coalescing walks,
Proposition~\ref{p:mobnds} (a), and \eqref{CRWold} to obtain
\begin{align}\nonumber\cT_{1,3}&\le C\delta_N\log
N\int_{u_N}^{u_N\vee
t}(t-s+\delta_N)^{-1}E\Bigl(\frac{1}{N'}\sum_{w\in\SN}\xi^N_{s-u_N}(w)\\
\nonumber&\phantom{\le C\delta_N\log N\int_{u_N}^{u_N\vee
t}(t-s+\delta_N)}\times\sum_{e\in\cN_N}p_N(e)\sum_{x\in\SN}\widehat
P(B_{u_N}^{N,0}=w-x,\sigma^N(0,e)>u_N)\Bigr)\,ds\\
\nonumber&\le C\delta_N\log
N\Bigl[\sum_{e\in\cN_N}p_N(e)\widehat
P(\sigma^N(0,e)>u_N)\Bigr]\int_{u_N}^{t\vee
u_N}(t-s+\delta_N)^{-1}E(X^N_{s-u_N}(\1))\,ds\\
\nonumber&\le C_T\delta_NX_0^N(\1)\frac{\log N}{\log
(Nu_N)}1(t>u_N)\log\Bigl(\frac{t-u_N+\delta_N}{\delta_N}\Bigr)\\
\label{e:T13bnd}&\le C_T\delta_NX_0^N(\1)\log
\Bigl(1+\frac{t}{\delta_N}\Bigr).
\end{align}
In the last line we have used $\delta_N\ge cN^{-1/2}$ for
$N$ large (which we may assume), and hence the same for
$u_N$.  Use \eqref{e:T11bnd}-\eqref{e:T13bnd} in
\eqref{e:T1decomp} to conclude
\begin{equation}\label{e:T1} \cT_1\le
C_TX_0^N(\1)\Bigl[\delta_N\log\Bigl(1+\frac{t}{\delta_N}\Bigr)+\frac{\delta_N}{t+\delta_N}\Bigr].
\end{equation}

We next decompose $\cT^{(j)}_3$ as
\begin{align}\label{T3decomp}
\nonumber\cT^{(j)}_3\le&\frac{3\delta_N}{N}\sum_{z\in\SN}E\Bigl(\Bigl(\int_{u_N}^{t\vee
u_N}d^{N,j}(s,\xi^N_s,\phi^z)-E(d^{N,j}(s,\xi^N_s,\phi^z)|\cF^N_{s-u_N})\,ds\Bigr)^2\Bigr)\\
\nonumber&+\frac{3\delta_N}{N}\sum_{z\in\SN}E\Bigl(\Bigl(\int_{u_N}^{t\vee
u_N}E(d^{N,j}(s,\xi^N_s,\phi^z)|\cF^N_{s-u_N})\,ds\Bigr)^2\Bigr)\\
\nonumber&+\frac{6\delta_N}{N}\sum_{z\in\SN}E\Bigl(\int_0^{u_N\wedge
t}\Bigl[\int_0^{s_1}d^{N,j}(s_1,\xi^N_{s_1},\phi^z)d^{N,j}(s_2,\xi^N_{s_2},\phi^z)\,ds_2\Bigr]\,ds_1\Bigr)\\
:=&\cT^{(j)}_{3,1}+\cT^{(j)}_{3,2}+\cT^{(j)}_{3,3}.
\end{align}
Equation \eqref{dNjbnd1} implies,
\begin{equation}\label{dN*bound}
\frac{1}{N'}\sum_{x\in\SN}|d^{N,j}(x,\xi^N_s)|\le
\frac{C}{N'}\sum_{x\in\SN}\Bigl[\Bigl(\sum_{e\in\bar\cN_N}\xi^N_s(x+e)\Bigr)\Bigr]\le
C_{\ref{dN*bound}}X_s^N(\1).
\end{equation} Recall that
\begin{equation}\label{dN3formula}
d^{N,j}(s,\xi^N_s,\phi^z)=\frac{\ell_N^{(j)}}{N'}\sum_{x\in\SN}p^N_{t-s+\delta_N}(z-x)d^{N,j}(x,\xi^N_s).
\end{equation}

For $\cT^{(j)}_{3,3}$, use \eqref{dN3formula} (take absolute
values and use the triangle inequality), do the sum over $z$
first, and then apply Chapman-Kolmogorov and
\eqref{pNtbound} to conclude that
\begin{align}
\label{CKstep}\cT^{(j)}_{3,3}\le
\frac{6\delta_N}{N}\int_0^{u_N\wedge
t}\int_0^{s_1}&\frac{(\log
N)^6}{(N')^2}NC_{\ref{pNtbound}}(2(t+\delta_N)-s_1-s_2)^{-1}\\
\nonumber&\times\sum_{x_1\in\SN}\sum_{x_2\in\SN}E(|d^{N,j}(x_1,\xi^N_{s_1})|\,|d^{N,j}(x_2,\xi^N_{s_2})|)\,ds_2ds_1.
\end{align} An application of \eqref{dN*bound} and the
second moment bound %Proposition~\ref{p:mobnds} (a)
\eqref{stbnd} now gives
\begin{align}\label{T33bound} \nonumber \cT^{(j)}_{3,3}&\le
C\delta_N (\log N)^6 (X_0^N(\1)+X^N_0(\1)^2)
\int_0^{u_n}\int_0^{s_1}(2(t+\delta_N)-s_1-s_2)^{-1}\,ds_2ds_1\\
&\le C\delta_N(\log N)^{6-p}(X^N_0(\1)+X_0^N(\1)^2),
\end{align}
the last by a bit of calculus. 

For $\cT^{(j)}_{3,1}$, first use the orthogonality
\[E\Bigl(\prod_{i=1}^2(d^{N,j}(s_i,\xi^N_{s_i},\phi^z)-E(d^{N,j}(s_i,\xi^N_{s_i},\phi^z)|\cF^N_{s_i-u_N}))\Bigr)=0\text{
if }s_2-u_N>s_1\] to see that
\[\cT^{(j)}_{3,1}\le\frac{6\delta_N}{N}\sum_{z\in
\SN}E\Bigl(\int_{u_N}^{u_N\vee
t}\Bigl[\int_{s_1}^{s_1+u_N}\prod_{i=1}^2\Bigl(|d^{N,j}(s_i,\xi^N_{s_i},\phi^z)|+E\Bigl(|d^{N,j}(s_i,\xi^N_{s_i},\phi^z)|\Bigr|\cF^N_{s_i-u_N}\Bigr)\Bigr)\,ds_2\Bigr]ds_1\Bigr).\]
Argue just as in \eqref{CKstep}, but now with a different
product inside the integral, to conclude that
\begin{align*} \cT^{(j)}_{3,1}\le& C\delta_N(\log
N)^6\int_{u_N}^{u_N\vee
t}\Bigl[\int_{s_1}^{s_1+u_N}\frac{1}{N'}\sum_{x_1\in\SN}\frac{1}{N'}\sum_{x_2\in\SN}(2(t+\delta_N)-s_1-s_2)^{-1}\\
&\phantom{\frac{C\delta_N(\log N)^6}{N}\int_{u_N}^{u_N\vee
t}\Bigl[\int}\times
E\Bigl(\prod_{i=1}^2\Bigl[|d^{N,j}(x_i,\xi^N_{s_i})|+E(|d^{N,j}(x_i,\xi^N_{s_i})||\cF^N_{s_i-u_N})\Bigr]\Bigr)ds_2ds_1.
\end{align*}
Now bring the sums through the expectation and product and
apply \eqref{dN*bound} to conclude
\begin{align*} \cT^{(j)}_{3,1}\le &C\delta_N(\log
N)^6\int_{u_N}^{u_N\vee
t}\Bigl[\int_{s_1}^{s_1+u_N}(2(t+\delta_N)-s_1-s_2)^{-1}E\Bigl(\Bigl(X^N_{s_1}(\1)+E(X^N_{s_1}(\1)|\cF^N_{s_1-u_N})\Bigr)\\
&\phantom{C\delta_N(\log N)^6\int_{u_N}^{u_N\vee
t}(2(t+\delta_N)-s_1-s_2)}\times\Bigl(X^N_{s_2}(\1)+E(X^N_{s_2}(\1)|\cF^N_{s_2-u_N})\Bigr)\Bigr)\,ds_2\Bigr]ds_1\\
\le &C\delta_N(\log N)^6\int_{u_N}^{u_N\vee
t}\Bigl[\int_{s_1}^{s_1+u_N}(2(t+\delta_N)-s_1-s_2)^{-1}\\
&\phantom{C\delta_N(\log N)^6\int_{u_N}^{u_N\vee t}}\times
E((X^N_{s_1}(\1)+X^N_{s_1-u_N}(\1))(X^N_{s_2}(\1)+X^N_{s_2-u_N}(\1)))\,ds_2\Bigr]ds_1,
\end{align*}
where in the last we have used the Markov property and mean
mass bound from Propostion~\ref{p:mobnds}. The second moment
bound \eqref{stbnd} from the same result therefore shows
that
\begin{align}\label{T31bound}
\nonumber\cT^{(j)}_{3,1}\le & C_T\delta_N(X^N_0(\1)+X^N_0(\1)^2)(\log N)^6u_N\int_{u_N}^{u_N\vee t} \frac{1}{2(t-s_1+\delta_N)-u_N}\,ds_1\\
\nonumber\le &C_T\delta_N(X^N_0(\1)+X^N_0(\1)^2)(\log N)^{6-p}1(t>u_N)\log\Bigl(\frac{2(t+\delta_N)-3u_N}{2\delta_N-u_N}\Bigr)\\
\nonumber\le &C_T\delta_N(X^N_0(\1)+X^N_0(\1)^2)(\log N)^{6-p}\log\Bigl(\frac{2(t+\delta_N)}{\delta_N}\Bigr)\quad\text{ (recall $u_N\le \delta_N$)}\\
= &C_T(X^N_0(\1)+X^N_0(\1)^2)(\log N)^{6-p}\delta_N\Bigr[\log 2+\log\Bigl(1+\frac{t}{\delta_N}\Bigr)\Bigl].
\end{align}

Turning to $\cT^{(j)}_{3,2}$, for $j=2,3$ introduce
\begin{equation}\label{HNdef}H^{N,j}(\xi^N_0,x,u)=E_{\xi^N_0}(d^{N,j}(x,\xi^N_{u})),
\end{equation}
and recall from \eqref{dNjvm}
 that $\hat H^{N,j}(\xi^N_0,x,u):=E_{\xi^N_0}(d^{N,j}(x,\xi^{N,\VM}_{u}))$ satisfies
\begin{equation}\label{hatHNdef}
\begin{aligned} \hat
H^{N,2}(\xi^N_0,x,u_N)&=\sum_{\emptyset\neq
A\subset\cN_N}r^{N,a}(\sqrt N A)\widehat
E(I^{N,+}(x,u_N,A,\xi_0^N)),\\ \hat
H^{N,3}(\xi^N_0,x,u_N)&=\sum_{\emptyset\neq
A\subset\cN_N}r^{N,s}(\sqrt N A)\widehat
E((I^{N,+}-I^{N,-})(x,u_N,A,\xi_0^N)).
\end{aligned}
\end{equation}
The Markov property of $\xi^N$ and then Chapman-Kolmogorov, imply that for $j=2,3$, 
\begin{align*} &\cT^{(j)}_{3,2}\\
&=\frac{6\delta_N}{N}E\Bigl(\int_{u_N}^{t\vee
u_N}\Bigl[\int_{u_N}^{s_1}\frac{(\ell_N^{(j)})^2}{(N')^2}\sum_{x_1\in\SN}
\sum_{x_2\in\SN}\Bigl[\sum_{z\in\SN}\prod_{i=1}^2p^N_{t-s_i+\delta_N}(z-x_i)\Bigr]
\prod_{k=1}^2H^{N,j}(\xi^N_{s_k-u_N},x_k,u_N)\,ds_2\Bigr]ds_1\Bigr)\\
&=6\delta_N E\Bigl(\int_{u_N}^{t\vee
u_N}\Bigl[\int_{u_N}^{s_1}\frac{(\ell_N^{(j)})^2}{(N')^2}\sum_{x_1\in\SN}\sum_{x_2\in\SN}
p^N_{2(t+\delta_N)-s_1-s_2}(x_2-x_1)\prod_{k=1}^2H^{N,j}(\xi^N_{s_k-u_N},x_k,u_N)\,ds_2\Bigr]ds_1\Bigr).
\end{align*}
The bound \eqref{dN*bound} and Proposition~\ref{p:mobnds}(a) show that 
\begin{align}\label{sumHbnd}
\frac{1}{N'}\sum_{x\in\SN}|H^{N,j}(\xi^N_0,x,u_N)|\le
E_{\xi_0^N}\Bigl(\frac{1}{N'}\sum_{x\in\SN}|d^{N^j}(x,\xi^N_{u_N})|\Bigr)
\le CE_{\xi_0^N}(X^N_{u_N}(\1))\le
C_{\ref{sumHbnd}}X^N_0(\1).
\end{align}
Similarly we have
\begin{equation}\label{sumhatHbnd}
\frac{1}{N'}\sum_{x\in\SN}|\hat H^{N,j}(\xi^N_0,x,u_N)|\le C_{\ref{sumhatHbnd}}X^N_0(\1).
\end{equation}

Now use \eqref{dNjdiff}, the coupling from \eqref{couple1},
and the triangle inequality to conclude that
\[|d^{N,j}(x,\xi_{u_N}^N)-d^{N,j}(x,\xi_{u_N}^{N,\VM})|\le
C\sum_{y\in
x+\bar\cN_N}(2\overline\xi^N_{u_N}(y)-\xi_{u_N}^N(y)-\xi_{u_N}^{N,\VM}(y))\
\ \text{ for $j=2,3$}.\]
This implies that 
\begin{align}\label{Hincbound}
\nonumber\frac{1}{N'}\sum_{x\in\SN}|H^{N,j}(\xi^N_0,x,u_N)-\hat
H^{N,j}(\xi^N_0,x,u_N)|&\le CE(2\overline
X^N_{u_N}(\1)-\underline X^N_{u_N}(\1)-\underline
X^{N,\VM}_{u_N}(\1))\\ &\le C(\log N)^{3-p}X_0^N(\1).
\end{align}
In the last we used Lemma~\ref{l:masscmps}. For $j=2,3$,
define
\[\hat\cT^{(j)}_{3,2}=6\delta_N E\Bigl(\int_{u_N}^{t\vee
u_N}\Bigl[\int_{u_N}^{s_1}\frac{(\ell_N^{(j)})^2}{(N')^2}\sum_{x_1\in\SN}\sum_{x_2\in\SN}
p^N_{2(t+\delta_N)-s_1-s_2}(x_2-x_1)\prod_{k=1}^2\hat
H^{N,j}(\xi^N_{s_k-u_N},x_k,u_N)\,ds_2\Bigr]ds_1\Bigr).\]
Use \eqref{pNtbound}, and then \eqref{sumHbnd},
\eqref{sumhatHbnd} and \eqref{Hincbound}, to conclude
\begin{align*} |\cT^{(j)}_{3,2}-\hat \cT^{(j)}_{3,2}|&\le
C\delta_N\int_{u_N}^{t\vee
u_N}\int_{u_N}^{s_1}(\ell_N^{(j)})^2(2(t+\delta_N)-s_1-s_2)^{-1}\\
&\quad\times\frac{1}{N'}\sum_{x_1\in\SN}\frac{1}{N'}\sum_{x_2\in\SN}\Bigl|E\Bigl(\prod_{k=1}^2
H^{N,j}(\xi^N_{s_k-u_N},x_k,u_N)-\prod_{k=1}^2 \hat
H^{N,j}(\xi^N_{s_k-u_N},x_k,u_N)\Bigr)\Bigr|\,ds_2ds_1\\
&\le C\delta_N(\log N)^6\int_{u_N}^{t\vee
u_N}\int_{u_N}^{s_1}(2(t+\delta_N)-s_1-s_2)^{-1}
E\Bigl(\Bigl[\frac{1}{N'}\sum_{x_1\in\SN}|H^{N,j}(\xi^N_{s_1-u_N},x_1,u_N)|\Bigr]\\
&\phantom{\le C\delta_N(\log N)^6}\times
\Bigl[\frac{1}{N'}\sum_{x_2\in\SN}|H^{N,j}(\xi^N_{s_2-u_N},x_2,u_N)-\hat
H^{N,j}(\xi^N_{s_2-u_N},x_2,u_N)|\Bigr]\Bigr)\,ds_2sds_1\\
&\quad+C\delta_N(\log N)^6\int_{u_N}^{t\vee
u_N}\int_{u_N}^{s_1}(2(t+\delta_N)-s_1-s_2)^{-1}
E\Bigl(\Bigl[\frac{1}{N'}\sum_{x_2\in\SN}|\hat
H^{N,j}(\xi^N_{s_2-u_N},x_2,u_N)|\Bigr]\\ &\phantom{\le
C\delta_N(\log
N)^6}\times\Bigl[\frac{1}{N'}\sum_{x_1\in\SN}|H^{N,j}(\xi^N_{s_1-u_N},x_1,u_N)-\hat
H^{N,j}(\xi^N_{s_1-u_N},x_1,u_N)|\Bigr]\Bigr)\,ds_2sds_1\\
&\le C\delta_N(\log N)^6(\log N)^{3-p} \int_{u_N}^{t\vee
u_N}\int_{u_N}^{s_1}(2(t+\delta_N)-s_1-s_2)^{-1}E(X^N_{s_1-u_N}(\1)X^N_{s_2-u_N}(\1))ds_2ds_1\\
&\le C_T\delta_N(\log
N)^{-2}(X_0^N(\1)+X_0^N(\1)^2)\int_{u_N}^{t\vee
u_N}\int_{u_N}^{s_1}(2(t+\delta_N)-s_1-s_2)^{-1}ds_2ds_1.
\end{align*}
In the last line we have again used
Proposition~\ref{p:mobnds}(b). The above integral is
uniformly bounded in $N$ and $t\le T$, and so for $j=2,3$,
\begin{equation}\label{TvshatT}
|\cT^{(j)}_{3,2}-\hat \cT^{(j)}_{3,2}|\le C_T\delta_N(X_0^N(\1)+X_0^N(\1)^2)(\log N)^{-2}.
\end{equation}

Next consider $\hat \cT^{(3)}_{3,2}$. For
$i\in\{2,...,|\bar\cN|\}$ and $\emptyset\neq A\subset
\cN_N$, recall the notation $I^{N,\pm}_i$ from \eqref{INi},
and let
\[\hat H^{N}_{i,A}(\xi_0^N,x,u_N)=\widehat
E(I_i^{N,+}(x,u_N,A,\xi_0^N)-I_i^{N,-}(x,u_N,A,\xi_0^N)).\]
Then from \eqref{hatHNdef} and \eqref{INrep} we have
\begin{equation}\label{^HNdecompa} \hat
H^{N,3}(\xi_0^N,x,u_N)=\sum_{i=2}^{|\bar\cN|}\
\sum_{\emptyset\neq A\subset \cN_N}r^{N,s}(\sqrt N A)\hat
H^{N}_{i,A}(\xi_0^N,x,u_N).
\end{equation}
Let $i\in\{3,\dots,|\bar\cN|\}$ and $\emptyset\neq
A\subset\cN_N$.  On $\{B_{u_N}^{N,x+A}\subset \xi_0^N,\
B_{u_N}^{N,x+\bar\cN_N\setminus A}\subset \hxi_0^N,\
|B_{u_N}^{N,x+\bar\cN_N}|=i\}$, there are distinct points
$b_1,b_2,b_3 \in \bar\cN_N$ so that
$\xi_0^N(B_{u_N}^{N,x+b_1})=1$ and
$\sigma^N_x(b_1,b_2,b_3)>u_N$.  The same conclusion holds on
$\{B_{u_N}^{N,x+\bar\cN_N\setminus A}\subset \xi_0^N,\
B_{u_N}^{N,x+A}\subset \hxi_0^N,\
|B_{u_N}^{N,x+\bar\cN_N}|=i\}$.  It follows from
\eqref{normr} that
\begin{align}\label{ige3bnd}\nonumber\sum_{i=3}^{|\bar\cN|}\
\sum_{\emptyset\neq A\subset \cN_N}&|r^{N,s}(\sqrt NA)\hat
H^{N}_{i,A}(\xi_0^N,x,u_N)|\\ &\le
C\sum_{b_1,b_2,b_3\in\bar\cN_N\text{ distinct}}\widehat
E(\xi_0^N(B^{N,x+b_1}_{u_N})1(\sigma_x^N(b_1,b_2,b_3)>u_N)).
\end{align}
In the right-hand side of the above, sum over the possible
values of $B^{N,x+b_1}_{u_N}$, as in the proof of
Lemma~\ref{l:oldl5.2}(a), and then use \eqref{CRWold} with
$n=3$ and \eqref{pNtbound} to conclude that
\begin{align}\label{ell3ormorewithp}
\nonumber\frac{1}{N'}&\sum_{x_2\in\SN}p^N_{2(t+\delta_N)-s_1-s_2}(x_2-x_1)\Bigl[\sum_{i=3}^{|\bar\cN|}\
\sum_{\emptyset\neq A\subset\cN_N}|r^{N,s}(\sqrt NA)\hat
H^{N}_{i,A}(\xi^N_{s_2-u_N},x_2,u_N)|\Bigr]\\ &\le
C_{\ref{ell3ormorewithp}}(\log
N)^{-3}(2(t+\delta_N)-s_1-s_2)^{-1}X^N_{s_2-u_N}(\1).
\end{align} The same reasoning as above %(the analagous
result in CDM2010 is (88)) gives for all $s\ge u_N$,
\begin{equation}\label{ell3ormore}
\frac{1}{N'}\sum_{x\in\SN}\Bigl[\sum_{i=3}^{{|\bar\cN|}}\
\sum_{\emptyset\neq A\subset\cN_N}|r^{N,s}(\sqrt N A)\hat
H^{N}_{i,A}(\xi^N_{s-u_N},x,u_N)|\Bigr]\le
C_{\ref{ell3ormore}}(\log N)^{-3}X^N_{s-u_N}(\1).
\end{equation}

Next consider the contribution to $\hat\cT^{(3)}_{3,2}$ from
the $i=2$ term in \eqref{^HNdecompa}.  For
$u_N\le s_i\le t$, $x_i\in\SN$ ($i=1,2$), and $k=1,2$,
define
\[\Phi^{(s_{3-k},x_{3-k})}(s_k,x_k)=p^N_{2(t+\delta_N)-s_k-s_{3-k}}(x_k-x_{3-k})=p^N_{2(t+\delta_N)-s_1-s_2}(x_2-x_1).\]
By Lemma 2.1 of \cite{CMP}, we have 
\begin{equation}\label{LippN}
\Vert \Phi^{s_{3-k}x_{3-k}}(s_k,\cdot)\Vert_{\text{Lip}}\le C_{\ref{LippN}}(2(t+\delta_N)-s_1-s_2)^{-3/2}.
\end{equation}

For $k=1,2$, by the definitions of $\hat H^{N,3}_2$ in
\eqref{^HNidef} and $\hat H^N_{2,A}$, and \eqref{hNiAdef} we
have,
\begin{align}\label{hatH2psum} \nonumber
\frac{1}{N'}&\Bigl|\sum_{x_k\in
S_N}p^N_{2(t+\delta_N)-s_1-s_2}(x_2-x_1)\sum_{\emptyset\neq
A\subset\cN}r^{N,s}(\sqrt N A)\hat
H^{N}_{2,A}(\xi^N_{s_k-u_N},x_k,u_N)\Bigr|\\
\nonumber&=(\log N)^{-3}|\hat
H^{N,3}_2(\xi^N_{s_k-u_N},u_N,\Phi^{(s_{3-k},x_{3-k})}_{s_k})|\\
&\le
C_{\ref{hatH2psum}}(2(t+\delta_N)-s_1-s_2)^{-3/2}X^N_{s_k-u_N}(\1)(\log
N)^{-p/2},
\end{align}
where in the last line we used \eqref{dN32} and \eqref{LippN}. 
For each $\emptyset\neq A\subset\cN_N$ we choose  $a\in A$.  
A much cruder calculation than that above shows that,
\begin{align}\label{hatH2sum}
\nonumber\frac{1}{N'}&\sum_{x_1\in\SN}\sum_{\emptyset\neq
A\subset\cN}|r^{N,s}(\sqrt NA)\hat
H^N_{2,A}(\xi^N_{s_1-u_N},x_1,u_N)|\\ \nonumber&\le
\frac{\|r\|}{N'}\sum_{x_1\in\SN}\sum_{\emptyset\neq
A\subset\cN}\widehat
E\Bigl(\Bigl(\xi^N_{s_1-u_N}(B^{N,x_1+a}_{u_N})+\xi^N_{s_1-u_N}(B^{N,x_1}_{u_N})\Bigr)\\
\nonumber&\phantom{\le
{\frac{\|r\|}{N'}}\sum_{x_1\in\SN}\sum_{\emptyset\neq
A\subset\cN}\widehat
E\Bigl(}\times1(|B^{N,x_1+A}_{u_N}|=|B_{u_N}^{N,x_1+\bar\cN_N\setminus
A}|=1, B_{u_N}^{N,x_1+a}\neq B_{u_N}^{N,x_1})\Bigr)\\
\nonumber&\le
\frac{\|r\|}{N'}\sum_{w\in\SN}\xi^N_{s_1-u_N}(w)\sum_{x_1\in\SN}\Bigl[\widehat
P(B^{N,a}_{u_N}=w-x_1,\sigma^N(0,a)>u_N)\\
\nonumber&\phantom{\le
\frac{\|r\|}{N'}\sum_{w\in\SN}\xi^N_{s_1-u_N}(w)\sum_{x_1\in\SN}\Bigl[}+\widehat
P(B^{N,0}_{u_N}=w-x_1,\sigma^N(0,a)>u_N)\Bigr]\\ &=
2\|r\|X^N_{s_1-u_N}(\1)\widehat P(\sigma^N(0,a)>u_N)\le
C_{\ref{hatH2sum}}X^N_{s_1-u_N}(\1)(\log N)^{-1},
\end{align}
the last by \eqref{CRWold} with $n=2$ and $u_N\ge cN^{-1/2}$
(by\eqref{deltacondn}).  Now argue exactly as in the
derivation of (89) in \cite{CMP} following the derivation
there from just below (88) to (89).  Here the hypotheses
(85)-(88) of that argument are provided by
\eqref{hatH2psum}, \eqref{hatH2sum}, \eqref{ell3ormorewithp}
and \eqref{ell3ormore}, respectively.  From this derivation
we may conclude that
\begin{align}\label{hatT32bound} \nonumber|\hat
\cT^{(3)}_{3,2}|\le &6\delta_N(\log N)^6\int_{u_N}^{t\vee
u_N}\int_{u_N}^{s_1}\Bigl|E\Bigl(\frac{1}{N'}\sum_{x_1\in\SN}\frac{1}{N'}\sum_{x_2\in\SN}p^N_{2(t+\delta_N)-s_1-s_2}(x_2-x_1)\\
\nonumber&\phantom{6\delta_N(\log
N)^6\int_{u_N}}\times\prod_{k=1}^2\Bigr[\sum_{i=2}^{|\bar\cN|}\sum_{\emptyset\neq
A\subset\cN_N}r^{\N,s}(\sqrt NA)\hat
H^{N}_{i,A}(\xi^N_{s_k-u_N},x_k,u_N)\Bigr]\Bigr)\Bigr|\,ds_2ds_1\\
\nonumber\le &C\delta_N\int_{u_N}^{t\vee
u_N}\int_{u_N}^{s_1}\frac{E(X^N_{s_1-u_N}(\1)X^N_{s_2-u_N}(\1))}{(2(t+\delta_N)-s_1-s_2)^{3/2}}ds_2ds_1(\log
N)^{6-p/2}((\log N)^{-1}+(\log N)^{-3})\\
\nonumber&\quad+C\delta_N\int_{u_N}^{t\vee
u_N}\int_{u_N}^{s_1}E(X^N_{s_1-u_N}(\1)X^N_{s_2-u_N}(\1))(2(t+\delta_N)-s_1-s_2)^{-1}ds_2ds_1\\
\le &C_T\delta_N(X_0^N(\1)+X^N_0(\1)^2).
\end{align}
In the last line we have used \eqref{stbnd}.

The corresponding bound on $\hat \cT_{3,2}^{(2)}$ is much
simpler.  The analogue of \eqref{ige3bnd} is
\[\sum_{i=2}^{\bar\cN}\sum_{\emptyset\neq
A\subset\cN}|r^{N,a}(\sqrt N A)|\hat
E(I^{N,+}_i(x,u_N,A,\xi^N_0)) \le CX_0^N(1)(\log N)^{-1}.\]
Now proceed as in \eqref{ell3ormorewithp} using the supnorm
bound on $p^N_{2(t+\delta_N-s_1-s_2}(x_1-x_2)$ to see that
\begin{align}\nonumber |\hat\cT^{(2)}_{3,2}|&\le
C\delta_N\int_{u_N}^{t\vee
u_N}\int_{u_N}^{s_1}(2(t+\delta_N)-s_1-s_2)^{-1}E(X^N_{s_1-u_N}(1)X^N_{s_2-u_N}(1))\,ds_1ds_2\\
&\le C_T\delta_N(X_0^N(1)+X_0^N(1)^2).\label{hatT232}
\end{align}

Finally use the decompositions \eqref{3.8bound1} and
\eqref{T3decomp} along with the bounds on $\cT_0$, $\cT_2$,
$\cT_1$, $\cT_{3.3}$, $\cT_{3.1}$, $|\cT^{j)}_{3.2}-\hat
\cT^{(j)}_{3.2}|$, and $|\hat T^{(j)}_{3.2}|$ from
\eqref{e:T0}, \eqref{e:T2}, \eqref{e:T1}, \eqref{T33bound},
\eqref{T31bound}, \eqref{TvshatT}, \eqref{hatT32bound} and
\eqref{hatT232}, respectively, to see that
\[E\Bigl(\int\int1_{\{|x-y|\le\sqrt{\delta_N}\}}dX^N_t(x)dX^N_t(y)\Bigr)\le
C_T(X^N_0(\1)+X_0^N(\1)^2)\delta_N\Bigl(1+\log\Bigl(1+\frac{t}{\delta_N}\Bigr)+(t+\delta_N)^{-1}\Bigr).\]
The result follows by noting that
$\log\Bigl(1+\frac{t}{\delta_N}\Bigr)+(t+\delta_N)^{-1}$ is
bounded away from $0$ uniformly in $N$, $t>0$ and so we can
drop the initial $1$ on the right-hand side.  \qed

\section{Proof of Theorem~\ref{t:ThetaOP}}\label{sec:percsetup}
We assume the hypotheses of Theorem~\ref{t:ThetaOP}, and
work in the setting of Sections~\ref{sec:proofCCT} and
\ref{sec:couplingsemimart}, so that the monotone,
asymptotically symmetric voter model perturbation,
$\{\xi^{[\vep]}:0<\vep\le \vep_0\}$ is constructed as in
Proposition~\eqref{p:infinitesde}(a), along with it's
associated measure-valued process $X^N$ in
\eqref{genlXNdef}.  For real numbers $K_0>2$ and $L'>3$, let
$I=[-L',L']^2$, $\tilde I=(-K_0L',K_0L')^2$, and recall the
notation $I_{\pm e_i}=\pm L'e_i+[-L'+1,L'-1]^2$, $i=1,2$,
from \eqref{Hyp2}. Recall also that $N$ is chosen as in
\eqref{Ndefn}. We construct our killed processes
$\uxi^{[\vep]}$ as in \eqref{killsconstruct}, where killing
is done when $|x|\ge M_0:=\lfloor \sqrt NK_0L'\rfloor$
($x\in\Z^2$).  We define
$\underline{\xi}^N_t(x)=\underline{\xi}^{[\vep_N]}_{Nt}(\sqrt
N x)$ for $x\in\SN$ as in \eqref{ulinexiNdef}, and let
$\xi^N_t(x)=\xi^{[\vep_N]}_{Nt}(\sqrt Nx)$ for $x\in\SN$ as
in Section~\ref{sec:couplingsemimart}.  The killed measure
valued process $\underline{X}^N$ is defined as in
\eqref{killedmvd} and so the killing here is done for
$|x|\ge K_0L'$ ($x\in\SN$). By
Proposition~\ref{p:infinitesde} (a),(b) our processes are
therefore coupled so that
\begin{equation}\label{barorder}\uxi^{N}\le \xi^{N}\text{
and hence }\underline{X}^N\le X^N.
\end{equation}

Recall from Section~\ref{s:qvdrifttm} that $(P^N_t,t\ge0)$
is the semigroup of a rate $N$ random walk on $\SN$ with
step kernel $p_N$.  For $x\in\SN$ let $\tilde B^{N,x}$
denote a random walk starting at $x\in\SN$ with semigroup
$(P^N_t)$.

A key step in verifying condition \eqref{Hyp2} is to show that the killed and unkilled processes are close
on certain time scales through the following version of Lemma~8.1 of \cite{CMP}. We stress that for now  $K_0$ and $L'$ are arbitrary real parameters.  

\begin{lemma}\label{l:CMPL81} There is a positive constant
$c_{\ref{l:CMPL81}}$ and a nondecreasing
function $C_{\ref{l:CMPL81}}:\R_+\to \R_+$ so that if $t>0$, 
$K_0>2$ and $L'>3$, and 
$\underline{X}^N_0 = X^N_0$ 
is supported on $I$, then 
\begin{align}\label{e:CMPL81}
\nonumber E[ X^N_t(\1) - \underline{X}^N_t(\1)] \le X_0^N(\1)\Bigl[ c_{\ref{l:CMPL81}}&e^{c_{\ref{l:CMPL81}}t}
P\Bigl( \sup_{s\le t}|\tB^{N,0}_s|>(K_0-1)L' -3 \Bigr)\\
& + C_{\ref{l:CMPL81}}(t)
(1\vee X^N_0(\1))(\log N)^{-1/6}\Bigr].
\end{align}
\end{lemma}
\noindent The Lemma is proved below but we first turn to the
main result of this Section.  Given Lemma~\ref{l:CMPL81},
the proof of Theorem~\ref{t:ThetaOP} is done just as the
proof of Lemma~6.2 of \cite{CP14} (for Lotka-Volterra
models). We outline the argument below for completeness.

\medskip

\noindent{\it Proof of Theorem~\ref{t:ThetaOP} (sketch).} Recall that we must show, after  perhaps reducing $\vep_0$,  that
\begin{align}\label{perceventbnd} \nonumber&\text{There are
$T'>1$, $K,J'\in\N$ with $K>2$, and $L'>3$, so that if
}0<\vep\le\vep_0,\\ \nonumber&\text{then for }
L=\lfloor\sqrt{N}L'\rfloor,\,\underline{X}^N_0([-L',L']^2)\ge
J'\text{ implies }\\ &P(\underline{X}_{T'}^N(I_{e})\ge
J'\text{ for all } e\in\{\pm e_i,i=1,2\})\ge
1-6^{-5(2K+1)^3}.
\end{align}
Note first that our hypotheses imply that
Theorem~\ref{t:SBMgen} holds. The limiting super-Brownian
motion in that result has drift $\Theta_2+\Theta_3$, which
is positive by hypothesis, and therefore will continue to
grow exponentially up to time $T'$ with high probability if
it has a large enough initial mass. It is therefore not hard
to prove an analogue of \eqref{perceventbnd} for this
limiting process (see (6.7) of \cite{CP14}).  By the
convergence theorem (Theorem~\ref{t:SBMgen}), the same bound
will hold for $X^N_{T'}$ if $N$ is large enough, and so
$\vep$ is sufficiently small. This is where we may need to
reduce $\vep_0$. Here we also use monotonicity to reduce to
the case where $X^N_0$ is finite and apply a subsequence
argument to assume these initial measures converge and so
the convergence theorem holds. To derive the same bound for
$\underline{X}^N$, and hence gain the necessary spatial
independence required for our comparison to oriented
percolation, we need to show $\underline{X}_{T'}^N$ is close
to $X_{T'}^N$.  This is where Lemma~\ref{l:CMPL81} is
needed.  The inputs required to carry out the proof of
Lemma~6.2 of \cite{CP14} are Lemma~\ref{l:CMPL81} and the
weak convergence of the rescaled process to SBM with a
positive drift, given here by Theorem~\ref{t:SBMgen}. Our
definition of $I_e$ for $e=\pm e_i$ is slightly different
from that in \cite{CP14} but it results in only a trivial
change.  Also we have been a bit more careful in choosing
integer parameters here. So once $K\in\N$ and $L'>3$ are
chosen as in the proof of Lemma~6.2 of \cite{CP14}, one
takes $N$ large and sets $L=\lfloor\sqrt{N}L'\rfloor$ (as in
\eqref{perceventbnd}), and then chooses $K_0=K_0(N)(>2)$ so
that $K_0L'\sqrt N=KL$. (In this way $K_0$ is comparable to
$K$.) This equality ensures killing for the unscaled process
outside $(-KL,KL)^2$ corresponds to killing
$\underline{X}^N$ outside $\tilde I$ as in
Lemma~\ref{l:CMPL81}. (Note that $K_0$ may depend on $N$ in
Lemma~\ref{l:CMPL81}.) The rest of the argument is now
identical to that of Lemma~6.2 in \cite{CP14}. \qed

\medskip

Recall from Section~\ref{s:qvdrifttm} that
$\{B^{N,x}:\,x\in\SN\}$ is a system of rate $w_NN$
coalescing random walks in $\SN$ with step kernel $p_N$.
Let $T'_x=\inf\{t\ge0:B_t^{N,x}\notin \tilde I\}$, let
$\Delta$ denote a cemetery state, and define a ``killed''
coalescing random walk system,
$\{\underline{B}^{N,x}_t,x\in\SN\}$, by
\[
\underline{B}^{N,x}_t = \begin{cases}
B^{N,x}_t&\text{if }t<T'_x\\
\Delta & \text{if }t\ge T'_x.
\end{cases}
\]
We define the killed random walk $\tilde\uB^{N,x}$ (recall
the step rate here is $N$ and there is no coalescing) in the
same way and denote its associated killed semigroup by
$(\uP^N_t,t\ge 0)$. Of course,
$\uB^{N,x}_t=\tilde\uB^{N,x}\equiv \Delta$ for all
$x\notin \tilde I$.  We will use the convention that
$\xi(\Delta)=0$ for all
$\xi\in\{0,1\}^\SN$.  
\medskip

\noindent{\sl{Proof of Lemma~\ref{l:CMPL81}.}} We follow the proof of
Lemma~8.1 in \cite{CMP} for $2$-dimensional Lotka-Volterra
models, but some modifications are needed. Assume $X^N_0$
(and hence $\xi^N_0=\uxi^N_0$) is supported on
$I=[-L',L']^2$ and $T'>0$.  Let $f:\SN\cup\{\Delta\}\to\R$
with $f(\Delta)=0$ and set $\Phi(s,x) = \uP^N_{t-s}f(x)$,
$s\le t$.  We will assume $t\in[0,T']$ in what follows. The
killed analogue of \eqref{SMG1} is derived as in the proof
of Lemma~3.2 of \cite{CP07} where a general class of voter
model perturbations is considered. The argument there uses a
different representation for the perturbation but applies to
our representations without change and gives (see the last
display on p. 113 of \cite{CP07})
\begin{equation}\label{CMP121}
\underline{X}^N_t(f) =
\underline{X}^N_0(\uP^N_{t}f)
+ \int_{0}^t \sum_{j=2}^3d^{N,j}(s,\underline{\xi}^N_s,\Phi)ds
+ \underline{M}^N_t(\Phi),
\end{equation}
where $\underline{M}^N_t(\Phi)$ is a square-integrable, mean
zero martingale.  Next, choose $h:\R^2\cup\{\Delta\}\to
[0,1]$ such that $h(\Delta)=0$ and
\[
[-K_0L'+3,K_0L'-3]^2\subset\{h=1\} \subset
\text{Supp}(h)\subset [-K_0L'+2,K_0L'-2]^2,
\quad |h|_{\Lip}\le 1,
\]
and define, for $s\le t$ and $x\in\SN$, $\Psi(s,x) =
\underline{P}^N_{t-s}h(x).$ By Lemma~8.4 in \cite{CMP} there
is a constant $C_{\ref{Psi1/2}}>0$ such that
\begin{equation}\label{Psi1/2}
\|\Psi\|_{N}\le C_{\ref{Psi1/2}}\ \ \text{ for all }N.
\end{equation}
By \eqref{semimartI} (with $c=0$ and $\Phi=\1$),
\eqref{CMP121} (with $\Phi=\Psi$ and $f=h$), the inequality
$h\le 1$ and $X^N_0=\uX^N_0$, we have
\begin{align}\label{CMP122}
\nonumber E[X^N_t(\1)-\underline{X}^N_t(\1)]& \le
E[X^N_t(\1)-\underline{X}^N_t(h)] \\ 
&=
X^N_0(\1-\uP^N_t(h)) + E\Big[
\int_0^t\sum_{j=2}^3(d^{N,j}(s,\xi^N_s,\1) -
d^{N,j}(s,\underline{\xi}^N_s,\Psi))\, ds\Big]. 
\end{align}
It follows that
\begin{equation}
E[ X^N_t(\1)-\uX^N_t(\1)] \le
\cU_0+\cU_1+\cU_2+\cU_3+\cU_4,
\end{equation}
where
\begin{align*}
\cU_0 &= X^N_0(\1-\uP^N_t(h)),\\
\cU_1 &= E\Big[\int_0^{t_N\wedge t}\sum_{j=2}^3 (d^{N,j}(s,\xi^N_s,\1)
-d^{N,j}(s,\uxi^N_s,\Psi))\,ds\Big],\\
\cU_2 &= E\Big[\int_{t_N}^{t\vee t_N}\sum_{j=2}^3d^{N,j}(s,\xi^N_s,\1-\Psi)\, ds\Big] ,\\
\cU_3 &= E\Big[\int_{t_N}^{t\vee t_N}(d^{N,2}(s,\xi^N_s,\Psi)
-d^{N,2}(s,\uxi^N_s,\Psi))\,ds\Big],\\
\cU_4 &= E\Big[\int_{t_N}^{t\vee t_N}(d^{N,3}(s,\xi^N_s,\Psi)
-d^{N,3}(s,\uxi^N_s,\Psi))\,ds\Big].
\end{align*}
The labeling matches that of (127) in \cite{CMP}.

We claim that there is a positive function
$C_{\ref{U1}}:(0,\infty)\to(0,\infty)$ and positive
constants $\tilde K$, $c_{\ref{U2}}$ such that for any
$t\le T'$, if $|\cdot|_\infty$ denotes the $L^\infty$ norm
on $\R^d$, then
\begin{align}
|\cU_0| &\le X^N_0(\1) P\Big(\sup_{u\le t}
|\tB^{N,0}_u|_\infty>(K_0-1)L'-3\Big)\label{U0},\\
|\cU_1|&\le C_{\ref{U1}}(T')X^N_0(\1)(\log N)^3 t_N \label{U1},\\
\nonumber|\cU_2|&\le X^N_0(\1)\Big[c_{\ref{U2}}e^{c_{\ref{U2}} t}\hat P\Big(\sup_{u\le t}
|{\tB}^{N,0}_u|_\infty > (K_0-1)L'-3\Big)\\
&\qquad\qquad\qquad+ C_{\ref{U1}}(T')(1\vee X^N_0(\1))(\log N)^{-1/2}\Big] \label{U2},\\
|\cU_j|&\le C_{\ref{U1}}(T')
(\log N)^{-1/6}(X^N_0(\1)+X^N_0(\1)^2)+
\tilde K \int_0^tE[X^N_s(\1)-\underline{X}^N_s(\1)]ds,\ j=3,4. \label{U4}
\end{align}
Assuming 
\eqref{U0}--\eqref{U4} for now, and recalling that $t_N=(\log N)^{-19}$,
 we see that for some function $C:(0,\infty)\to(0,\infty)$, and all $t\le
T'$,  
\begin{multline*} E[X^N_t(\1)-\uX_t^N(\1)] \le
X^N_0(\1)\Big[ c_{\ref{U2}}e^{c_{\ref{U2}} t}P\Bigl(
\sup_{s\le t}|{\tB}^{N,0}_s|_\infty>(K_0-1)L' -3 \Bigr)\\ +
C(T')(1\vee X^N_0(\1))(\log N)^{-1/6} \Big]+ 2\tilde K
\int_0^t E[X^N_s(\1)-\uX^N_s(\1)]ds.
\end{multline*}
By replacing $C(T')$ with $\inf_{T\ge T'}C(T)$ we may assume
that $C(\cdot)$ is non-decreasing. Now take $T'=t$ in the
above and use Gronwall's inequality to complete the proof of
Lemma~\ref{l:CMPL81}.

Thus, our remaining task in this section is to verify
\eqref{U0}--\eqref{U4}.  Equation \eqref{U0} is (128) of
\cite{CMP} (which uses the fact that $X_0^N$ is supported on
$I$ and so applies here as well).  Equation \eqref{U1}
follows from \eqref{dn3roughbnd} and its counterpart for
$\underline{\xi}^N$ (the proof is the same), the mean mass
bound Proposition~\ref{p:mobnds}(a), and $\underline{X}^N\le
X^N$.

The derivation of \eqref{U2} follows that of (130) in
\cite{CMP}. Since it does not involve $\underline{\xi}^N_t$
we can apply the bounds from Section~\ref{s:qvdrifttm}. We
have
\begin{align} \nonumber|\cU_2|\le&
\Big|E\Bigl(\int_{t_N}^{t\vee t_N}
\sum_{j=2}^3\big(E(d^{N,j}(s,\xi^N_s,1-\Psi)|\cF_{s-t_N})-\Theta^N_jX^N_{s-t_N}(1-\Psi_{s-t_N})\big)\,ds\Bigr)\Big|\\
\nonumber&\qquad+\Big|E\Bigl(\int_{t_N}^{t\vee t_N}
(\Theta^N_2+\Theta^N_3)X^N_{s-
t_N}(1-\Psi_{s-t_N})\,ds\Bigr)\Bigr|\\
\label{U2a}:=&|\Delta_1|+|\Delta_2|.
\end{align}
For $\Delta_1$, first use Proposition~\ref{r:conddriftbnd}
and \eqref{Psi1/2}, and then Corollary~\ref{c:keyint} and
Proposition~\ref{p:mobnds}(a) to see that for some
$C_{\ref{U2b}}(T')$ (and all $t\le T'$)
\begin {equation}\label{U2b}
|\Delta_1|\le C_{\ref{U2b}}(T')(\log N)^{-1/2}(X^N_0(\1)+X^N_0(\1)^2).
\end{equation}
For $\Delta_2$, use \eqref{Tht2lim}, \eqref{Tht3lim}, and
Lemma~\ref{l:CMPL35} (recall
$\Vert\Psi_{s-t_N}\Vert_{\Lip}\le C_{\ref{Psi1/2}}$ by
\eqref{Psi1/2} and take $T'\ge C_{\ref{Psi1/2}}$ to apply
Lemma~\ref{l:CMPL35}), and then use
$P^N_{s-t_N}\Psi_{s-t_N}\ge\underline{P}^N_{s-t_N}\Psi_{s-t_N}=\underline{P}^N_th$
to see that for some constants $c_i>0$,
\begin{align} \nonumber|\Delta_2|&\le c_1\int_{t_N}^{t\vee
t_N}e^{c_{\ref{firstmomeasbnd}}(t-s)}X_0^N(P^N_{s-t_N}(1-\Psi_{s-
t_N}))\,ds+C_{\ref{firstmomeasbnd}}(T')(\log
N)^{-1/2}(X^N_0(\1)+X_0^N(\1)^2)\\ \nonumber&\le
c_1\int_{t_N}^{t\vee
t_N}e^{c_{\ref{firstmomeasbnd}}(t-s)}X_0^N(1-\underline{P}^N_t
h)\,ds+C_{\ref{firstmomeasbnd}}(T')(\log
N)^{-1/2}(X^N_0(\1)+X_0^N(\1)^2)\\
\label{U2c}&\le
c_2e^{c_{\ref{firstmomeasbnd}}t}X_0^N(\1)P(\sup_{u\le
t}|\tB^{N,0}_u|_\infty>(K_0-1)L'-3)+C_{\ref{firstmomeasbnd}}(T')(\log
N)^{-1/2}(X^N_0(\1)+X_0^N(\1)^2),
\end{align}
where \eqref{U0} is used in the last line.  So by
\eqref{U2a}-\eqref{U2c}, we have \eqref{U2}.

We turn now to the more involved proof of \eqref{U4}.
Recall from Section~\ref{ss:compproc}, that $\xi_t^{N,\VM}$
denotes a rescaled voter model on $\SN$ with rate function
$Nw_Nc^{N,\VM}(x,\xi)$. Let $\uxi^{N,\VM}$ denote the
corresponding killed voter model, which has rate function
$Nw_Nc^{N,\VM}(x,\xi)1(x\in\tilde I)$ and initial condition
$\uxi_0^N$ supported on $\tilde I$. We will assume
$\xi^{N,\VM}$ has the same initial condition and so by the
monotonicity of the voter model, just as in
\eqref{barorder}, we may assume
\begin{equation}\label{votercouplin}\uxi^{N,\VM}\le
\xi^{N,\VM}.
\end{equation}
We will also use the following killed duality equation which
is a special case of (9.36) in \cite{CP08}:
\begin{align}\nonumber &\text{For $\uxi^N_0$ supported in $\tilde I$ and 
finite disjoint $A,B\subset \SN$,} \\
\label{kdualeq}
&E_{\uxi^N_0}\Big[\prod_{a\in A}\uxi^{N,\VM}_t(a)
\prod_{b\in B}(1- \uxi^{N,\VM}_t(b))\Big]
= \widehat E\Big[\prod_{a\in A}\uxi^N_0(\uB^{N,a}_t)
\prod_{b \in B}(1- \uxi^N_0(\uB^{N,b}_t))
\Big].
\end{align}

In view of the above we assume for now that $\uxi^N_0$ is
supported on the larger set $\tilde I$. Recalling
\eqref{hatHNdef}, for $u>0$ and $j=2,3$ we define
\begin{equation}
\hat
\uH^{N,j}(\uxi^N_0,x,u)=E_{\uxi^N_0}(d^{N,j}(x,\uxi_u^{N,\VM})).
\end{equation}
By \eqref{kdualeq} and just as in the derivation of
  \eqref{dNjvm},
  \begin{align*}
\hat\uH^{N,2}(\uxi^N_0,x,u)&= \sum_{\emptyset\neq
    A\subset   \cN_N}r^{N,a}(\sqrt NA)\widehat
  E\Big(\prod_{y\in 
  A}\uxi^N_0(\uB^{N,x+y}_{u})\prod_{z \in \bar\cN_N\setminus
  A}(1- \uxi^N_0(\uB^{N,x+z}_{u}))\Big)\\  
\hat\uH^{N,3}(\uxi^N_0,x,u)&= \sum_{\emptyset\neq A\subset
  \cN_N}r^{N,s}(\sqrt NA)\widehat E\Big(\prod_{y\in
  A}\uxi^N_0(\uB^{N,x+y}_{u})\prod_{z \in \bar\cN_N\setminus
  A}(1- \uxi^N_0(\uB^{N,x+z}_{u}))\\  
&\phantom{= \sum_{\emptyset\ne A\subsetneq \cN}
r^{\vep_N}_{|A|}}- \prod_{y\in A}(1-\uxi^N_0(\uB^{N,x+y}_{u}))
\prod_{z \in \bar\cN_N\setminus A}\uxi^N_0(\uB^{N,x+z}_{u})
\Big).
\end{align*}
With the definitions \eqref{^HN2} and \eqref{^HN3} in mind
we also introduce
\begin{align*}
\hat\uH^{N,j}(\uxi^N_0,u,\Psi_s) &=
\frac{\ell^{(j)}_N}{N'}\sum_{x\in \SN}\Psi(s,x)
\hat \uH^{N,j}(\uxi^N_0,x,u),\ j=2,3.
\end{align*}
As for \eqref{dN3rep1}, \eqref{kdualeq} implies for $j=2,3$,
\[
E_{\uxi^N_0}[d^{N,j}(s,\uxi^{N,\VM}_{u},\Psi)] =
\hat\uH^{N,j}(\uxi^N_0,u,\Psi_s) .
\]

The next result is a killed version of Lemma~\ref{l:d3H}. We
give the proof at the end of this section.
\begin{lemma}\label{l:d3H_} There is a constant
$C_{\ref{d3H_}}$ such that for $j=2,3$, all $T'>0$, $\Phi\in
C_b([0,T']\times\R^2)$ and all $s\in[t_N,T']$,
\begin{equation}\label{d3H_}
\Big|E_{\uxi^N_0}\big[d^{N,j}(s,\uxi^N_{s},\Phi)|\cF^N_{s-t_N}\big]
- \hat \uH^{N,j}(\uxi^N_{s-t_N},t_N,\Phi_{s-t_N}) \Big| \le
C_{\ref{d3H_}}\|\Phi\|_{1/2,N}(\log
N)^{-13/2}\uX^N_{s-t_N}(\1).
\end{equation}
\end{lemma}

Let us assume again that $\xi^N_0=\uxi^N_0$ is supported on
the smaller set $I$ (as opposed to $\tilde I$), and for
$j=2,3$ write $\cU_{j+1}=\sum_{i=1}^5\cV_{j,i}$, where
\begin{align}
\cV_{j,1} &= E\Big[\int_{t_N}^{t_N\vee t}[d^{N,j}(s,\xi^N_s,\Psi)
- \Theta^N_jX^N_{s-t_N}(\Psi)]ds\Big],\\
\cV_{j,2}&=
\Theta^N_jE\Big[\int_{t_N}^{t_N\vee t}[X_{s-t_N}(\Psi) -
\uX_{s-t_N}(\Psi)] ds\Big],\\
\cV_{j,3}&=
E\Big[\int_{t_N}^{t_N\vee t}[
\Theta^N_j\uX_{s-t_N}(\Psi) -
\hat H^{N,j}(\uxi^N_{s-t_N},t_N,\Psi_{s-t_N})]ds\Big],\\
\cV_{j,4}&=
E\Big[\int_{t_N}^{t_N\vee t}[
\hat H^{N,j}(\uxi^N_{s-t_N},t_N,\Psi_{s-t_N})-
\hat\uH^{N,j}(\uxi^N_{s-t_N},t_N,\Psi_{s-t_N})]ds\Big],\\
\cV_{j,5}&=
E\Big[\int_{t_N}^{t_N\vee t}[
\hat\uH^{N,j}(\uxi^{N}_{s-t_N},t_N,\Psi_{s-t_N})
- d^{N,j}(s,\uxi^N_s,\Psi)]ds.
\end{align}
We bound these one at a time.

By \eqref{dsThtX} and  \eqref{keyint},
\begin{align}
|\cV_{j,1}| &\le  C_{\ref{dsThtX}}\|\Psi\|_{N}
\int_{0}^{(t-t_N)^+}
 \Bigl(\frac{1}{t_N\log 
  N}E[\scrI^N(\xi^N_{s})]+(\log
N)^{-1/2}E[X_{s}^N(\1)]\Bigr)ds\label{U41bnd0}\\
&\le C_{\ref{dsThtX}} \|\Psi\|_{N}\Big[C_{\ref{keyint}}(T')  (\log
N)^{-1/2} \Big(X^N_0(\1)+X^N_0(\1)^2\Big) + (\log
N)^{-1/2}\int_0^{(t-t_N)^+}E[X^N_s(\1)]ds \Big]\nonumber\\
&\le C_{\ref{U41bnd}}(T')(\log N)^{-1/2}
\Big(X^N_0(\1)+X^N_0(\1)^2\Big), \label{U41bnd}
\end{align}
where we have used
Proposition~\ref{p:mobnds}(a) and \eqref{Psi1/2} in the last
step.  
Recalling that $\uxi^N_t\le \xi^N_t$, letting
$\tilde K=\sup_N (|\Theta^N_2|+|\Theta^N_3|)<\infty$ 
  and using 
$\|\Psi\|_\infty\le 1$,  we get
\begin{equation}\label{U42bnd}
|\cV_{j,2}| \le \tilde K 
\int_{0}^{(t-t_N)^+}E[X_{s}(\1) -
\uX_{s}(\1)] ds. 
\end{equation}
Change variables in $\cV_{j,3}$ to rewrite it as an integral over
$[0,(t-t_N)^+]$, set $u=s$ and replace $\xi^N_0$ with $\uxi^N_{s}$
in  Proposition~\ref{p:^HTht2} and Proposition~\ref{p:^HTht}
to obtain 
\[
|\cV_{j,3}| \le
C(T')\|\Psi\|_{N}
\int_{0}^{(t-t_N)^+}
 \Bigl(\frac{1}{t_N\log 
  N}E[\scrI^N(\uxi^N_{s})]+(\log
N)^{-1/2}E[\uX_{s}^N(\1)]\Bigr)ds,
\]
which is the same as the right-hand side of \eqref{U41bnd0},
but with $\uxi^N_s$ instead of $\xi^N_s$. Since $\uxi^N_s\le
\xi^N_s$, \eqref{U41bnd} gives
\begin{equation}\label{U43bnd}
|\cV_{j,3}| \le
C_{\ref{U41bnd}}(\log N)^{-1/2}
\Big(X^N_0(\1)+X^N_0(\1)^2\Big).
\end{equation}
By \eqref{Psi1/2}, Lemma~\ref{l:d3H_}, $\uX^N(\1)\le X^N(\1)$ and
Proposition~\ref{p:mobnds}(a),
\begin{equation}\label{U45bnd}
|\cV_{j,5}| \le C_{\ref{d3H_}}\|\Psi\|_{1/2,N}
(\log N)^{-13/2}\int_0^{(t-t_N)^+} E[\uX^N_s(\1)]ds
\le C_{\ref{U45bnd}}(T')(\log N)^{-13/2}X^N_0(\1). 
\end{equation}

Turning to $\cV_{j,4}$, we use $\Psi_s\equiv 0$ on $(\tilde I)^c$ and 
$\|\Psi\|_\infty \le 1 $
to obtain
\begin{equation}
\Big|\hat H^{N,j}(\uxi^N_{s},t_N,\Psi_{s})-
\hat\uH^{N,j}(\uxi^N_{s},t_N,\Psi_{s})\Big|
\le \frac{(\log N)^3}{N'}\sum_{x\in \tilde I}
|\hat H^{N,j}(\uxi^N_s,x,t_N)
-\hat\uH^{N,j}(\uxi^N_s,x,t_N)|.
\label{H-uH}
\end{equation}
For $x\in \SN$ we may use \eqref{dNjdiff}, the coupling
\eqref{votercouplin}, and then duality (recall
\eqref{kdualeq}) to see that
\begin{align}
\nonumber\Big|
\hat H^{N,j}(\uxi^N_s,x,t_N)- \hat\uH^{N,j}(\uxi^N_s,x,t_N) \Big|
\nonumber&=\Big|E_{\uxi^N_s}[d^{N,j}(x,\xi_{t_N}^{N,\VM})-d^{N,j}(x,\uxi_{t_N}^{N,\VM})]\Big|\\
\nonumber&\le 2\|r\|E_{\uxi^N_s}\Bigl[\sum_{y\in\bar\cN_N}(\xi^{N,\VM}_{t_N}(x+y)-\uxi^{N,\VM}_{t_N}(x+y))\Bigr]\\
\label{hatHinc}&=2\|r\|\sum_{y\in\bar\cN_N}\hat E[\uxi^N_s( B^{N,x+y}_{t_N})-\uxi^N_s( {\underline{B}}^{N,x+y}_{t_N})].
\end{align}
Now define
\[
\cA(\delta) = \{x\in\SN\cap \tilde I: \inf_{y\in(\tilde I)^c}|x-y|\le
\delta\} \ \ \text{and}\ \ 
\cA(\delta)' = \{x\in\SN\cap \tilde I: \inf_{y\in(\tilde I)^c}|x-y|>
\delta\} .
\]
We will decompose the sum in \eqref{H-uH} into sums over 
$x\in\cA(t_N^{1/3})$ and $x\in \cA(t_N^{1/3})'$.
By \eqref{hatHinc}, 
\begin{align}
\frac{(\log N)^3}{N'}&\sum_{x\in\cA(t_N^{1/3})}
|\hat H^{N,j}(\uxi^N_s,x,t_N)- \hat \uH^{N,j}(\uxi^N_s,x,t_N)|
\nonumber\\
&\le 2\|r\|\frac{(\log N)^3}{N'}\sum_{x\in\cA(t_N^{1/3})}
\hat E\Bigl[\sum_{y\in\bar\cN_N}\uxi^N_s(B^{N,x+y}_{t_N})\Bigr]\nonumber\\
&=2\|r\|\frac{(\log N)^3}{N'}\sum_{x\in\cA(t_N^{1/3})}
\sum_{w\in\SN}\uxi^N_s(w)
\sum_{y\in\bar\cN_N}\hat P(B^{N,y}_{t_N}=w-x)\nonumber\\
&\le 2\|r\| \frac{(\log N)^3}{N'}
\Big[ \sum_{w\notin\cA(2t_N^{1/3})}\uxi^N_s(w)
\sum_{y\in\bar\cN_N}\hat
P(|B^{N,y}_{t_N}|_\infty > t_N^{1/3})
+ |\bar\cN|\sum_{w\in\cA(2t_N^{1/3})}\uxi^N_s(w)\Big]
\nonumber\\
&\le C_{\ref{H-uH1}} |\bar\cN|(\log N)^3 \Big[X^N_s(\1) (N^{-1}+t_N)t_N^{-2/3} +
X^N_{s}(\cA(2t_N^{1/3}))\Big]\label{H-uH1},
\end{align}
for some constant $C_{\ref{H-uH1}}$. In the last line we
used $\uX^N\le X^N$ and Chebychev's inequality, and in the
next to last inequality we noted that for
$w\notin\cA(2t_N^{1/3})$ and $x\in\cA(t_N)^{1/3}$, we must
have $|w-x|>t_N^{1/3}$.  To bound
$E[X^N_{s}(\cA(2t_N^{1/3}))]$ we will need a bound on the
mean measure similar to Lemma~\ref{l:CMPL35}, but now in
terms of the $L^1$ norm of $\Psi:\SN\to R_+$,
$\|\Psi\|_1=\frac{1}{N}\sum_{x\in\SN}\Psi(x)$.  The result
we need is
\[E(X^N_s(\Psi))\le X_0^N(P^N_s(\Psi))+C_TX^N_0(\1)(\log
N)^3\|\Psi\|_1\Bigl(1+\log\Bigl(1\vee\frac{1}{\|\Psi\|_1}\Bigr)\Bigr)\quad\forall
s\le T'.\] This follows easily from \eqref{semimartI} (with
$c=0$) and \eqref{dNjbnd1} as in the proof of Lemma~8.6 of
\cite{CMP}.  Use this with
$\Psi=1\bigl\{\cA(2t_N^{1/3})\bigr\}$ to take means in
\eqref{H-uH1}, recalling that $X^N_0$ is supported on $I$ to
bound $X^N_0(P^N_s(\Psi))$, and obtain
\begin{equation}
\frac{(\log N)^3}{N}\int_0^{(t-t_N)^+} \sum_{x\in\cA(t_N^{1/3})}
E\Big[|\hat H^{N,j}(\uxi^N_s,x,t_N)- \hat\uH^{N,j}(\uxi^N_s,x,t_N)|\Big]
ds
 \le C_{\ref{H-uH2}}(T') X^N_0(\1) (\log N)^{-1/6}.\label{H-uH2}
\end{equation}
See the derivation of (141) in \cite{CMP} for the details.

If $x\in \cA(t_N^{1/3})'$ and $y\in\bar\cN_N$, then
$\uxi^{N}_{s}(B_{t_N}^{N,x+y})-\uxi^{N}_{s}(\underline{B}_{t_N}^{N,x+y})\neq
0$ implies that $B^{N,x+y}$ exits $\tilde I$ before $t_N$,
and hence moves a distance exceeding $t_N^{1/3}$ from $x$ by
time $t_N$.  Therefore \eqref{hatHinc} implies
\begin{align*}
\Big|
\hat H^{N,j}(\uxi^N_s,x,t_N,)- \hat\uH^{N,j}(\uxi^N_s,x,t_N) \Big|
\le2\|r\|\sum_{y\in \bar\cN_N}\widehat E\Big[\uxi^N_s(B^{N,x+y}_{t_N})
1\Big\{ \sup_{u\le t_N}|B^{N,x+y}_u-x|>t_N^{1/3}
\Big\}\Big].
\end{align*}
With this bound, using Proposition~\ref{p:mobnds}(a), we have
\begin{align}
\frac{(\log N)^3}{N'}&\int_0^{t-t_N^+}\sum_{x\in\cA(t_N^{1/3})'}
E\Big[\big|\hat H^{N,j}(\xi^N_s,x,t_N)- \hat
\uH^{N,j}(\xi^N_s,x,t_N)\big|\Big]ds
\nonumber\\
&\le \frac{2\|r\|(\log N)^3}{N'}\int_0^t
E\Big[\sum_{w\in\SN}\sum_{x\in\cA(t_N^{1/3})'}
\sum_{y\in\bar\cN_N} 
\uxi^N_s(w)\widehat P\Big(B^{N,y}_{t_N}=w-x,
\sup_{u\le t_N}|B^{N,y}_u|>t_N^{1/3}\Big)\Big]ds
\nonumber\\
&\le 2\|r\|(\log N)^3\sum_{y\in\bar\cN_N}\widehat P(
\sup_{u\le t_N}|B^{N,y}_u|>t_N^{1/3})
\int_0^tE[\uX^N_s(\1)]ds\nonumber \\
&\le2\|r\| C_{\ref{p:mobnds}}tX^N_0(\1)(\log N)^3 \sum_{y\in\bar\cN_N}\widehat P(
\sup_{u\le t_N}|B^{N,y}_u|^2>t_N^{2/3})\nonumber \\
&\le C_{\ref{H-uH3}}(T')X^N_0(\1)(\log N)^{-10/3},\label{H-uH3}
\end{align}
where the weak $L^1$ inequality for nonnegative submartingales is used in the last. 
Use the above and \eqref{H-uH2} in \eqref{H-uH} to
%Combining \eqref{H-uH}, \eqref{H-uH1} and \eqref{H-uH2} we
obtain
\begin{equation}\label{U44bnd}
|\cV_{j,4}| \le C_{\ref{U44bnd}}(T')(\log N)^{-1/6}X^N_0(1).
\end{equation}
%for a constant $C_{\ref{U44bnd}}>0$ depending on $T'$. 
The inequalities \eqref{U41bnd}, \eqref{U42bnd},
\eqref{U43bnd}, \eqref{U45bnd} and \eqref{U44bnd} imply 
\eqref{U4}, and we are done.
\qed

\medskip

\noindent{\sl{Proof of Lemma~\ref{l:d3H_}.}} From
Section~\ref{ss:compproc} we may recall the bounding biased
voter model $\bar{\xi}^N$ with rate function ${\bar
c}^{N,b}$ as in \eqref{bvoterratesdef}, constructed as in
\eqref{construct}.  We then construct the killed version of
$\bar{\xi}$, denoted $\bar{\uxi}$, using the same equation
with rate function
\[{\bar{\underline{c}}}^{N,b}(x,\xi)={\bar
c}^{N,b}(x,\xi)1(x\in \tilde I).\] We assume all these
processes have the same initial state $\uxi^N_0=\xi^N_0$,
supported on $\tilde I$, as our previous processes. In
Section~\ref{ss:compproc} we verified condition
\eqref{compcond} for the rates of $(\xi^N,\bxi^N)$ and for
the rates of $(\xi^{N,\VM},\bxi^N)$.  The same condition is
then immediate for the corresponding killed processes as one
simply multiplies the required inequalities by $1(x\in\tilde
I)$.  So we may now apply \eqref{SDEcoupling} for the killed
processes to conclude that
\begin{equation}\label{orderings} \uxi^N\le
{\bar{\uxi}}^N\quad\text{and}\quad \uxi^{N,\VM}\le
{\bar{\uxi}}^N.
\end{equation}
If
$\bar{\underline{X}}_t=\frac{1}{N'}\sum_{x\in\SN}{\bar{\uxi}}^N_t(x)\delta_x$  and $c_0=\Vert r\Vert(2+\underline{p}^{-1})$, 
then we may use the last display on p. 113 of \cite{CP07} to
conclude that
\begin{equation}\label{meankbvm}
E[\bar{\underline{X}}_t^N(\1)]={\underline
X}^N_0({\underline P}^N_t(\1))+c_0\int_0^t\frac{(\log
N)^3}{N'}\sum_{x\in\SN}{\underline P}^N_{t-s}(\1)(x)(1-{\bar
\uxi}^N_s(x)\sum_{y\in\cN_N}{\bar\uxi}^N_s(x+y)\,ds.
\end{equation}
To see this, first note that the setting in \cite{CP07} is
for a general class of voter model perturbations. This
includes our biased voter model with (the notation is from
\cite{CP07}) $\delta_N\equiv 0$, and
$\beta_N(\{y\})=c_0(\log N)^3$ for $y\in \cN_N$, while
$\beta_N$ is zero otherwise.  The above formula then follows
from p. 113 of \cite{CP07}. Now take the difference of
\eqref{meankbvm} with \eqref{CMP121} and use \eqref{crude1}
to conclude (recall \eqref{orderings})
\begin{align} \nonumber E[{\bar{\underline
X}}_{t_N}(\1)-{\underline X}^N_{t_N}
(\1)]&=E\Bigl[\int_0^{t_N}\Bigl(c_0\Bigl\{\frac{(\log
N)^3}{N'} \sum_{x\in\SN}{\underline
P}^N_{t_N-s}(\1)(x)(1-{\bar
\uxi}^N_s(x))\sum_{y\in\cN_N}{\bar\uxi}^N_s(x+y)\Bigr\}\\
&\phantom{=E\Bigl[\int_0^{t_N}\Bigl\{}\
-\sum_{j=2}^3d^{N,j}(x,\uxi^N_s,\1)\Bigr)\,ds\Bigr]\\
\nonumber&\le c_0\int_0^{t_N} |\bar\cN|(\log N)^3
E[{\bar{X}}^N_s(\1)+X^N_s(\1)]\,ds\\ \nonumber&\le
c_0|\bar\cN|(\log N)^3\int_0^{t_N}[e^{c_0(\log
N)^3s}X_0^N(\1)+CX^N_0(\1)]\,ds\\
\label{Diff1}&\le C(\log N)^{-16}X_0^N(\1),
\end{align}
where Lemma~4.1 of \cite{CP05}), and
Proposition~\ref{p:mobnds}\,(a) are used in the next to last
line. The same reasoning (in fact it is simpler as there are
no drift terms) gives
\begin{equation}\label{Diff2} E[{\bar{\underline
X}}_{t_N}(\1)-{\underline X}^{N,\VM}_{t_N}(\1)]\le C(\log
N)^{-16}X_0^N(\1).
\end{equation}

Now argue exactly as in the proof of Lemma~\ref{l:dvmcomp},
using \eqref{Diff1} and \eqref{Diff2} in place of
Lemma~\ref{l:masscmps}, to obtain the following killed
version of Lemma~\ref{l:dvmcomp} (recall now $p=19$ from the
definition of $t_N$):
\begin{align}\nonumber\label{lemma6.2old} &\text{There is a
$C_{\ref{lemma6.2old}}$ so that for $j=2,3$, all $T'>0$,
$s\in[0,T']$ and all $\Phi\in C_b([0,T']\times\R^2)$,}\\
&\qquad
E_{\uxi^N_0}[|d^{N,j}(s,\uxi^N_{t_N},\Phi)-d^{N,j}(s,\uxi^{N,\VM}_{t_N},\Phi)|]\le
C_{\ref{lemma6.2old}}\|\Phi\|_\infty(\log
N)^{-13}\underline{X}^N_0(\1).
\end{align} The proof of Lemma~\ref{l:d3H_} is now completed
using the Markov property just as in the derivation of
Lemma~\ref{l:d3H}. \qed

\medskip
\noindent{\it Proof of Corollary~\ref{cor:surv}}. By
Theorem~\ref{t:ThetaOP} and Theorem~\ref{t:SBMgenintro} we
have the percolation condition \eqref{Hyp2}. The required
result now follows by a comparison to supercritical oriented
percolation and Theorem~\ref{t:SBMgenintro} itself, as in
the proof of survival in Proposition~5.3 of \cite{CP07}.  In
fact, the derivation is now a bit easier as the uniform
bounds proved there are not required. \qed
	
\section{Appendix: The $|\cN|=8$ case of
  Lemma~\ref{l:qcanc}}

Recall that for $|\cN|=8$, 
\begin{align}\label{A3:M}
M(k,j) &= \sum_{\text{odd }i \le j\wedge k} \binom ji
\binom{|8-j}{k-i} , \qquad 1\le k,j \le 8, 
\\\label{A3:a}
a_\ell &= \left(\frac{\ell}{8}\right)^q,
\qquad 1\le \ell\le 8.
\end{align}
Our goal is to verify that  if $\balpha=\ba\bM^{-1}$ then
there exists a $q_0<1$ such that 
\begin{equation}\label{app3goal}
\alpha_\ell(q) >0  \ \forall\ q_0<q<1,\ 1\le \ell\le 8.
\end{equation}
The conclusion of Lemma~\ref{l:qcanc} then follows as
described in Section~\ref{sec:canc}.

It is straightforward to check that $\bM$ given by
\eqref{A3:M} is
\[
\bM =
\begin{bmatrix}
1 & 2 & 3 & 4 & 5 & 6 & 7 & 8 
\\
7 & 12 & 15 & 16 & 15 & 12 & 7 & 0 
\\
21 & 30 & 31 & 28 & 25 & 26 & 35 & 56 
\\
35 & 40 & 35 & 32 & 35 & 40 & 35 & 0 
\\
35 & 30 & 25 & 28 & 31 & 26 & 21 & 56 
\\
21 & 12 & 13 & 16 & 13 & 12 & 21 & 0 
\\
7 & 2 & 5 & 4 & 3 & 6 & 1 & 8 
\\
1 & 0 & 1 & 0 & 1 & 0 & 1 & 0 .
\end{bmatrix}
\]
Using \texttt{maple}, we obtain
\[
\bM^{-1} =
\begin{bmatrix}
-\frac{3}{64} & -\frac{1}{32} & -\frac{1}{64} & 0 & \frac{1}{64} & \frac{1}{32} & \frac{3}{64} & \frac{1}{16} 
\\ \noalign{\smallskip}
 -\frac{7}{64} & -\frac{1}{32} & \frac{1}{64} & \frac{1}{32} & \frac{1}{64} & -\frac{1}{32} & -\frac{7}{64} & -\frac{7}{32} 
\\\noalign{\smallskip}
 -\frac{7}{64} & \frac{1}{32} & \frac{3}{64} & 0 & -\frac{3}{64} & -\frac{1}{32} & \frac{7}{64} & \frac{7}{16} 
\\\noalign{\smallskip}
 0 & \frac{5}{64} & 0 & -\frac{3}{64} & 0 & \frac{5}{64} & 0 & -\frac{35}{64} 
\\\noalign{\smallskip}
 \frac{7}{64} & \frac{1}{32} & -\frac{3}{64} & 0 & \frac{3}{64} & -\frac{1}{32} & -\frac{7}{64} & \frac{7}{16} 
\\\noalign{\smallskip}
 \frac{7}{64} & -\frac{1}{32} & -\frac{1}{64} & \frac{1}{32} & -\frac{1}{64} & -\frac{1}{32} & \frac{7}{64} & -\frac{7}{32} 
\\\noalign{\smallskip}
 \frac{3}{64} & -\frac{1}{32} & \frac{1}{64} & 0 & -\frac{1}{64} & \frac{1}{32} & -\frac{3}{64} & \frac{1}{16} 
\\\noalign{\smallskip}
 \frac{1}{128} & -\frac{1}{128} & \frac{1}{128} & -\frac{1}{128} & \frac{1}{128} & -\frac{1}{128} & \frac{1}{128} & -\frac{1}{128} 
 \end{bmatrix},
\]
and it is straightforward but tedious to check that this is
indeed correct by verifying that the product of the above matrices is the
identity matrix.

Given $\bM^{-1}$ and $\ba$ as above, 
$\balpha=\ba\bM^{-1}$ is given by 
\begin{align*}
\alpha_1(q)&=-\frac{3
  \left(\frac{1}{8}\right)^{q}}{64}-\frac{7
  \left(\frac{1}{4}\right)^{q}}{64}-\frac{7
  \left(\frac{3}{8}\right)^{q}}{64}+\frac{7
  \left(\frac{5}{8}\right)^{q}}{64}+\frac{7
  \left(\frac{3}{4}\right)^{q}}{64}+\frac{3
  \left(\frac{7}{8}\right)^{q}}{64}+\frac{1}{128}\\
\alpha_2(q)&=-\frac{\left(\frac{1}{8}\right)^{q}}{32}-\frac{\left(\frac{1}{4}\right)^{q}}{32}+\frac{\left(\frac{3}{8}\right)^{q}}{32}+\frac{5 
  \left(\frac{1}{2}\right)^{q}}{64}+\frac{\left(\frac{5}{8}\right)^{q}}{32}-\frac{\left(\frac{3}{4}\right)^{q}}{32}-\frac{\left(\frac{7}{8}\right)^{q}}{32}-\frac{1}{128} 
\\
\alpha_3(q)&=-\frac{\left(\frac{1}{8}\right)^{q}}{64}+\frac{\left(\frac{1}{4}\right)^{q}}{64}+\frac{3
  \left(\frac{3}{8}\right)^{q}}{64}-\frac{3
  \left(\frac{5}{8}\right)^{q}}{64}-\frac{\left(\frac{3}{4}\right)^{q}}{64}+\frac{\left(\frac{7}{8}\right)^{q}}{64}+\frac{1}{128}\\ 
\alpha_4(q)&=\frac{\left(\frac{1}{4}\right)^{q}}{32}-\frac{3 \left(\frac{1}{2}\right)^{q}}{64}+\frac{\left(\frac{3}{4}\right)^{q}}{32}-\frac{1}{128}\\
\alpha_5(q)&=\frac{\left(\frac{1}{8}\right)^{q}}{64}+\frac{\left(\frac{1}{4}\right)^{q}}{64}-\frac{3 \left(\frac{3}{8}\right)^{q}}{64}+\frac{3 \left(\frac{5}{8}\right)^{q}}{64}-\frac{\left(\frac{3}{4}\right)^{q}}{64}-\frac{\left(\frac{7}{8}\right)^{q}}{64}+\frac{1}{128}\\
\alpha_6(q)&=\frac{\left(\frac{1}{8}\right)^{q}}{32}-\frac{\left(\frac{1}{4}\right)^{q}}{32}-\frac{\left(\frac{3}{8}\right)^{q}}{32}+\frac{5 \left(\frac{1}{2}\right)^{q}}{64}-\frac{\left(\frac{5}{8}\right)^{q}}{32}-\frac{\left(\frac{3}{4}\right)^{q}}{32}+\frac{\left(\frac{7}{8}\right)^{q}}{32}-\frac{1}{128}\\
\alpha_7(q) &=\frac{3
  \left(\frac{1}{8}\right)^{q}}{64}-\frac{7
  \left(\frac{1}{4}\right)^{q}}{64}+\frac{7
  \left(\frac{3}{8}\right)^{q}}{64}-\frac{7
  \left(\frac{5}{8}\right)^{q}}{64}+\frac{7
  \left(\frac{3}{4}\right)^{q}}{64}-\frac{3
  \left(\frac{7}{8}\right)^{q}}{64}+\frac{1}{128} \\
\alpha_8(q)&=\frac{\left(\frac{1}{8}\right)^{q}}{16}-\frac{7 \left(\frac{1}{4}\right)^{q}}{32}+\frac{7 \left(\frac{3}{8}\right)^{q}}{16}-\frac{35 \left(\frac{1}{2}\right)^{q}}{64}+\frac{7 \left(\frac{5}{8}\right)^{q}}{16}-\frac{7 \left(\frac{3}{4}\right)^{q}}{32}+\frac{\left(\frac{7}{8}\right)^{q}}{16}-\frac{1}{128}
\end{align*}

By pairing off terms it is easy to see that $\alpha_1(q)\ge
1/128>0$ for all $0\le q\le 1$. As described in
Section~\ref{sec:canc}, to prove \eqref{app3goal} it remains
only to verify \eqref{cangoal}, so we consider
$\alpha'_\ell(q), 2\le \ell\le 8$. By differentiating,
plugging in $q=1$ and simplifying, we obtain
\begin{align*}
\alpha_2'(1) &= \frac{3 }{128}\log 2-\frac{3 }{256}\log 3+\frac{5 }{256}\log 5-\frac{7 }{256}\log  7\\
\alpha_3'(1) &= \frac1{64}\log 2+\frac{3}{512}\log  3-\frac{15}{512}\log 5+\frac{7}{512} \log 7\\
\alpha_4'(1) &=-\frac{5}{128} \log 2+\frac{3}{128}\log 3\\
\alpha_5'(1) &=\frac1{64}\log 2-\frac{15 }{512}\log 3+\frac{15}{512} \log 5-\frac{7}{512}\log 7\\
\alpha_6'(1) &=\frac{3}{128} \log 2-\frac{9}{256} \log 3-\frac{5}{256} \log 5+\frac{7}{256} \log 7\\
\alpha_7'(1) &=\frac{5}{64} \log 2+\frac{63}{512} \log 3-\frac{35}{512} \log 5-\frac{21}{512} \log 7\\
\alpha_8'(1) &= -\frac{101}{128} \log 2+\frac{35}{128} \log 5+\frac{7}{128} \log 7
\end{align*}
We need only verify that each of the
above are strictly negative. 
Resorting again to \texttt{maple}, the above derivatives can
be written in the form 
\begin{align*}
\alpha_2'(1) &= -\frac1{256}\log\left(\frac{22235661}{200000}\right)\\
\alpha_3'(1) &= -\frac1{512}\log\left(\frac{30517578125}{5692329216}\right)\\
\alpha_4'(1) &= -\frac1{128}\log\left(\frac{32}{27}\right)\\
\alpha_5'(1) &= -\frac1{512}\log\left(\frac{11816941917501}{7812500000000}\right)\\
\alpha_6'(1) &= -\frac1{256}\log\left(\frac{61509375}{52706752}\right)\\
\alpha_7'(1) &= -\frac1{512}\log\left(
\frac{1625582413058972472208552062511444091796875}{1258458428839311554156984626190103821156352}\right)\\
\alpha_8'(1) &= -\frac1{128}\log\left(\frac{2535301200456458802993406410752}{2396825584582984447479248046875}\right).
\end{align*}
This completes the proof of \eqref{cangoal} for the case
$|\cN|=8$, and so we are done.

\newpage
\bibliographystyle{amsplain}

\end{document}